\documentclass[11pt]{amsart}
\usepackage{geometry}                % See geometry.pdf to learn the layout options. There are lots.
\usepackage[usenames,dvipsnames,svgnames,table]{xcolor}    % Activate to begin paragraphs with an empty line rather than an indent
\usepackage{graphicx}
\usepackage{tikz}
\usetikzlibrary{decorations.pathreplacing}
\usepackage{lineno, hyperref}
\usepackage{amssymb,amsfonts, amsmath, amsthm}
\usepackage{epstopdf,enumitem}
\usepackage{ulem}
\usepackage{cancel}
\usepackage{thmtools}
\usepackage{bbm,color}
\definecolor{grey}{rgb}{0.6,0.6,0.6}
\definecolor{orange}{rgb}{1.0,0.5,0.5}
\definecolor{brown}{rgb}{0.5,0.25,0.0}
\definecolor{pink}{rgb}{1.0,0.5,0.5}
\definecolor{black}{rgb}{0,0,0}
\definecolor{LightCyan}{rgb}{0.26,0.41,0.88}
\definecolor{Gris}{rgb}{0.83,0.83,0.83}
\definecolor{lightblue}{rgb}{0.55, 0.65, 0.83}
\definecolor{blue}{rgb}{0.25, 0.25, 0.85}
\definecolor{lightgrey}{rgb}{0.74, 0.83, 0.9}

\newtheorem*{theorem*}{Theorem}

\declaretheorem[style=definition,qed=$\blacksquare$,numberwithin=section]{definition}
\declaretheorem[style=definition,qed=,sibling=definition]{remark}
\declaretheorem[style=definition,qed=$\blacksquare$, sibling = definition]{lemma-definition}

\declaretheorem[style=theorem, sibling = definition]{theorem}
\declaretheorem[style=theorem, sibling = definition]{lemma}
\declaretheorem[style=theorem, sibling = definition]{proposition}
\declaretheorem[style=theorem, sibling = definition]{claim}

\declaretheorem[style=theorem, sibling = definition]{question}

\newcommand{\N}{\mathbb{N}}
\newcommand{\NN}{\mathbb{N}}

\newcommand{\Lm}{\mathcal{L}}
\newcommand{\cL}{\Lm}

\newcommand{\cO}{\mathcal{O}}

\newcommand{\ord}{\mathrm{ord}}
\newcommand{\order}{\ord \,}

\newcommand{\Lring}{\Lm_{\text{ring}}}

\newcommand{\Lpres}{\Lm_{\text{Pres}}}

\newcommand{\ac}{\text{ac}\,}
\newcommand{\acm}{\mathrm{ac}_{m}\,}

\newcommand{\Int}{\text{Int}}

\newcommand{\Cen}{\text{Centers}}

\title{Clustered cell decomposition in $P$-minimal structures}

\author{Saskia Chambille}
\author{Pablo Cubides Kovacsics }
\author{Eva Leenknegt}%\corref{mycorrespondingauthor}}
\address{Saskia Chambille, KU Leuven, Department of Mathematics, Celestijnenlaan 200B, 3001 Heverlee, Belgium }
\email{saskia.chambille@kuleuven.be}
\address{Pablo Cubides Kovacsics, Universit\'e de Caen, Laboratoire de math\'ematiques Nicolas Oresme, CNRS UMR 6139,
14032 Caen cedex, France }
\email{pablo.cubides@unicaen.fr}
\address{Eva Leenknegt, KU Leuven, Department of Mathematics, Celestijnenlaan 200B, 3001 Heverlee, Belgium }
\email{eva.leenknegt@kuleuven.be}

\begin{document}

%\begin{frontmatter}

%\author[KUL,LILLE]{Saskia Chambille}
%\ead{saskia.chambille@kuleuven.be}
%
%\author[PCK]{Pablo Cubides Kovacsics }
%\ead{pablo.cubides@unicaen.fr}
%
%\author[KUL]{Eva Leenknegt\corref{mycorrespondingauthor}}
%\ead{eva.leenknegt@kuleuven.be}
%
%\address[KUL]{KU Leuven, Department of Mathematics, Celestijnenlaan 200B, 3001 Heverlee, Belgium }
%\address[PCK] {Universit\'e de Caen, Laboratoire de math\'ematiques Nicolas Oresme, CNRS UMR 6139,
%14032 Caen cedex, France }
%\address[LILLE]{UniversitŽ de Lille 1, Laboratoire Paul PainlevŽ, CNRS UMR 8524, CitŽ Scientifique, 59655 Villeneuve d'Ascq Cedex, France}

%\date{}                                           % Activate to display a given date or no date

\maketitle

\begin{abstract}
We prove that in a $P$-minimal structure, every definable set can be partitioned as a finite union of classical cells and regular clustered cells. 
This is a generalization of previously known cell decomposition results by Denef and Mourgues, which were dependent on the existence of definable Skolem functions. Clustered cells have the same geometric structure as classical, Denef-type cells, but do not have a definable function as center. Instead, the center is given by a definable set whose fibers are finite unions of balls. 
\end{abstract}

%\end{frontmatter}

%\linenumbers
%\section{}
%\subsection{}
\normalem
%\input{file.tex}
% !TEX root = centers_main.tex
\section{Introduction}

The aim of this paper is to present an unconditional description of definable sets in $P$-minimal structures,  in the spirit of Denef's work on cell decomposition for semi-algebraic sets \cite{denef-86}.
More precisely, we intend to show the following (for a formal and more detailed statement we refer to Theorem \ref{thm:celldecomposition}): 

\begin{theorem*} Every definable set $X\subset S\times K$ can be partitioned as a finite union of classical cells and regular clustered cells.\end{theorem*} 

Roughly speaking, classical cells are Denef-type cells, with a definable function as center. The geometric structure of a clustered cell is the same as that of a classical cell (or possibly a finite disjoint union of classical cells that only differ by their center). The main difference is that the centers of clustered cells are not given by definable functions, but instead are picked from a definable set whose fibers are finite unions of balls. 
\\\\
Note that in a structure with Skolem functions, one does not need to consider clustered cells, since one could simply choose centers from each ball and use these to split the cluster into classical cells. 
In fact, for $P$-minimal structures with definable Skolem functions, a cell decomposition theorem had already been obtained before, first by Mourgues \cite{mou-09}, whose results were later extended and refined by Darni\`ere-Halupczok \cite{darniereETAL:2015}. 

Up until quite recently, it was not so clear whether the dependence on Skolem functions was an actual restriction to the scope of these theorems, given that there were as yet no known examples of structures that didn't admit such functions. However, this changed when Nguyen and one of the authors provided such an example \cite{cubi-nguyen}, thereby showing that the existence of Skolem functions really is a restricting condition in \cite{mou-09} and \cite{darniereETAL:2015}. 
\\\\
Hence the main question we wanted to answer in this paper: is it possible to obtain a cell decomposition result valid in all $P$-minimal structures, without imposing such conditions? The question is asked here specifically in the context of Denef-Mourgues type cell decompositions,
%with , our intention was to  stay as close as possible to the spirit of Denef's original theorem for $p$-adic semi-algebraic structures, 
as  there already existed other unconditional, but more topologically-oriented decomposition results (see \cite{cub-dar-lee:15}).

In a previous paper, which was motivated by questions about integration in $P$-minimal structures, two of the authors provided a first proto-version of such a description \cite{cubi-leen-2015}. However, this version, while strong enough for its intended use, still failed to provide sufficient  intuition about the geometric structure of definable sets. The current paper, which uses the work of \cite{cubi-leen-2015} as a foundation, aims to remedy this. Moreover, we intend to use our results to answer questions regarding $p$-adic integration and cell preparation in $P$-minimal structures (in future work). 

%Our objective in this paper is to show the following (for a formal and more detailed statement we refer to Theorem \ref{thm:celldecomposition}): 
%
%\begin{theorem*} Every definable set $X\subset S\times K$ can be partitioned as a finite union of classical cells and inseparably clustered cells.\end{theorem*} 
%
%Roughly speaking, classical cells are Denef-type cells, with a definable function as center. The geometric structure of a clustered cell is the same as that of a classical cell (or possibly a finite disjoint union of classical cells that only differ by their center). The main difference is that the centers of clustered cells are not given by definable functions, but instead are picked from a definable set whose fibers are finite sets of balls. 
%
%Note that in a structure with Skolem functions, \textcolor{purple}{one does not need to consider clustered cells, since one could simply choose centers from each ball and use these to split the cluster into classical cells.} 

\

\label{sec:def}
Before introducing and explaining different types of cells and related notions, we recall some preliminaries. The next section lists some of the notations we will be using, and gives a definition of $P$-minimality. %after which we will quickly recall the tree picture of valued fields.

\subsection{Notations and preliminary definitions} \label{subsec:notations}
Let $K$ be a $p$-adically closed field (that is, elementarily equivalent to a $p$-adic field). We use the notation $\Gamma_K$ for the value group, $\ord:K\to\Gamma_K\cup\{\infty\}$ for the valuation, $q_K$ for the number of elements of the residue field $k_K$, $\cO_K$ for the valuation ring of $K$, $\mathcal{M}_K$ for the maximal ideal and $\pi_K$ for a uniformizing element. We write $B_{\gamma}(a)$  for the closed ball around $a$ with radius $\gamma$:
\[B_{\gamma}(a):= \{ x \in K \mid \ord(x-a) \geqslant \gamma\}. \]
%Since $\Gamma_K$ is discrete, every ball is a closed ball.
For $m >0$, write $\acm: K^{\times} \to (\cO_K/\pi_K^m\cO_K)^{\times}$ for the unique group homomorphism such that $\acm(\pi_K) =1$ and $\acm(u) \equiv u \mod \pi_K^m$ for any unit $u \in \cO_K$. That such an \emph{angular component map} exists (and is indeed unique) was shown in Lemma 1.3 of \cite{clu-lee-2011}. We extend this to $K$ by putting $\acm(0) = 0$. 
%$m$-th angular component map, which can be defined as
%\[\acm(x) := \left\{\begin{array}{ll} \tilde{x}\hspace{-5pt} \mod \pi_K^m & \text{if } 0 \neq x = \pi_K^{\ord x}\tilde{x},\\
%0& \text{if } x=0.
%\end{array}\right.\]
%In every expansion of a $p$-adically closed field $K$, such angular component maps exist and can be defined in a unique way, as was shown in Lemma 1.3 of \cite{clu-lee-2011}.} 
For positive integers $n,m$, let $Q_{n,m}$ be the set
\[ Q_{n,m} := \{ x \in K^{\times} \mid \ord \,x \equiv 0 \hspace{-7pt}\mod n \wedge \acm(x) = 1\}.\]
Note that for $x \in \lambda Q_{n,m}$, the value of $\lambda$ encodes both $\acm(x)$ and $(\ord \,x \hspace{-5pt}\mod n$). 
\\\\
Following \cite{cubi-leen-2015}, we will work with a two-sorted version of $P$-minimality, where we consider both the field sort and the value group sort $\Gamma_K \cup \{\infty\}$ to be of equal importance. 
Let $(K, \Gamma_K;\Lm_2)$ be a two-sorted structure, with language $\Lm_2 = (\Lm, \cL_{Pres}, \ord)$. Here $\cL$, the language for the $K$-sort, is assumed to be an expansion of the ring language $\Lring$. For the value group sort $\Gamma_K\cup\{+\infty\}$, we use the language of Presburger arithmetic $\cL_{Pres} = (+,-,<,\{\equiv_n\}_n)$. The sorts are connected through the valuation map $\ord: K\to \Gamma_K\cup\{+\infty\}$. The definition of $P$-minimality naturally extends to this context. %\textcolor{blue}{Add reference to previous cell decomposition paper} 

%When the language $\Lm_2$ is clear from the context, we will simplify the notation $(K, \Gamma_K;\Lm_2)$ to just $(K, \Gamma_K)$. 
\begin{definition}
A two-sorted structure $(K, \Gamma_K; \Lm_2)$ with $\Lm_2 = (\Lm, \Lpres, \ord)$ and $\Lring \subseteq \Lm$ is said to be \emph{$P$-minimal} if the underlying structure $(K, \Lm)$ is $P$-minimal, that is, for every $(K',\Lm)$ elementarily equivalent to $(K,\Lm)$, the $\Lm$-definable subsets of $K'$ are $\Lring$-definable.
\end{definition}
 From now on we will work in a $P$-minimal structure $(K, \Gamma_K;\Lm_2)$. We refer to \cite{cubi-leen-2015} and the last section of the current paper for further discussion and justification of this choice of setting. 

\

By \emph{definable} we always mean definable with parameters. The set $S$ denotes a definable set whose variables may include both $K$-variables and $\Gamma_K$-variables.
%For notational purposes, we will fix a definable set $S \subseteq K^{m_0} \times \Gamma_K^{m_0}$ which we call a \emph{parameter set}. 
Given a set $X\subseteq S\times Y$ and $s\in S$, we write \[X_s:=\{t\in Y\mid (s,t)\in X\}\] to denote the fiber over $s$. The topological closure of $X$ will be denoted as $\text{Cl}(X)$.

\subsection{Cells}

%\textcolor{blue}{As promised in the introduction, we start this section by giving a more extensive definition of a $K$-cell condition.}

In our view, a cell has two major \emph{ingredients}: its center (which we will discuss further on), and the formula $C$ defining the cell.

\begin{definition}[$K$-Cell condition] \label{def:cell-condition}
\item A \emph{$K$-cell condition} over $S$ is a formula of the form
\[C(s, c, t) :=  s\in S \wedge \alpha(s) \ \square_1 \ \ord(t-c) \ \square_2 \ \beta(s) \wedge t-c \in \lambda Q_{n,m}\,, \]
where $t$ and $c$ are variables over $K$, $\alpha, \beta$ are definable functions $S \to \Gamma_K$, squares $\square_1,\square_2$ may denote either $<$ or    $\emptyset$ (i.e. `no condition'), $\lambda \in K$ and $n,m \in \N\backslash\{0\}$. The variable $c$ is called the center of the $K$-cell condition. 
\item A $K$-cell condition $C$ is called a \emph{0-cell condition}, resp.\ a \emph{1-cell condition} if $\lambda =0$, resp.\ $\lambda \neq 0$.
\end{definition}
\noindent Note that in the above definition, $t$ is assumed to be a variable in the field sort $K$ (while the parameter set $S$ may contain both $K$- and $\Gamma_K$-variables). In \cite{cubi-leen-2015}, $\Gamma$-cell conditions were also introduced, for analogous formulas with $t$ ranging over the value group sort $\Gamma_K$. 
We will say a bit more about this in Section \ref{subsec:two-sorted}. However, in the current paper we will concentrate almost exclusively on $K$-cell conditions, and hence we will often omit the $K$ and simply speak of \emph{cell conditions}.
 %Similarly, the parameter set associated to the cell condition will always be assumed to be $S$, and we will only mention it explicitly when there is a risk of ambiguity.
%\\
%\textcolor{blue}{expand a bit and explain difference between K-cells and Gamma-cells and refer to previous paper}
%\begin{remark}
%We will always assume that the upper bound $\beta(s)$ of a cell condition is as small as possible, that is, if $\square_2$ equals $<$, then we will assume that $\beta(s) -1 \equiv \ord \lambda \mod n$. It is easy to see that this can always be imposed on $\beta(s)$ in a definable way. 
%\end{remark}
\begin{remark}\label{rem:notation} We will use the following notational convention. Capital $C$ will always denote a cell condition over some set of parameters $S$ for which the symbols $\alpha,\beta,\lambda,\square_1, \square_2, n,m $ are fixed as in the previous definition. In particular, the letters $\alpha$ and $\beta$ will only be used to denote the functions picking the lower and upper bounds in a cell condition $C$. If multiple cell conditions are discussed at the same time, say $C_1,\ldots, C_r$, the same index will be applied to the symbols in the associated formula. Thus, $\alpha_i$ and $\beta_i$ denote the functions picking the lower and upper bounds of a cell condition $C_i$, and the use of $\square_{i1}, \square_{i2}, \lambda_i, n_i, m_i$ follows similar conventions. 
%\textcolor{red}{CHECK THAT $n,m$ are not used anywhere else in paper for other things anymore}
\end{remark}
Let $C$ be a cell condition over $S$ and $\sigma: S \to K$  a function (not necessarily definable). Using this function as the center for $C$, we get the induced set 
 \[C^\sigma:= \{(s,t) \in S \times K \mid C(s, \sigma(s),t)\}.\]
When there is no dependence on parameters (i.e., if $C$ is a cell condition over $S=\Gamma^{0} \times K^{0}$), a function $\sigma:S\to K$ will be identified with a point $\sigma\in K$.
Sets of the form $C^\sigma$ will be informally called \emph{cells over $S$} (or simply \emph{cells}, when the parameter set $S$ is clear from the context). 
The reader will probably be most familiar with \emph{classical cells}, that is, cells $C^{\sigma}$ for which the function $\sigma$ is definable. For instance, one may think of semi-algebraic or sub-analytic cells, where the center $\sigma$ is a semi-algebraic, resp. a subanalytic function (see \cite{denef-86, clu-2003}). 
 \\\\
%Let $C$ be a cell condition over $S$ and $\sigma:S\to K$ a function. 
We will denote the fiber of a cell $C^\sigma$ over $s\in S$ by
\[ 
C^{\sigma(s)} := \{t\in K\mid C(s,\sigma(s),t)\}.
\]
When $C$ is a $0$-, resp.\ a $1$-cell condition, we will call $C^{\sigma}$ a \emph{0-cell}, resp.\ a \emph{1-cell}.   
\begin{definition} Let $C$ be a $K$-cell condition over $S$ and $\sigma: S \to K$  a function. The \emph{leaf of $C^{\sigma(s)}$ at height $\gamma$} corresponds to the ball 
\[
C^{\sigma(s), \gamma}:=\{t\in C^{\sigma(s)} \mid  \ord(t-\sigma(s))=\gamma\}. 
\qedhere\]
%\[
%C^{\sigma(s), \gamma}:=\{t\in K: \ord(t-\sigma(s))=\gamma \ \wedge \  t-\sigma(s) \in \lambda Q_{n,m}\}. 
%\]
\end{definition}
The fibers $C^{\sigma(s)}$ of a cell $C^{\sigma}$ can be visualised in the following way. Here we adopt the perspective used also in \cite{mac-has-94, has-mac-97},  representing elements and basic subsets of valued fields by trees (see more in Section 6). 
 \\\\
\begin{minipage}{0.45 \textwidth} 
%Notice that $C^{\sigma}$ is a disjoint union of balls of different sizes, each correspo
%, which corresponds to the set of maximal balls contained in the cell fiber $C^{\sigma(s)}$.}\\ %For this, we need to make a distinction between $0$-cells and $1$-cells (i.e., cells $C^{\sigma}$ where $C$ is respectively a $0$- or a $1$-cell condition). Also the values of $\square_1$ and (mainly) $\square_2$ will be important. \\
%The picture of such tree depends essentially on whether $C$ is a 0-cell or a 1-cell condition, and on the particular values of $\square_1$ and $\square_2$ in the latter case. 
When $C$ is a 0-cell condition, fibers correspond to points: $C^{\sigma(s)}=\{\sigma(s)\}$.
\\\\
When $C$ is a 1-cell condition, the fiber $C^{\sigma(s)}$ is the  disjoint union of its leaves $C^{\sigma(s), \gamma}$. One can check that a leaf at height $\gamma$ corresponds to a ball of radius $\gamma + m$.
\\\\
 Note  that $\sigma(s)\notin C^{\sigma(s)}$, and that $\sigma(s)\in \text{Cl}(C^{\sigma(s)})$ if and only if $\square_2=\emptyset$.

%, wich we will call the leaves of $C^{\sigma(s)}$. 
\end{minipage}
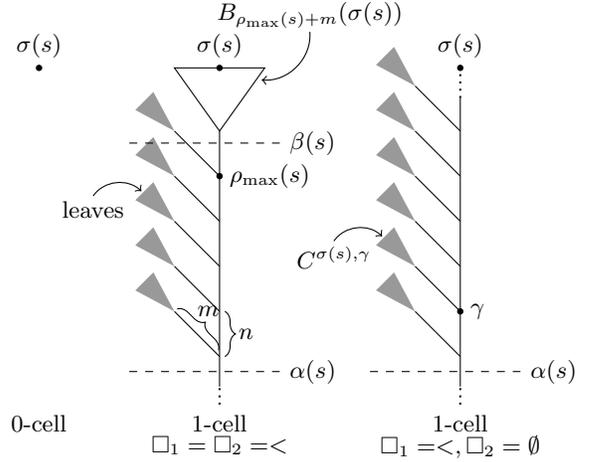
\begin{minipage}{0.05 \textwidth}\quad \end{minipage}
\begin{minipage}{0.5 \textwidth}
%\begin{center}
%\includegraphics[scale=0.35]{leaf1.jpg}
%\end{center}

%. We will refer to the maximal balls contained in $C^{\sigma(s)}$ as the \emph{leaves of $C^{\sigma(s)}$}.\\\vspace{-8pt}\\ %To make a picture, we need to further distinguish between the cases where $\square_2$ equals $<$ or $\emptyset$.  
%As the pictures suggests, when $C$ is a 1-cell condition, we will refer to the maximal balls contained in $C^{\sigma(s)}$ as the \emph{leaves of $C^{\sigma(s)}$}. For such fibers, notice that $\sigma(s)\notin C^{\sigma(s)}$. Moreover, $\sigma(s)\in Cl(C^{\sigma(s)})$ if and only if $\square_2=\emptyset$. 
%{\color{red} When $\square_1=<$, note that our assumption that $\beta(s)$ is as small as possible implies $C^{\sigma(s)}$ always has a leaf at height $\beta(s)-1$.}
%\\\\
%I REWROTE THIS BIT. PLEASE CHECK THAT THIS IS CORRECT:

\begin{tikzpicture}[scale = 0.8]
\node[{above}] at (0,5.3){\footnotesize $\sigma(s)$};
\draw [fill] (0,5.3) circle [radius = 1.25pt];

\node [{below}] at (0,-0.3){\footnotesize 0-cell};

\draw (2.25,5.3) -- (3,4.25) -- (3.75,5.3) -- cycle;
\draw (3,0) -- (3,4.25);
\draw [dotted, line width=0.75pt] (3,-0.3) -- (3,0);
\draw (3,3.5) -- (2.25,4.25);
\draw (3,2.75) -- (2.25, 3.5);
\draw (3,2) -- (2.25,2.75);
\draw (3,1.25) -- (2.25,2);
\draw (3,0.5) -- (2.25,1.25);
\path [fill = grey] (2.25,4.25) -- (1.6,4.6) -- (1.9,4.9) -- cycle;
\path [fill = grey] (2.25,3.5) -- (1.6,3.85) -- (1.9,4.15) -- cycle;
\path [fill = grey] (2.25,2.75) -- (1.6,3.1) -- (1.9,3.4) -- cycle;
\path [fill = grey] (2.25,2) -- (1.6,2.35) -- (1.9,2.65) -- cycle;
\path [fill = grey] (2.25,1.25) -- (1.6,1.6) -- (1.9,1.9) -- cycle;
\draw [dashed] (1.5,4.05) -- (4,4.05);
\draw [dashed] (1.5,0.25) -- (4,0.25);
\node[{right}] at (4,4.05){\footnotesize$\beta(s)$};
\node[{right}] at (4,0.25){\footnotesize$\alpha(s)$};
\node[{above}] at (3,5.3){\footnotesize$\sigma(s)$};
\node[{right}] at (3,3.5){\footnotesize$\rho_{\max}(s)$};
\draw [fill] (3,5.3) circle [radius = 1.25pt];
\draw [fill] (3,3.5) circle [radius = 1.25pt];
\draw [decorate,decoration={brace,amplitude=4pt,mirror,raise=2pt},yshift=0pt] (3,0.5) -- (3,1.25) node [black,midway,xshift=10pt] {\footnotesize$n$};
\draw [decorate,decoration={brace,amplitude=4pt,mirror,raise=1.5pt},yshift=0pt, rotate=45] (2.47,-1.68) -- (2.47,-0.71) node [black,midway,xshift=5pt,yshift=8pt] {\footnotesize$m$};
\node [{above}] at (4.5,5.8){\footnotesize$B_{\rho_\text{max}(s)+m}(\sigma(s))$};
\path[->] (4.5,5.9)  edge [bend left=50] (3.7,5.1);
\node [{below}] at (0.9,3.25){\footnotesize leaves};
\path[->] (0.9,3.17)  edge [bend left=50] (1.7,3.25);
%\draw [decorate,decoration={brace,amplitude=4pt,mirror,raise=3pt},yshift=0pt] (3,3.5) -- (3,4.25) node [black,midway,xshift=12pt,yshift=-2pt] {$m$};
%\draw [dashed] (3.45,5) -- (3,4.25);
%\node[{above}] at (3.6,5){$\sigma'$};
%\draw [dashed] (1.5,3.5) -- (4,3.5);
%\node[{right}] at (4,3.4){$\rho_\text{max}$};

\node [{below}] at (3,-0.3){\footnotesize 1-cell};
\node [{below}] at (3,-0.65){\footnotesize$\square_1 = \square_2 = <$};

\draw (7,0) -- (7,4.8);
\draw [dotted, line width=0.75pt] (7,-0.3) -- (7,0);
\draw [dotted, line width=0.75pt] (7,4.8) -- (7,5.3);
\draw (7,4.25) -- (6.25,5);
\draw (7,3.5) -- (6.25,4.25);
\draw (7,2.75) -- (6.25, 3.5);
\draw (7,2) -- (6.25,2.75);
\draw (7,1.25) -- (6.25,2);
\draw (7,0.5) -- (6.25,1.25);
\path [fill = grey] (6.25,5) -- (5.6,5.35) -- (5.9,5.65) -- cycle;
\path [fill = grey] (6.25,4.25) -- (5.6,4.6) -- (5.9,4.9) -- cycle;
\path [fill = grey] (6.25,3.5) -- (5.6,3.85) -- (5.9,4.15) -- cycle;
\path [fill = grey] (6.25,2.75) -- (5.6,3.1) -- (5.9,3.4) -- cycle;
\path [fill = grey] (6.25,2) -- (5.6,2.35) -- (5.9,2.65) -- cycle;
\path [fill = grey] (6.25,1.25) -- (5.6,1.6) -- (5.9,1.9) -- cycle;
\draw [dashed] (5.5,0.25) -- (8,0.25);
\node[{right}] at (8,0.25){\footnotesize$\alpha(s)$};
\node[{above}] at (7,5.3){\footnotesize$\sigma(s)$};
\draw [fill] (7,5.3) circle [radius = 1.25pt];
\node [{below}] at (4.9,2.5){\footnotesize$C^{\sigma(s),\gamma}$};
\path[->] (4.9
,2.42)  edge [bend left=50] (5.7,2.5);
\draw [fill] (7,1.25) circle [radius = 1.25pt];
\node[{right}] at (7,1.25){\footnotesize$\gamma$};

\node [{below}] at (7,-0.3){\footnotesize 1-cell};
\node [{below}] at (7,-0.65){\footnotesize$\square_1 = <, \square_2 = \emptyset$};
\end{tikzpicture}
\end{minipage}\\\\
When $\square_2$ denotes $<$, the center of a cell $C^{\sigma}$ is not unique. Indeed, write $\rho_\text{max}(s)$
for the height of the top leaf of $C^{\sigma(s)}$  (so $\beta(s) - n \leqslant \rho_\text{max}(s) \leqslant \beta(s) - 1$). Note that $\rho_{\max}: S \to \Gamma_K$ is a definable function which only depends on the cell condition, and not on the choice of the center. 
It is easy to see that one still gets the exact same fiber $C^{\sigma(s)}$, if  $\sigma(s)$ is replaced by any other element of the ball $B_{\rho_{\max}(s) + m}(\sigma(s))$. 
%
%
%there exists a function $\rho_C: S \to \Gamma_K$, with the property that This function $\rho$ is definable from the parameters in the cell condition $C$ in the following way.  Then one can put
%\[ \rho_C(s) := \gamma_{\max}(s) + m. \] 
Hence, it is reasonable to consider the set
\begin{equation*}%\label{eq:sigma}
\Sigma = \{(s,c) \in S \times K \mid c \in B_{\rho_{\max}(s) + m}(\sigma(s))\}
\end{equation*}
as the set of centers for $C^{\sigma}$.
%ndeed, there exists a function $\rho_C: S \to \Gamma_K$, with the property that we still get the exact same fiber $C^{\sigma(s)}$ whenever we replace $\sigma(s)$ by any other element of the ball $B_{\rho_C(s)}(\sigma(s))$. 
%This function $\rho$ is definable from the parameters in the cell condition $C$ in the following way. Assume that the highest leaves of $C^{\sigma}$ occur at height $\gamma_\text{max}(s)$ (so that $\beta(s) - n \leqslant \gamma_\text{max}(s) \leqslant \beta(s) - 1$). Then one can put
%\[ \rho_C(s) := \gamma_{\max}(s) + m. \] 
%Note that $\rho_C$ depends only on the cell condition and not on the choice of center. This implies that it is reasonable to consider the complete set
%\begin{equation}\label{eq:sigma}
%\Sigma = \{(s,\sigma') \in S \times K \mid \sigma' \in B_{\rho_C(s)}(\sigma(s))\}
%\end{equation}
%as the set of centers for the cell $C^\sigma$, and view the graph of $\sigma$ as a particular section of $\Sigma$.}
In $P$-minimal structures without definable Skolem functions, it might happen that $\Sigma$ itself is a definable set, yet no section of $\Sigma$ is definable. %Let \textcolor{magenta}{$\sigma$ denote some choice of section for $\Sigma$, which is obviously not definable.} 
 Nevertheless, even when $\sigma$ is a non-definable section  of $\Sigma$,  the cell $C^{\sigma}$  will still be definable (as a set), since we have the equality
\[
C^{\sigma}=\{(s,t) \in S\times K \mid (\exists c)[c \in \Sigma_s \wedge C(s,c,t)]\ \}.
\]
It is therefore natural to consider the following notion. 

\begin{definition}\label{def:sigmacell} Let $C$ be a cell condition and $\Sigma \subseteq S \times K$ be a definable set. The set $C^{\Sigma}\subseteq S\times K$ is defined as
\[
C^{\Sigma}:= \{(s,t) \in S\times K \mid (\exists c)[c \in \Sigma_s \wedge C(s,c,t)]\ \ \}.
\]
\item Every (not necessarily definable) section $\sigma: S \to K$ of $\Sigma$ is called a \emph{potential center} of $C^{\Sigma}$. We call the induced sets $C^{\sigma}$ \emph{potential cells}. 
\end{definition}

Let us stress that, given two different sections $\sigma$ and $\sigma'$ of $\Sigma$, the induced cells $C^{\sigma}$ and $C^{\sigma'}$ may be very different (possibly even disjoint) subsets of $C^{\Sigma}$, since we have not yet imposed any conditions on $\Sigma$. If we want sets $C^{\Sigma}$ to be useful building blocks in our cell decomposition, we will have to significantly restrict the type of set that can occur for $\Sigma$. Indeed, every definable set $X\subseteq S\times K$ is already of the form $C^\Sigma$ if we were to take $\Sigma=X$, and $C$ a 0-cell condition over $S$. 
\\\\
Int his paper, we will show that it is sufficient to consider certain definable sets $\Sigma\subseteq S\times K$ for which there is $k\in \NN$ such that every fiber $\Sigma_s$ is the disjoint union of $k$ balls. For such a $\Sigma$, the corresponding set $C^\Sigma$ will have the following structure. 

Let $\sigma_1, \ldots, \sigma_{k}$ be sections of $\Sigma$ such that for every $s \in S$, the set $\{\sigma_1(s), \ldots, \sigma_{k}(s)\}$ contains representatives of each of the $k$  disjoint balls covering $\Sigma_s$. For any such choice, $C^\Sigma$ partitions as
\[C^\Sigma = C^{\sigma_1} \cup \ldots \cup C^{\sigma_{k}}.\]
Note that $C^\Sigma$ is definable even when no section $\sigma_i$ is definable. Such sets $C^{\Sigma}$ are what we call \emph{clustered cells} %. We refer to Sections \ref{sec:refinement} and \ref{sec:separation} for a more formal definition 
(for a formal definition, see Definitions \ref{def:clustered_cell} and \ref{def:regularcluster}). The main theorem of this paper essentially states that any definable set can be partitioned as a finite union of classical and clustered cells. 

\

The remainder of this paper is structured as follows. In Section \ref{sec:semi-alg}, we will revisit semi-algebraic cell decomposition for subsets of $K$, and show that every definable set $X \subseteq K$ admits a so-called \emph{admissible} cell decomposition. Such a decomposition imposes some technical restrictions on the way centers can appear as elements of a cell, and controlling this will be crucial in later proofs.
 %that will be required to obtain Lemma \ref{lemma:annoying_balls_are_small}, Lemma \ref{lemma:externalbounded} and Proposition \ref{prop:large-small}.} 

A first strengthening of the decomposition result from \cite{cubi-leen-2015} is proven in Section \ref{sec:refinement}. This intermediate result allows us to decompose a definable set into finitely many classical cells and objects called \emph{cell arrays}. Roughly speaking, a cell array is a definable set which geometrically has the structure of a finite union of cells (which may not be definable individually), possibly involving multiple cell conditions. 
 
In Section \ref{sec:separation}, we prove a finiteness result for centers. We will use this in Section \ref{sec:largesmall} to partition cell arrays into classical and regular clustered cells (where only a single cell condition is involved). The regularity condition, which is explored in Section \ref{sec:regularity}, imposes further restrictions on the set of centers.

The full cell decomposition theorem (Theorem \ref{thm:celldecomposition}) will be presented in Section \ref{subsec:reduced}. 
This last section also includes a discussion of our main result, putting it into the context of two-sorted $P$-minimality, and adding some additional remarks and open questions.

%\\\\
%Given a clustered cell $C^\Sigma$ which is equal, as a set, to the union of $k$ cells $C^{\sigma_i}$, we would like the $C^{\sigma_i}$ to appear as independent sets whenever possible, and only allow cells to cluster together if they cannot be \emph{definably separated} (we will give a precise definition of this idea in Definition \ref{def:inseparable}). %This idea induces naturally the following notion of inseparability for multi-balls and clustered cells. 
%
%
%\sout{Geometrically, an inseparably clustered cell of order $k$ looks like a union of $k$ disjoint cells $C^{\sigma_i}$. However, we are forced to consider them as a single object because, while their union is definable, the separate cells are not (and also sub-unions of cells cannot be defined).}
%\\\\
%{\color{magenta}
%
%\noindent (The notion of regularity imposes some further conditions on the structure of the set $\Sigma$, see Definition \ref{def:semi-reduced}.)}
%
%Add something here about the structure of the rest of the paper

%\section{Some further remarks} \textcolor{red}{TO DO}

% !TEX root = centers_main.tex
\section{Semi-algebraic cell decomposition revisited} \label{sec:semi-alg}

Since every ball is the disjoint union of $q_K$ smaller balls, semi-algebraic sets $X\subseteq K$ admit infinitely many different cell decompositions. A decomposition $\mathcal{C}$ consists of the following data: a finite set $I$ and, for each $i\in I$, a cell condition $C_i$ and a center $\sigma_i\in K$. We denote this as $\mathcal{C}=\{C_i^{\sigma_i}\mid i\in I\}$. Note that since all cells are subsets of $K$, the center $\sigma$ of every cell $C^{\sigma}$ is an element of $K$ rather than a function.
We will also use the notations 
\[
\mathcal{C}(K):= \bigcup_{i\in I} C_i^{\sigma_i} \hspace{1cm} \text{and} \hspace{1cm} \Cen(\mathcal{C}):=\{\sigma_i \mid i\in I\}.  
\]

Two decompositions $\mathcal{C}$ and $\mathcal{D}$ are \emph{equivalent} if they define the same set, that is, if $\mathcal{C}(K)=\mathcal{D}(K)$. Given a set $X\subseteq K$ and a ball $B$, we use the notation $B \sqsubseteq X$ to indicate that $B$ is maximal with respect to inclusion in $X$. 
\\\\
In this section we will define a collection of so-called \emph{admissible decompositions} and show that every semi-algebraic set $X\subseteq K$ admits a decomposition from this collection. First we need to introduce some further notation. 
%
% For simplicity, cell decompositions will often just be called \emph{decompositions} and will be denoted by calligraphic letters $\mathcal{C}$ and $\mathcal{D}$ (possibly with subscripts). Since we will focus on decompositions rather than on the sets they decompose, we will use the following notation:

\begin{definition} Let  $\mathcal{C}=\{C_i^{\sigma_i}\mid i\in I\}$ be a decomposition. Define the subset of cells $\mathcal{C}^\ast\subseteq \mathcal{C}$ as 
\[
\mathcal{C}^\ast:=\{C_i^{\sigma_i}\mid \sigma_i\neq 0 \wedge \square_{1,i}=\square_{i,2}=<\}.
\]
We define the set $W(\mathcal{C})$ as the following subset of centers in $\mathcal{C}^\ast$:  
\[
W(\mathcal{C}):=\left\{\sigma\in \Cen(\mathcal{C}^\ast) \ \left| \ (\exists \gamma\in \Gamma_K) \left[B_\gamma(\sigma) \sqsubseteq \mathcal{C}^\ast (K) \wedge \bigwedge_{C_i^{\sigma_i}\in \mathcal{C}^\ast} B_\gamma(\sigma)\not\subset C_i^{\sigma_i} \right] \right\}\right..
\qedhere \]
\end{definition}
In words, $W(\mathcal{C})$ consists of those centers in $\mathrm{Centers}(\mathcal{C}^*)$ which are in $\mathcal{C}^*(K)$, but where the biggest ball in $\mathcal{C}^\ast(K)$ around this center is not contained within a single cell of $\mathcal{C}^*$.
We are now able to define what admissible decompositions are. 

\begin{definition}\label{def:admissible} A decomposition $\mathcal{C}=\{C_i^{\sigma_i}\mid i\in I\}$ is called \emph{pre-admissible} if it satisfies the following properties: 
\begin{enumerate} 
\item[(a)] For every 0-cell $C_i^{\sigma_i}$, if $\sigma_i\neq 0$ then $\sigma_i\in X\setminus\Int(X)$. 
\item[(b)] For every 1-cell $C_i^{\sigma_i}$, if $\sigma_i\neq 0$ and $\square_{i,1} = <$ then $\ord \sigma_i\leqslant \alpha_i$. 
\item[(c)] For every 1-cell $C_i^{\sigma_i}$ in which $\square_{i,1}=\emptyset$, it holds that $\sigma_i=0$.
\end{enumerate}
It is called \emph{admissible} if it moreover satisfies 
\begin{enumerate}
\item[(d)] $W(\mathcal{C})=\emptyset$. \qedhere
\end{enumerate}
\end{definition}
\noindent Condition (a) ensures that elements defined by 0-cells different from $\{0\}$ are isolated points. Condition (c) will later imply that cells for which $\square_1=\emptyset$, will always be centered at 0. Conditions (b) and  (d), which might seem arbitrary at this point, will be needed for technical reasons in later proofs.\\\\ The goal of this section is to prove the following theorem:

\begin{theorem}\label{thm:admissible}
Every semi-algebraic set $X\subseteq K$ has an admissible cell decomposition. 
\end{theorem}

We split the proof of Theorem \ref{thm:admissible} into two steps: we first show (in the next lemma) that semi-algebraic sets always have a pre-admissible decomposition. The second step will then be to prove that every pre-admissible decomposition can be modified into an admissible one.

\begin{lemma}\label{lem:goodcells} Every semi-algebraic set $X\subseteq K$ has a pre-admissible decomposition.
\end{lemma}

\begin{proof} Let $\mathcal{C}=\{C_i^{\sigma_i}\mid i\in I\}$ be a cell decomposition of $X$. Let $a(\mathcal{C})$ ( resp.\ $b(\mathcal{C})$ and $c(\mathcal{C})$) be the number of cells in $\mathcal{C}$ which are counterexamples of part (a) of Definition \ref{def:admissible} (resp.\ of (b) and (c)). If $a(\mathcal{C})>0$ (resp.\ $b(\mathcal{C})>0,\  c(\mathcal{C})>0$), we will show how to produce a cell decomposition $\widehat{\mathcal{C}}$ of $X$ such that $a(\widehat{\mathcal{C}})\leqslant a(\mathcal{C})-1$ (and similarly for $b(\mathcal{C})$ and $c(\mathcal{C})$). By iterating this process a finite number of times, one can then obtain a cell decomposition satisfying (a) (resp.\ (b) and (c)). 

Fix an index $j\in I$ such that $\sigma_j\neq 0$, $C_{j}^{\sigma_{j}}$ is the 0-cell $C_{j}^{\sigma_{j}}=\{\sigma_{j}\}$ and $\sigma_{j}\in \Int(X)$. Write $\mathcal{C} = \{C_{j}^{\sigma_j}\} \cup \mathcal{C}_1 \cup \mathcal{C}_2$, where
\[ \mathcal{C}_1 := \{ C_i^{\sigma_i} \in \mathcal{C} \mid i\neq j \wedge \sigma_j \in \text{Cl}(C_i^{\sigma_i})\},\] and $\mathcal{C}_2 = \mathcal{C} \setminus (\{C_{j}^{\sigma_j}\} \cup \mathcal{C}_1)$.
Let $X'$ be the set $X' = C_{j}^{\sigma_{j}} \cup \mathcal{C}_1(K)$. Let $\gamma\in \Gamma_K$ be minimal such that $B_{\gamma}(\sigma_{j})$ is contained in $X'$. If no minimal $\gamma$ exists, set $\gamma:=\ord(\sigma_{j})$. Note that this case  only occurs when $X'=K$. Indeed,  by a result of Cluckers (see Lemma 2 and Theorem 6 of  \cite{clu-presb03}), $P$-minimal definable subsets of $\Gamma_K$ are Presburger-definable. From this it follows that
every definable subset of $\Gamma_K$ without a minimal element must be unbounded from below, hence $X'$ contains arbitrarily large balls. Let $\zeta \in K$ be such that $\ord(\sigma_j-\zeta)=\gamma-1$ and let $D^{\zeta}$ be the cell
\[
D^{\zeta}:=\{t\in K \mid \ord(t-\zeta)=\gamma-1 \wedge t-\zeta\in\lambda Q_{1,1}\},
\]
where we have chosen $\lambda \in K$ such that $D^{\zeta} = B_{\gamma}(\sigma_{j})$. 
For every 1-cell $C_i^{\sigma_i}\in \mathcal{C}_1$,  
let $D_i^{\sigma_i}$ be the 1-cell obtained from $C_i^{\sigma_i}$ by replacing $\square_{i,2}$ by $<$ and making $\gamma$ the upper bound. Then the set of cells $\widehat{\mathcal{C}}$ formed by 
\[ 
\{D^{\zeta}\} \cup \{D_i^{\sigma_i} \mid C_i^{\sigma_i} \in \mathcal{C}_1\} \cup \mathcal{C}_2
\]
is a cell decomposition of $X$. Clearly, $a(\widehat{\mathcal{C}})\leqslant a(\mathcal{C})-1$. 
\\\\
Suppose that $\mathcal{C}$ satisfies (a). Let $C_{j}^{\sigma_{j}} \in \mathcal{C}$ be a 1-cell centered at $\sigma_{j}\neq 0$ for which either $\alpha_{j}<\ord(\sigma_{j})$, or $\square_{1j} = \emptyset$.   We need to consider two cases, depending on whether $\ord(\sigma_{j})<\beta_{j}$ or $\beta_{j}\leqslant\ord(\sigma_{j})$. We will only discuss the first case in detail, as the second one is completely similar.
If $\ord(\sigma_{j})<\beta_{j}$, first partition the cell $C_{j}^{\sigma_{j}}$ further as \begin{align*}
D^{\sigma_{j}}&:=\{t\in K \mid \ord(\sigma_{j})<\ord(t-\sigma_{j}) \ \square_{j2} \ \beta_{j} \wedge  t-\sigma_{j}\in\lambda_{j} Q_{n_{j},m_{j}}\},\\
E&:=\{t\in K\mid \alpha_{j} \ \square_{j1} \ \ord(t-\sigma_j) \  < \ord(\sigma_{j}) +1 \wedge t-\sigma_j\in\lambda_{j} Q_{n_{j},m_{j}}\}.
\end{align*}
To prove our claim, we need to show how the cell $E$ can be partitioned as a finite union of 1-cells centered at 0. %{\color{gray}The partition of a cell $C_j^{\sigma_j}$ where $\beta_{j}\leqslant\ord(\sigma_{j})$ can be done in a similar way.} 
Put $M_j:= \min\{m_j, \ord(\sigma_j) -\alpha_j\}$ (or just $M_j = m_j$ if $\square_{j1} = \emptyset$). We will first partition $E$ further as $E' \cup E_0 \cup \ldots \cup E_{M_j-1}$, where
\begin{align*}
E' &:= \{t \in K \mid \alpha_{j} \ \square_{j1} \ \ord(t - \sigma_j) \  < \ord(\sigma_{j}) -m_j +1 \wedge t - \sigma_j\in\lambda_{j} Q_{n_{j},m_{j}}\},\\
E_{i} &:= \{t \in K \mid \ord(t-\sigma_j) = \ord(\sigma_{j}) - i \wedge t - \sigma_j\in\lambda_j Q_{n_{j},m_{j}}\},
\end{align*}
Note that most of these sets are actually already cells centered at zero (and some might be empty). Indeed, for $E'$ we can rewrite the description of the set as
\[E' = \{t \in K \mid \alpha_{j} \ \square_{j1} \ \ord(t) \  < \ord(\sigma_{j}) -m_j +1 \wedge t\in\lambda_{j} Q_{n_{j},m_{j}}\}.\]
 Similarly, for $1\leqslant i \leqslant M_j -1$, we have that 
\[E_{i} = \{t \in K \mid \ord(t) = \ord(\sigma_{j}) - i \wedge t\in\mu_{i} Q_{n_{j},m_{j}}\},\]
  where $\mu_{i} \in K$ is chosen in such a way as to assure that $t-\sigma_j \in \lambda_{j} Q_{n_j,m_j}$. %Note that some of these cells might be empty.
  \\\\ 
 When $i=0$, we need to do a bit more work. A further partitioning will be necessary. For $0\leqslant k < m_j$, let $E_{0, k}$ be the set
 \[ E_{0,k}:= \{t \in E_0 \mid \ord t = \ord \sigma_j +k \},\]
 and we write $E_{0, >}$ for the set
 \[E_{0,>}:= \{t \in E_0 \mid \ord t \geqslant \ord \sigma_j +m_j\}.\]
 Then clearly, if they are non-empty, the sets $E_{0,k}$ are cells centered at zero, since for a suitably chosen value $\mu_{0,k} \in K$, they can be rewritten as
 \[E_{0,k} = \{t \in K \mid \ord(t) = \ord(\sigma_{j}) +k \wedge t\in\mu_{0,k} Q_{n_{j},m_{j}-k}\}.\]
 Finally, consider the set $E_{0,>}$. First note that this set is empty unless $-\sigma_j \in \lambda_jQ_{n_j,m_j}$, as for elements of this set it holds that $\left[t-\sigma_j \in \lambda_jQ_{n_j,m_j} \Leftrightarrow -\sigma_j \in \lambda_jQ_{n_j,m_j}\right]$. Moreover, if  $E_{0,>}$ is non-empty, it equals the ball $B_{\ord(\sigma_j) + m_j}(0)$. In this case, we will partition $E_{0,>}$ into cells $\{F_0, \ldots, F_{q_K-1}\}$ as follows.
Put $F_{0}:=\{0\}$ and for $1\leqslant r\leqslant q_K-1$, define
\[
F_{r}:=\{t\in K \ | \  \ord(\sigma_j)+m_j-1 < \ord(t) \wedge t\in \hat{\mu}_{r}Q_{1,1},\}
\]
where $\hat{\mu}_{1},\ldots,\hat{\mu}_{q_K-1} \in K$ are representatives such that $ac_1(K^\times)=ac_1(\{\hat{\mu}_{1},\ldots,\hat{\mu}_{q_K-1}\})$. 
To summarize, we obtain the following decomposition of $E_0$, which we will denote as $\mathcal{E}_0$. Put
\[\mathcal{E}_0:= \left\{ \begin{array}{lcl}
\{E_{0,k} \mid 0 \leqslant k < m_j\} \cup \{F_{r} \mid 0 \leqslant r \leqslant q_K -1\} & \ & \text{if} -\sigma_j \in \lambda_jQ_{n_j,m_j}, \\
\{E_{0,k} \mid 0 \leqslant k < m_j\} & \ & \text{otherwise}.
\end{array}\right.\]
%We will write $\mathcal{E}_0$ for this decomposition of $E_{0, <}$. Hence, we put $\mathcal{E}_0 := \emptyset$ if $E_{0, <}$ is empty, and $\mathcal{E}_0 := \{F_{r} \mid 0 \leqslant r \leqslant q_K -1\}$ otherwise.
 
%  Finally,   
%for $i=0$ we have two cases depending on whether $-\sigma_j\in \lambda_jQ_{n_j,m_j}$ or not. In the first case, $0\in C^{\sigma_j, \ord(\sigma_j)}$ and we split this leaf as the union of cells $E_{0,0},\ldots,E_{0,q_K-1}$ defined as follows. The cell $E_{0,0}=\{0\}$ and for $1\leqslant r\leqslant q_K-1$ 
%\[
%E_{0,r}:=\{t\in K \ | \ \ord(\sigma_j)+m-1<\ord(t) \wedge t\in \mu_{0,r}Q_{1,1},\}
%\]
%where $\mu_{0,1},\ldots,\mu_{0,q_K-1} \in K$ are representatives such that $ac_1(K^\times)=ac_1(\{\mu_{0,1},\ldots,\mu_{0,q_K-1}\})$. {\color{gray}In this case we put $\mathcal{E}_0 := \{E_{0,r} \mid 0 \leqslant r \leqslant q_K -1\}$, hence $\mathcal{E}_0$ is a decomposition of $E_0$.} When $0\notin C^{\sigma_j, \ord(\sigma_j)}$ (and $C^{\sigma_j, \ord(\sigma_j)}\neq\emptyset$), then all elements of this leaf have the same order, say $\gamma {\color{gray}< \ord(\sigma_j)+m_j}$. Then, for $l:=m_j-(\gamma-\ord(\sigma_j)){\color{gray}\geqslant 1}$ we define the set $E_0$ to be
%\[
%E_{0}:=\{t\in K \ | \ \ord(t)=\gamma \wedge t\in \mu_{0}Q_{1,l}\}
%\]
%where $\mu_0\in K$ is chosen such that $E_0=C^{\sigma_j, \ord(\sigma_j)}$. {\color{gray}In this case we put $\mathcal{E}_0:= \{E_0\}$, which is obviously a decomposition of $E_0$.}
%
%\
 Now let $\widehat{\mathcal{C}}$ be the decomposition obtained by replacing $C_{j}^{\sigma_{j}}$ by the cells in $\{D^{\sigma_j}, E'\} \cup\{E_{i} \mid 1 \leqslant i \leqslant M_j{-1}\} \cup \mathcal{E}_0$.
 If $C_j^{\sigma_j}$ was a cell contradicting (b), (resp. (c)), then $\widehat{\mathcal{C}}$ is a cell decomposition of $X$ for which $b(\widehat{\mathcal{C}})=b(\mathcal{C})-1$ and $c(\widehat{\mathcal{C}}) \leqslant c(\mathcal{C})$ (resp. $c(\widehat{\mathcal{C}})=c(\mathcal{C})-1$ and $b(\widehat{\mathcal{C}}) \leqslant b(\mathcal{C})$). Moreover, no new 0-cells that are not centered at 0, were added during this process, so $\widehat{\mathcal{C}}$ still satisfies property (a). Repeating this partitioning process for a finite number of cells then yields the lemma.
%We may therefore assume from now on that $\mathcal{C}$ satisfies both (a) and (b). Property (c) can be achieved by a process which is very similar to what we did for (b), and we leave this case to the reader.
\end{proof}

It remains to show that every pre-admissible decomposition allows an equivalent admissible decomposition. 
%It suffices to show this for pre-admissible decompositions $\mathcal{C}$ for which all cells have $\square_1=\square_2=<$. Indeed, notice that any subset of a pre-admissible decomposition is again pre-admissible. Then, suppose $\mathcal{C}$ is a pre-admissible decomposition and let $\mathcal{D}$ be the subset of cells in $\mathcal{C}$ for which $\square_{1}=\square_{2}= <$. By assumption, there is an equivalent admissible decomposition $\mathcal{D}'$ of $\mathcal{D}$. To conclude, notice that by the form of condition (d) in Definition \ref{def:admissible}, the decomposition $\mathcal{C}'=(\mathcal{C}\setminus \mathcal{D}) \cup \mathcal{D}'$ is an equivalent admissible decomposition of $\mathcal{C}$. 
We need the introduce some additional notations first. Given a cell $C^{\sigma}$ with $\square_1=\square_2=<$, and an interval $(\alpha',\beta')$,  we put
\[
C^{\sigma}_{|(\alpha',\beta')}:=\{t\in K \mid \widetilde{\alpha} \ < \ \ord(t-\sigma) \ < \widetilde{\beta} \wedge t-\sigma\in\lambda Q_{n,m}\},
\]
where $(\widetilde{\alpha}, \widetilde{\beta}) = (\alpha,\beta) \cap (\alpha',\beta')$,
\begin{lemma}\label{lem:admissible} Let $\mathcal{C}$ be a pre-admissible decomposition. 
%such that all cells have $\square_1=\square_2=<$. 
Then there exists an equivalent decomposition $\mathcal{D}$ which is admissible.  
\end{lemma}

\begin{proof} We use induction on $l$, for $0\leqslant l\leqslant L=|W(\mathcal{C})|$, to show that there exist equivalent pre-admissible decompositions $\mathcal{D}_l$ such that 
\begin{enumerate}
\item $\mathcal{D}_0=\mathcal{C}$;
\item if $W(\mathcal{D}_{l})\neq\emptyset$ then $|W(\mathcal{D}_{l+1})|\leqslant |W(\mathcal{D}_l)|-1$.
\end{enumerate}
The result will then follow by putting $\mathcal{D}:=\mathcal{D}_{L}$. 
For $l=0$, there is nothing to prove. Suppose that $\mathcal{D}_l:=\{C_j^{\sigma_j}\mid j\in J\}$ has already been constructed. If $W(\mathcal{D}_l)=\emptyset$, we set $\mathcal{D}_{l+1}=\mathcal{D}_l$ and there is again nothing to prove. 

Otherwise, let $J^\ast\subseteq J$ be the set $J^\ast := \{ j\in J \mid C_j^{\sigma_j} \in \mathcal{D}_l^\ast\}$.
Choose an element  $j_0\in J^\ast$  such that $\sigma_{j_0}\in W(\mathcal{D}_l)$. 
By the definition of $W(\mathcal{D}_l)$, $\sigma_{j_0}\neq 0$ and there is $\rho\in \Gamma_K$ such that $B_{\rho}(\sigma_{j_0})\sqsubseteq \mathcal{D}_l^\ast(K)$ and $B_{\rho}(\sigma_{j_0})$ is not contained in a single cell $C_j^{\sigma_j}$ of $\mathcal{D}_l^\ast$. Let ${J'}\subset J^\ast$ be minimal such that 
\[
B_{\rho}(\sigma_{j_0})\subseteq \bigcup_{j\in {J'}} C_j^{\sigma_j}. 
\]
Note that $|{J'}|\geqslant 2$. For each $j\in {J'}$, let $Y_j$ be the subset of $\Gamma_K$ defined by 
\[
Y_j:=\{\gamma\in \Gamma_K \mid B_{\rho}(\sigma_{j_0})\cap C_j^{\sigma_j,\gamma}\neq \emptyset\}. 
\]
Then we have that 
\[
B_{\rho}(\sigma_{j_0})= \bigcup_{j\in {J'}}\bigcup_{\gamma\in Y_j} C_j^{\sigma_j,\gamma}. 
\]
Let $\gamma_{j,1}:=\min\{\gamma: \gamma\in Y_j\}$ and $\gamma_{j,2}:=\max\{\gamma: \gamma\in Y_j\}$. 

\begin{claim}\label{claim:Y} The following equality holds
\[
Y_j = \{\gamma\in \Gamma_K \mid \gamma_{j,1} \ \leqslant \ \gamma \ \leqslant \ \gamma_{j,2} \ \wedge \ \gamma\equiv \ord(\lambda_j)\mod n_j\}. 
\]
\end{claim}

The inclusion from left to right is trivial. For the remaining inclusion let $\gamma\in \Gamma_K$ be an element of the right-hand set. Since for $k=1,2$ the leaves
$C_j^{\sigma_j,\gamma_{j,k}}$ are subsets of  $B_{\rho}(\sigma_{j_0})$, 
the ball $B_{\rho}(\sigma_{j_0})$ must contain the smallest ball containing both leaves. Clearly such a ball contains $C_j^{\sigma_j,\gamma}$, which proves the claim.  

\

By Claim \ref{claim:Y}, we have that 
\begin{equation}\label{eq:cells1}%\tag{$\star$}
\bigcup_{j\in J'} C_j^{\sigma_j} = B_\rho(\sigma_{j_0}) \cup \bigcup_{j\in J'} C^{\sigma_j}_{j\ |(\alpha_j,\gamma_{j,1})} \cup \, C^{\sigma_j}_{j\ |(\gamma_{j,2},\beta_j)}.
\end{equation}
Note that some of these cells might be empty. We will now need to distinguish between three cases, indexed as $d=1,2,3$. For each case, one can define a decomposition $\mathcal{E}_{d}$ such that $\mathcal{E}_d(K)=B_\rho(\sigma_{j_0})$ as follows: 

\

\textbf{Case $d=1$:} Suppose that $0\in B_\rho(\sigma_{j_0})$. We will partition this ball as a union of cells $D_i^{0}$ which are centered at $0$. Let $D_0^0$ be the 0-cell $\{0\}$. Choose representatives $\mu_1,\ldots,\mu_{q_K-1} \in K$  such that $ac_1(K^\times)=ac_1(\{\mu_1,\ldots,\mu_{q_K-1}\})$. For $1\leqslant i\leqslant q_K-1$, we define the cells $D_i^0$ as follows:
\[
D_i^0:=\{t\in K\mid \rho-1<\ord(t) \wedge t \in \mu_iQ_{1,1}\}.
\] 
Now put $\mathcal{E}_1:=\{D_i^0\mid i\in\{0,\ldots,q_K-1\}\}$. One can check that $\mathcal{E}_1(K)=B_\rho(\sigma_{j_0})$. 

\

\textbf{Case $d=2$:} Suppose that $0\notin B_\rho(\sigma_{j_0})$, and that there exists $m\in\NN\setminus\{0\}$ such that $\ord(\sigma_{j_0}) =\rho-m$. Let $\lambda\in K$ be such that $B_{\rho}(\sigma_{j_0})$ is equal to the cell centered at zero
\[
E^{0}:=\{t\in K \mid \ord(t)= \rho-m \wedge t\in\lambda Q_{1,m}\}.
\]
If we put $\mathcal{E}_2=\{E^0\}$, then clearly it holds that $\mathcal{E}_2(K)=B_\rho(\sigma_{j_0})$.

\

\textbf{Case $d=3$:} Suppose that $0 \not \in \ B_\rho(\sigma_{j_0})$ and $\rho - \ord(\sigma_{j_0})>m$  for all $m\in\NN$. Since $B_{\rho}(\sigma_{j_0})\sqsubseteq\mathcal{D}_l^*(K)$, there exists $\zeta\in B_{\rho-1}(\sigma_{j_0})\setminus\mathcal{D}_l^*(K)$. In this case we have that 
\begin{equation}\label{eq:case3}
\ord(\zeta)=\ord(\sigma_{j_0})<\rho-m
\end{equation}
for every $m\in\NN$, so in particular $\zeta\neq 0$. Let $\lambda\in K$ be such that $B_{\rho}(\sigma_{j_0})$ is equal to the cell
\[
D^{\zeta}:=\{t\in K \mid \ord(t-\zeta)=\rho-1 \wedge t-\zeta\in\lambda Q_{1,1}\}.
\]
Define $\mathcal{E}_3=\{D^\zeta\}$, which again clearly satisfies $\mathcal{E}_3(K)=B_\rho(\sigma_{j_0})$.

\

Finally define $\mathcal{D}_{l+1}$ as 
\[
\mathcal{D}_{l+1}:=\bigcup_{j\in J\setminus J'} \{C_j^{\sigma_j}\} \cup \bigcup_{j\in J'} \{C^{\sigma_j}_{j\ |(\alpha_j,\gamma_{j,1})}\} \bigcup_{j\in J'} \{C^{\sigma_j}_{j\,|(\gamma_{j,2},\beta_j)}\}\cup \mathcal{E}_d, 
\]
where $d=1,2,3$ depending on the previous case distinction. 

\

The identity (\ref{eq:cells1}) shows that in all three cases, $\mathcal{D}_{l+1}$ is equivalent to $\mathcal{D}_l$. Let us now discuss why $\mathcal{D}_{l+1}$ is pre-admissible. First note that, if a cell $C_j^{\sigma_j}$ satisfies conditions (a)-(c) from Definition \ref{def:admissible}, then any restriction $C^{\sigma_j}_{j\ |(\alpha_j',\beta_j')}$ will also satisfy these conditions. Therefore, since $\mathcal{D}_l$ is pre-admissible, by the definition of $\mathcal{D}_{l+1}$ it suffices to check that the cells in $\mathcal{E}_d$ also satisfy conditions (a)-(c). Suppose first that $d=1$ or $d=2$. In both cases, all cells in $\mathcal{E}_d$ are centered at 0, so they satisfy these conditions by default. Now consider the remaining case, $\mathcal{E}_3=\{D^\zeta\}$. Since $D^\zeta$ is not a 0-cell and $\square_{1}\neq\emptyset$, conditions (a) and (c) are trivially satisfied. For condition (b) one needs to check that $\ord(\zeta)\leqslant\rho-2$, but this follows immediately from \eqref{eq:case3}. Hence, $\mathcal{D}_{l+1}$ is pre-admissible. 
\\\\It remains to show that $|W(\mathcal{D}_{l+1})|\leqslant |W(\mathcal{D}_l)|-1$. 

\begin{claim} \label{claim}
$W(\mathcal{D}_{l+1})\subseteq W(\mathcal{D}_l)$. 
\end{claim}
Let $\sigma\in W(\mathcal{D}_{l+1})$, and let $\delta\in \Gamma_K$ be such that $B_\delta(\sigma) \sqsubseteq \mathcal{D}_{l+1}^\ast(K)$ and $B_\delta(\sigma) $ is not contained in a single cell of $\mathcal{D}_{l+1}^\ast(K)$. We split in cases: 

\

\textbf{Case $d=1$ and $d=2$:} In both cases, $\mathcal{E}_d$ only consists of cells centered at 0. Therefore, $\mathcal{D}_{l+1}^\ast=(\mathcal{D}_{l+1}\setminus\mathcal{E}_d)^\ast$, which implies that  
\begin{equation}\label{eq:cells3}
B_\delta(\sigma)\subseteq \bigcup_{j\in J^{*}\setminus J'} C_j^{\sigma_j} \cup \bigcup_{j\in J'} C^{\sigma_j}_{j\ |(\alpha_j,\gamma_{j,1})} \bigcup_{j\in J'} C^{\sigma_j}_{j\,|(\gamma_{j,2},\beta_j)}. 
\end{equation}
Suppose first that there exists a single $j\in J'$ such that 
\begin{equation}\label{eq:cells2}
B_{\delta}(\sigma)\subseteq C^{\sigma_{j}}_{j\,|(\alpha_{j},\gamma_{j,1})} \cup C^{\sigma_{j}}_{j\,|(\gamma_{j,2},\beta_{j})}. 
\end{equation}
Our assumption on $B_{\delta}(\sigma)$ implies that $B_{\delta}(\sigma)$ intersects both cells on the right hand-side of (\ref{eq:cells2}). This situation cannot occur, since $B_{\delta}(\sigma)$ would then necessarily intersect leaves $C_{j}^{\sigma_{j},\gamma}$ with $\gamma_{j,1}\leqslant\gamma\leqslant\gamma_{j,2}$ as well, but these are not part of the union on the right hand-side of (\ref{eq:cells2}). 
Hence, the ball $B_\delta(\sigma)$ must have non-zero intersection with at least two cells that already occurred in the decomposition $\mathcal{D}_l^*$, % This shows there is some $\delta'\in \Gamma_K$ such that the ball $B_{\delta'}(\sigma)$ is contained in $\mathcal{D}_{l}(K)$ but is not contained in a single cell of $\mathcal{D}_{l}$, 
which means that $\sigma\in W(\mathcal{D}_l)$. This completes this case. 

\

\textbf{Case $d=3$:} By construction, we have that
\[
\Cen(\mathcal{D}_{l+1}^\ast)\subseteq \Cen(\mathcal{D}_l^\ast)\cup\{\zeta\}.
\] 
Note that $\zeta\notin \mathcal{D}_l^*(K)=\mathcal{D}_{l+1}^*(K)$, where the equality holds since we only added or altered cells with non-zero centers, for which $\square_1 = \square_2 = <$. Therefore we must have that $\sigma\neq\zeta$, hence $\sigma\in \Cen(\mathcal{D}_l^\ast)$. It suffices to show that $B_\delta(\sigma)\cap D^{\zeta}=\emptyset$. Indeed, if this intersection is empty, then the inclusion (\ref{eq:cells3}) will hold since $B_\delta(\sigma) \sqsubseteq \mathcal{D}_{l+1}^\ast(K)$, and we can conclude as in case 1. Suppose for a contradiction that $B_\delta(\sigma)\cap D^{\zeta}\neq\emptyset$. Recall that by construction, $D^{\zeta}=B_{\rho}(\sigma_{j_0})$ is a ball. Therefore, since no cell in $\mathcal{D}_{l+1}^\ast$ contains $B_\delta(\sigma)$ as a subset, we must have that $D^{\zeta}\subsetneq B_\delta(\sigma)$. This in turn implies that $B_{\rho-1}(\sigma_{j_0})\subseteq B_{\delta}(\sigma)$. Now since $\zeta\in B_{\rho-1}(\sigma_{j_0})$, the previous inclusion contradicts that $\zeta\notin \mathcal{D}_{l+1}^*(K)$. This completes the claim. 

\

%Now, by the definition of $\mathcal{D}_{l+1}$, the argument below shows again that
%\[
%B_\delta(\sigma)\subseteq \{C_j^{\sigma_j} \mid j\in J\setminus J'\}\cup \{C^{\sigma_j}_{j\ |(\alpha_j,\gamma_{j,1})} \mid  j\in J'\}\cup
%\{C^{\sigma_j}_{j\,|(\gamma_{j,2},\beta_j)} \mid j\in J'\}. 
%\]
%We conclude as in Case 1. This completes the claim. 

It follows from Claim \ref{claim} that $|W(\mathcal{D}_{l+1})|\leqslant |W(\mathcal{D}_l)|$. We show that $\sigma_{j_0}\notin W(\mathcal{D}_{l+1})$, which will imply that $|W(\mathcal{D}_{l+1})|\leqslant |W(\mathcal{D}_l)|-1$, since by assumption $\sigma_{j_0}\in W(\mathcal{D}_l)$. Again we split in cases. Suppose first that $d=1$ or $d=2$. In both cases, $\sigma_{j_0}$ is contained in a cell of $\mathcal{E}_d$, and hence cannot be contained in a cell of $\mathcal{D}_{l+1}^\ast$. For case $d=3$, suppose towards a contradiction that there is some $\delta\in\Gamma_K$ witnessing that $\sigma_{j_0}\in W(\mathcal{D}_{l+1})$. If $\delta\geqslant\rho$, then the ball $B_\delta(\sigma_{j_0})$ would be contained in $D^{\zeta}$, and since $D^\zeta\in\mathcal{D}_{l+1}^\ast$, this contradicts the assumption that $\sigma_{j_0}\in W(\mathcal{D}_{l+1})$. If $\delta<\rho$, then $\zeta\in B_{\delta}(\sigma_{j_0})\sqsubseteq \mathcal{D}_{l+1}^\ast(K)$, which contradicts that $\zeta\notin \mathcal{D}_{l+1}^*(K)$.   
%Suppose towards a contradiction there is some $\delta\in\Gamma_K$ witnessing that $\sigma_{j_0}\in W(\mathcal{D}_{l+1})$. If $\delta\geq\rho$, then the ball $B_\delta(\sigma_{j_0})$ is contained in $D^{\zeta}$ hence not in $W(\mathcal{D}_{l+1})$. If $\delta<\rho$, then $\zeta\in B_{\delta}(\sigma_{j_0})\sqsubseteq \mathcal{D}_{l+1}(K)$, which contradicts that $\zeta\notin \mathcal{D}_{l+1}(K)$. }
\end{proof}

\begin{proof}[Proof of Theorem \ref{thm:admissible}]
This is an immediate consequence of Lemmas \ref{lem:goodcells} and \ref{lem:admissible}. 
\end{proof}

%We finish the section by the following remark that will be used through the rest of the text.  

%\begin{remark} Let $\mathcal{C}=\{C_i^{\sigma_i}\mid i\in I\}$ be an admissible decomposition. We show some operations which preserve admissibility. First of all, notice that any subset of $\mathcal{C}$ is again an admissible decomposition. Suppose that all cells have $\square_{i,1}=\square_{i,2}=<$ and let $Y$ be a finite subset of $\Gamma_K$. Consider the set 
%\[
%\Delta:=\{\alpha_i\mid i\in I\} \cup \{\beta_i\mid i\in I\}\cup Y.  
%\]
%The set $\Delta$ naturally partitions $\Gamma_K$ into finitely many intervals 
%\[
%\Gamma_K=(\delta_0,\delta_1)\cup(\delta_2,\delta_3)\cup\cdots\cup (\delta_{2N-2},\delta_{2N-1})\cup (\delta_{2N},\delta_{2N+1})
%\] 
%where $\delta_0=-\infty$, $\delta_{2N+1}=\infty$ and for each $i\in I$, there are $0<k_i<l_i< N$ such that 
%\[
%(\alpha_i,\beta_i)= (\delta_{2k_i},\delta_{2k_i+1})\cup\cdots\cup (\delta_{2l_i},\delta_{2l_i+1}). 
%\]
%Consider the equivalent decomposition $\mathcal{D}$ which arises after replacing each cell $C_i^{\sigma_i}$ by the restrictions 
%\[
%{C_i^{\sigma_i}}_{\mid (\delta_j,\delta_j+1)} \hspace{1cm} \text{ for each } j=2k, k_i\leq k\leq l_i.
%\]
%After possibly shortening further some intervals so that cells with $\square_2=<$ have a leaf at $\beta-1$, such decomposition is also admissible. Thus, cells in an admissible decomposition of this form can always be  
%\end{remark}

% !TEX root = centers_main.tex
\section{Refinement of the decomposition} \label{sec:refinement}

%\subsection{Further definitions} 
In \cite{cubi-leen-2015}, two of the authors proved a weak, but unconditional version of cell decomposition for $P$-minimal structures. The building blocks used in that theorem are closely related to (classical) cells, but have a far more complex structure.
As a first step towards the main result of this paper, we will restate this version (using slightly different terminology) and consider some refinements of it, which will lead to Theorem \ref{thm:refinement}.
This theorem will be used as a basis for further improvements in later sections, where we will step by step reduce the complexity of the sets involved. We first need the following notations and definitions. 
\\\\
Let $S$ be a parameter set and $\Sigma \subseteq S \times K^r$ be a definable set. %$\Sigma \subseteq S \times K^r = \{(s, \sigma_1, \ldots, \sigma_r) \mid \ldots\}$. 
For each $i =1, \ldots, r$, we write $\Sigma^{(i)}$ for the projection
\[ \Sigma^{(i)}:=  \{ (s, c) \in S \times K \mid \exists \zeta_k: (s, \zeta_1, \ldots, \zeta_{i-1},c, \zeta_{i+1}, \ldots, \zeta_r) \in \Sigma\},\]
and $\Sigma^{(i)}_s$ for its fibers $(\Sigma^{(i)})_s$.

\begin{definition}\label{def:multi-cell} Let $C_1, \ldots, C_r$ be cell conditions and $\Sigma \subseteq S \times K^r$ be a definable set. The pair $\mathcal{A}=(\{C_i\}_{1\leqslant i\leqslant r}, \Sigma)$ is called a \emph{multi-cell} if the following conditions hold:
\begin{itemize}
\item[(i)] Every section $\sigma: s \mapsto (\sigma_1(s), \ldots, \sigma_r(s))$ of $\Sigma$ induces the same set $X$, where \[X =  C_1^{\sigma_1} \cup \ldots \cup C_r^{\sigma_r}.\]
We say that $X$ is \emph{the set defined or induced by $\mathcal{A}$}, and we also denote it by $\mathcal{A}(K)$.  
\item[(ii)] For every section $\sigma$ of $\Sigma$, the induced potential cells $C_1^{\sigma_1}, \ldots, C_r^{\sigma_r}$ are all disjoint.
%\item[(iii)] \textcolor{red}{This item needs to be removed here} %if $(s, \sigma_i(s)) \in \Sigma^{(i)}$ and $\square_{i,1} = <$, then $\Sigma^{(i)}$ also contains all other  $(s, \sigma_i'(s))$ for which $C_i^{\sigma_i(s)} = C_i^{\sigma'_i(s)}$. 
%Whenever $c \in \Sigma^{(i)}_s$ and $\square_{i,1}$ and $\square_{i,2}$ denote $<$, the set $\Sigma^{(i)}_s$ also contains all $ c' \in K$ such that 
%\[
%(\forall t) (C_i(s,c,t)\leftrightarrow C_i(s,c',t)). 
%\]
\end{itemize}
The multi-cell $\mathcal{A}$ is called \emph{admissible} if for every section $\sigma$ of $\Sigma$ and every $s\in S$, the fibers $C_i^{\sigma_i(s)}$ form an admissible decomposition of $X_s$.  
\end{definition}
We want to stress that the partition in part (i) of the above definition depends on the choice of section, and that different sections of $\Sigma$ will in general induce different partitions of $X$.
%\sout{We will use the notation $X = ( \{C_i\}_i, \Sigma)$ to denote that the given multi-cell induces the set $X$. 
If $\mathcal{A}$ is a multi-cell and $X=\mathcal{A}(K)$, then clearly $X$ is a definable subset of $S\times K$. %Note, however, that different sections of $\Sigma$ will in general induce different partitions of $X$ as in part (i). 
As is common practice in model theory, we will also refer to the set $X$ \emph{itself} as a multi-cell, by which we mean that there exists some multi-cell $\mathcal{A}$ such that $X=\mathcal{A}(K)$. %Hence, if $X=\mathcal{A}(K)$ for some multi-cell $\mathcal{A}$, we abuse of terminology and refer both to $\mathcal{A}$ and $X$ as multi-cells.  
\\\\
The cell decomposition theorem from \cite{cubi-leen-2015} can then be stated in the following way:

\begin{theorem}\label{thm:cellpartition} Let $X\subseteq S\times K$ be a definable set. There exists a finite partition of $X$ into multi-cells. 
\end{theorem}

We can now state a first refinement of Theorem \ref{thm:cellpartition}. Its proof is a word-for-word analogue of the proof of Theorem \ref{thm:cellpartition}, in which we replace each semi-algebraic cell decomposition by an admissible decomposition using Theorem \ref{thm:admissible}.
 
\begin{theorem}\label{thm:goopartition} Let $X \subseteq S \times K$ be a definable set. There exists a finite partition of $X$ into admissible multi-cells. 
\end{theorem}
 
Theorem \ref{thm:refinement} will be a refinement of the above theorem. In order to state it, we need the following definitions first. 

\begin{definition}\label{def:clustered_cell} A set $X\subseteq S\times K$ is a \emph{classical cell} if there exist a cell condition $C$ over $S$, and a definable function $\sigma:S\to K$ such that $X=C^\sigma$. 

The set $X$ is a \emph{clustered cell} if there exist a cell condition $C$ over $S$, and a definable set $\Sigma\subseteq S\times K$ such that $X=C^\Sigma$ and the following holds:
\begin{itemize}
\item[(1)] $C$ is a 1-cell condition over $S$, and both $\square_1$ and $\square_2$ denote $<$.
\item[(2)] For any potential center $\sigma: S \to K$, the condition $\ord \sigma(s) \leqslant \alpha(s)$ holds for all $s\in S$.
\item[(3)] If $\sigma, \sigma': S \to K$ are potential centers, then $\ord \sigma(s) = \ord \sigma'(s)$. %and $\ac_1(\sigma(s)) = \ac_1(\sigma'(s))$ for all $s\in S$.
\item[(4)] Whenever $c \in \Sigma_s$, the set $\Sigma_s$ also contains all $ c' \in K$ such that 
\[
(\forall t) (C(s,c,t)\leftrightarrow C(s,c',t)).
\qedhere\]
\end{itemize} 
%For any $s \in S, \sigma \in K$, we write $C^{s,\sigma}$ for the set $C(s,\sigma, K)$. \textcolor{red}{Notation? I also put this in the introduction..}
%\item \textcolor{magenta}{We will also assume that the set $\Sigma$ is normalized in the sense that, . }
\end{definition}
\noindent Note that a clustered cell $X=C^\Sigma$ may also be a classical cell, provided that $\Sigma$ has a definable section. Further, remark that conditions $(1)$ and $(2)$ imply that the potential cells $C^{\sigma}$ induced by $C^{\Sigma}$ satisfy the conditions outlined in the definition of pre-admissibility.

Another remark is that, even though the above definition includes some conditions on $\Sigma$, it still leaves the structure of the set $\Sigma$ quite unspecified. Condition (4) imposes that each $\Sigma_s$ is a union of balls, but at this point we do not yet require this to be a finite union. In Section \ref{sec:separation}, the structure of this set will be discussed in more detail.

%\begin{remark} 
%Notice that in contrast with the informal definition of clustered cells we gave in Section \ref{sec:def}, no assumption is imposed on $\Sigma$ besides conditions (ii) and (iii). In particular, $\Sigma$ is not assumed to be a multi-ball. The notion of clustered cell \emph{of order $k$} described in Section \ref{sec:def} will be given in Section \ref{XX}, Definition \ref{XX}. 
%\end{remark}
\begin{remark}\label{rem:order} Let $X = C^\Sigma$  be a clustered cell and $\sigma: S \to K$ a section of $\Sigma$. The condition that 
\(
\order \sigma(s) \leqslant \alpha(s)
\)
enforces that $\ord (t-\sigma(s)) > \min\{\ord t, \ord \sigma(s)\}$, and hence that \[\order t=\order\sigma(s)\]for all $t \in C^{\sigma(s)}$. 
\end{remark}

\begin{definition}\label{def:array}  Let $\mathcal{A} = (\{C_i\}_{1\leqslant i\leqslant r}, \Sigma)$ be a multi-cell with induced set $X=\mathcal{A}(K)$. We say that $\mathcal{A}$ is a \emph{cell array} if the following additional properties hold:
\begin{itemize}
\item[(i)] For every $i =1, \ldots, r$, the set $C_i^{\Sigma^{(i)}}$ is a clustered cell.
\item[(ii)] For every section $\sigma=(\sigma_1,\ldots,\sigma_r)$ of $\Sigma$ and all $s \in S$, we have that $\ord \sigma_i(s) = \ord \sigma_j(s)$ for $1 \leqslant i \leqslant j \leqslant r$.
\item[(iii)] All centers are non-zero, i.e.  $0 \not \in \Sigma^{(i)}_s$ for any $1 \leqslant i \leqslant r$ and $s \in S$.
\item[(iv)] For every $i = 1, \ldots, r$, let $\rho_{i, \max}(s)$ denote the height of the top leaf of $C_i$. For any section $\sigma_i$ of $\Sigma^{(i)}$, any $s \in S$ and any ball $B \subseteq X_s$ such that $\sigma_i(s) \in B$, it holds that \( B \subseteq B_{\rho_{i, \max}(s) +1 }(\sigma_i(s))\). 
\qedhere\end{itemize}
\end{definition}The last condition in this definition is a slight weakening of the admissibility condition $(d)$ from Definition \ref{def:admissible} in the previous section. This condition will play an important role in our proofs in later sections. The connection between both notions will be explained further in the proof of Theorem \ref{thm:refinement}. 
\\\\
 Similar to the case of multi-cells, we will refer to both $\mathcal{A}$ and its induced set $X = \mathcal{A}(K)$ as cell arrays.

The following notation will be used for both multi-cells and arrays. 
Let $\mathcal{A}=(\{C_i\}_{i\in I},\Sigma)$ be a multi-cell over $S$ and $S_1,\ldots, S_l$ a partition of $S$. For each $1\leq j\leq l$, we define $\mathcal{A}_{|S_j}$ to be the multi-cell over $S_j$ defined by $\mathcal{A}_{|S_j}:=(\{C_i,\}_{i\in I},\Sigma_{|S_j})$, where $\Sigma_{|S_j}:=\{(s,c)\in \Sigma \ \mid \ s\in S_j\}$. It is not hard to check that each $\mathcal{A}_{|S_j}$ is still a multi-cell, and that admissibility is preserved as well. Similarly, if $\mathcal{A}$ is a cell array, then so is $\mathcal{A}_{|S_j}$. 

Note that the cell conditions of $\mathcal{A}_{|S_j}$ are the same as the ones in the original array $\mathcal{A}$, and that no new potential centers were introduced in this procedure. Moreover, the sets $\mathcal{A}_{|S_j}(K)$ form a partition of $\mathcal{A}(K)$, and if $\mathcal{A}(K) = X$, then $\mathcal{A}_{|S_j}(K) = X_{|S_j}$.
\\\\
We will now state the main theorem of this section:

\begin{theorem} \label{thm:refinement} Let $X\subseteq S\times K$ be a definable set. There exists a partition of $X$ into sets $X_1,\ldots,X_n$ such that each $X_i$ is either a classical cell or a cell array.
\end{theorem}

\subsection{Splitting multi-cells}

%\sout{
%\emph{For the remainder of this section, $X$ will always denote a definable subset of $S\times K$ and $(\{C_i\}_i, \Sigma)$ an admissible multi-cell with $\Sigma\subseteq S\times K^r$ and such that $\mathcal{A}(K)=X$. Multi-cells will be assumed to be admissible unless otherwise stated.}
%}
For the remainder of the article we will write $(\{C_i\}_{i}, \Sigma)$ as shorthand for $(\{C_i\}_{1\leqslant i\leqslant r}, \Sigma)$ whenever $r$ is clear from the context. Multi-cells will be assumed to be admissible unless otherwise stated.

%We will often try to split a multi-cell by splitting the set of its cell conditions, as explained in the next definitions.

%In order to state this properly, we will need to extend the notation $(\{C_i\}_i, \Sigma)$ to cases where the first condition of the definition of multi-cells is not satisfied. In this wider context, we will interpret 
%$(\{C_i\}_i, \Sigma)$  to induce the set
%\[ \bigcup_{\sigma \text{ section of } \Sigma} [C_1^{\sigma_1} \cup \ldots \cup C_r^{\sigma_r}].\]
 %Note that this projection procedure does not always produce a partition of $X$. 

\begin{definition}\label{def:split}
 Let $\Sigma$ be a definable subset of $S \times K^r$, and let $1 \leqslant k < r$.  Define the following coordinate projections of $\Sigma$:
\begin{align*}
 \Sigma^{(1, \ldots, k)}&:= \{ (s, c) \in S \times K^k \mid \exists \zeta_i \in K: (s, c, \zeta_{k+1}, \ldots, \zeta_r) \in \Sigma\}.\\
  \Sigma^{(k+1, \ldots, r)}&:= \{ (s, c) \in S \times K^{r-k} \mid \exists \zeta_i \in K: (s, \zeta_{1}, \ldots, \zeta_k, c) \in \Sigma\}.
\end{align*}
 Let $\mathcal{A}=(\{C_i\}_i,\Sigma)$ be a multi-cell with $\mathcal{A}(K) = X$.  If the sets
\begin{equation}\label{eq:split}
X^{(1, \ldots, k)} := \bigcup_{\substack{\sigma \text{ section} \\ \text{of } \Sigma^{(1, \ldots, k)}}} C_1^{\sigma_1} \cup \ldots \cup C_k^{\sigma_k} \text{ \quad and \quad } X^{(k+1, \ldots, r)} := \bigcup_{\substack{\sigma' \text{ section} \\ \text{of } \Sigma^{(k+1, \ldots, r)}}} C_{k+1}^{\sigma'_{k+1}} \cup \ldots \cup C_r^{\sigma'_r}
\end{equation}
%for every section $\sigma: s \mapsto (\sigma_1(s), \ldots, \sigma_k(s))$ of $\Sigma^{(1, \ldots, k)}$ and for every section $\sigma': s \mapsto (\sigma'_{k+1}(s), \ldots, \sigma'_r(s))$ of $\Sigma^{(k+1, \ldots, r)}$ the sets
%\[
%C_1^{\sigma_1} \cup \ldots \cup C_k^{\sigma_k} \text{ \quad and \quad } C_{k+1}^{\sigma'_{k+1}} \cup \ldots \cup C_r^{\sigma'_r}
%\]
are disjoint, then we say that $\mathcal{A}$ \emph{can be split at $k$ (by projection)}: if we consider the multi-cells \[\mathcal{A}^{(1,\ldots, k)} := (\{C_1, \ldots, C_k\}, \Sigma^{(1, \ldots, k)}) \text{\quad and \quad} \mathcal{A}^{(k+1, \ldots r)}:= (\{C_{k+1}, \ldots, C_r\}, \Sigma^{(k+1, \ldots, r)}),\]
then the sets $\mathcal{A}^{(1, \ldots, k)}(K)$ and $\mathcal{A}^{(k+1, \ldots, r)}(K)$ form a partition of $\mathcal{A}(K)$. 
%If this condition is satisfied, then $X$ partitions as $X^{(1, \ldots, k)} \cup X^{(k+1, \ldots, r)}$, the sets which are {\color{ForestGreen} respectively induced by the multi-cells  
%\[(\{C_1, \ldots, C_k\}, \Sigma^{(1, \ldots, k)}) \text{\quad and \quad} (\{C_{k+1}, \ldots, C_r\}, \Sigma^{(k+1, \ldots, r)}).\qedhere\]}\textcolor{red}{rephrase this? Now it feels as if we define the same set two separate times}
\end{definition}
\noindent Note that in the above definition, condition \eqref{eq:split} ensures that $\mathcal{A}^{(1, \ldots, k)}$ and $\mathcal{A}^{(k+1, \ldots, r)}$ are multi-cells. Further, remark that $\mathcal{A}^{(1, \ldots, k)}(K) = X^{(1, \ldots, k)}$ and $\mathcal{A}^{(k+1, \ldots, r)}(K) = X^{(k+1, \ldots, r)}$.  \\\\For example, if for every section $\sigma$ of  $\Sigma^{(1)}$, $C_1^{\sigma}$ defines the same set, then $\mathcal{A}$ splits at 1.

\begin{definition} \label{def:defsplit}
We say that a multi-cell $\mathcal{A} = (\{C_i\}_i, \Sigma)$ \emph{splits at $k$ by definable choice} if there exists a definable section $\sigma_k: S \to K$ of $\Sigma^{(k)}$. Then $\mathcal{A}(K)$ partitions as the union of the classical cell $C_k^{\sigma_k}$ and the multi-cell $(\{C_1, \ldots, C_{k-1}, C_{k+1}, \ldots, C_r\}, \Sigma')$, where
\[
\Sigma' = \{(s, \zeta_1, \ldots, \zeta_{k-1}, \zeta_{k+1}, \ldots, \zeta_r) \in S \times K^{r-1} \mid (s, \zeta_1, \ldots, \zeta_{k-1},  \sigma_k(s), \zeta_{k+1}, \ldots, \zeta_r) \in \Sigma\}.
\qedhere\]
\end{definition}
%\[
%\Sigma' = \{(s, \zeta_1(s), \ldots, \zeta_{i-1}(s), \zeta_{i+1}, \ldots, \zeta_r(s)) \in S \times K^{r-1} \mid (s, \zeta_1(s), \ldots, \zeta_{i-1}(s),  \sigma_i(s), \zeta_{i+1}(s), \ldots, \zeta_r(s)) \in \Sigma\}.
%\]

%Note that if a multi-cell $X$ splits at $k$, then $X^{(1, \ldots, k)}$ and $X^{(k+1, \ldots, r)}$ form a partition of $X$.
Note that both these splitting procedures preserve admissibility for multi-cells. The same procedures can also be applied to cell arrays, to obtain a partitioning in smaller cell arrays (and classical cells).  

The following lemma, originally proven by Denef for semi-algebraic sets, will be used in later proofs.

%{\color{red} An iteration of the} following lemma ensures one can always split an array into a proper array and (possibly) some clustered cells of order 1. Its proof is a straightforward verification which is left to the reader.
%
%\begin{lemma}\label{lem:strict} Let $X = (\{C_j\}_j, \Sigma)$ be a multi-cell for which the clustered cell $C_1^{\Sigma^{(1)}}$ is of order 1. Then $X = (\{C_j\}_j, \Sigma)$ splits at 1. 
%\end{lemma}
\begin{lemma}[Denef, \cite{denef-84}]\label{lem:finiteskolem} Let $(K,\cL_2)$ be a $P$-minimal structure. Let $X\subseteq S\times K^l$ be a definable set and $k$ a positive integer such that for every $s\in S$ the fiber $X_s$ has less than $k$ elements. Then there exists a definable section $g :
S \to K^l$ of $X$, that is, $g(s)\in X_s$ for all $s\in S$. 
\end{lemma}
We will now show how a multi-cell can be split into smaller parts where the cell conditions involved satisfy further properties. %
\begin{lemma} \label{lemma:refinement1}
Let $\mathcal{A}= (\{C_i\}_{1\leqslant i\leqslant r},\Sigma)$ be a multi-cell over $S$. There exists a partition of $\mathcal{A}(K)$ as $Y_1 \cup Y_2$, such that
\begin{itemize}
\item[(i)] $Y_1$ can be partitioned as a finite union of classical cells;
\item[(ii)] there exist multi-cells $\mathcal{A'} = (\{C_i'\}_{i\in I}, \Sigma')$ over definable sets $S' \subseteq S$, such that the sets $\mathcal{A'}(K)$ form a finite partition of $Y_2$, and 
\begin{enumerate}\item[(a)]  all cell conditions $C_i'$ are 1-cell conditions and have $\square_1 = \square_2 = <$; \item[(b)] for all $s \in S'$ and $i \in I$, we have that $0 \not \in (\Sigma')^{(i)}_s$.
\end{enumerate}
\end{itemize}
\end{lemma}
\begin{proof}
Let $X:=\mathcal{A}(K)$. We will prove the lemma by sequentially partitioning off parts of $X$. %\sout{, until we are left with a set that satisfies the conditions of $(ii)$}.
We begin by isolating those cell fibers for which $0$ is a potential center. Consider the following inductive procedure. First, put 
\[S_{0}:= \{s \in S \mid 0 \in \Sigma_s^{(1)} \},\]
and $S_{1}:= S \backslash S_{0}$. This induces a partition of $X$ with respect to the multi-cells $\mathcal{A}_{|S_{k}}$ for $k=0,1$.

Now, $\mathcal{A}_{|S_{0}}$ admits a split at 1 by definable choice, using the constant function $\sigma_1: S_0 \to K: s \mapsto 0$. Write $\mathcal{A}_{|S_0}'=(\{C_i\}_{2\leqslant i\leqslant r},\Sigma')$ for the multi-cell that remains after the split. The multi-cell $\mathcal{A}_{|S_{1}}$ already has the property that $0 \not \in (\Sigma_{|S_{1}}^{(1)})_s$ for any $s\in S_{1}$. Repeating a similar procedure for all components of $\mathcal{A}_{|S_0}'$ and $\mathcal{A}_{|S_1}$ will yield a finite number of classical cells, and a finite number of multi-cells for which $0$ is not in any of the sets $\Sigma^{(i)}_s$. Hence, we may as well assume from now on that $\mathcal{A}$ itself is a multi-cell satisfying this property.
\\\\
As a next step, we will consider the $0$-cell conditions. Without loss of generality, we may assume that there exists a $k\in \{1,\ldots,r\}$ such that all cell conditions $C_i$ with $1\leq i\leq k$ are 0-cell conditions and all cell conditions $C_i$ with $i>k$ are 1-cell conditions. We need to show that $X$ splits at $k$ (by projection), i.e. that $X_1:= X^{(1 \ldots, k)}$ and $X^{(k+1, \ldots, r)}$ are disjoint sets.
Recall that $(\{C_i\}_i,\Sigma)$ is assumed to be an admissible multi-cell. Now part $(a)$ of Definition \ref{def:admissible} implies the following. If $(s,t)\in X_1 \cap X^{(k+1, \ldots, r)}$, then $t\in X_s\setminus \Int(X_s)$ since $(s,t)\in X_1$. However, $t\in \Int(X_s)$ since $(s,t)\in X^{(k+1, \ldots, r)}$, which is a contradiction. Using Lemma \ref{lem:finiteskolem}, the set $X_1$ can be partitioned into a finite number of classical cells. 
\\\\
 For the next part we work with $X\backslash X_1$ (which we will still call $X$, since we may as well assume that $X_1$ is empty). After reordering if necessary, there exists $k\in \{0,1,\ldots,r\}$ such that all cell conditions $C_i$ with $1\leqslant i\leqslant k$ are precisely those cell conditions for which $\square_1=\emptyset$. Note that part $(c)$ of Definition \ref{def:admissible} implies that $\Sigma^{(1, \ldots, k)} = S\times \{(0,\ldots,0)\}$, which actually implies that $k =0$, since we had assumed that all potential centers for $X$ were non-zero. %Hence, the set $X$ already satisfies part $(\textcolor{blue}{b})$ of $(ii)$. \textcolor{red}{REMARK: I would leave out this last sentence because we haven't shown all of part (a) of (ii) yet. \textcolor{blue}{Reply: I think the (a) should have been (b)..}} 
 
 After reordering if necessary, we can find $k\in \{1,\ldots,r\}$ such that all cell conditions $C_i$ with $1\leqslant i\leqslant k$ are precisely those cell conditions for which $\square_2=\emptyset$. Let $\sigma=(\sigma_1,\ldots,\sigma_r)$ and $\theta=(\theta_1,\ldots,\theta_r)$ be two sections of $\Sigma$. 

First note that for any $1 \leqslant j \leqslant k$, we have that  $\theta_j(s) \not \in X_s$. Indeed, suppose for a contradiction that $\theta_j(s) \in X_s$. Because the multi-cell for $X$ does not contain any 0-cell conditions, $X_s$ can be written as a finite disjoint union of open cell fibers $C_i^{\theta_i(s)}$. Note that $\theta_j(s) \in \text{Cl}\left(C_j^{\theta_j(s)}\right) \backslash C_j^{\theta_j(s)}$, and hence there must be some $i \neq j$ such that $\theta_j(s) \in C_{i}^{\theta_{i}(s)}$. Since this cell fiber $C_{i}^{\theta_{i}(s)}$ is open, it must contain a ball $B_\gamma(\theta_j(s))$. But this implies that $C_{i}^{\theta_{i}(s)} \cap C_{j}^{\theta_{j}(s)} \neq \emptyset$, which is a contradiction, so we conclude that $\theta_j(s) \not \in X_s$.

We will show that for every $s$, the sets $\{\theta_1(s),\ldots,\theta_k(s)\}$ and $\{\sigma_1(s),\ldots,\sigma_k(s)\}$ contain the same elements.
%$(\theta_1(s),\ldots,\theta_k(s))$ is a permutation of $(\sigma_1(s),\ldots,\sigma_k(s))$.
If this were not the case, there would exist $s\in S$ and $1\leqslant j\leqslant k$ such that $\theta_j(s)\neq \sigma_i(s)$ for all $1\leqslant i\leqslant k$. 
%We claim that $\theta_j(s) \not \in X_s$. Indeed, suppose for a contradiction that $\theta_j(s) \in X_s$. Because the multi-cell for $X$ does not contain any 0-cell conditions, $X_s$ can be written as a finite disjoint union of open cell fibers $C_i^{\theta_i(s)}$. Note that $\theta_j(s) \in \text{Cl}\left(C_j^{\theta_j(s)}\right) \backslash C_j^{\theta_j(s)}$, and hence there must be some $i \neq j$ such that $\theta_j(s) \in C_{i}^{\theta_{i}(s)}$. Since this cell fiber $C_{i}^{\theta_{i}(s)}$ is open, it must contain a ball $B_\gamma(\theta_j(s))$. But this implies that $C_{i}^{\theta_{i}(s)} \cap C_{j}^{\theta_{j}(s)} \neq \emptyset$, which is a contradiction, so we conclude that $\theta_j(s) \not \in X_s$.

Since $\theta_j(s) \in \text{Cl}(X_s) \backslash X_s$, the set $X_s$ contains elements $t\in K$ which are arbitrarily close to $\theta_j(s)$. But since $\theta_j(s)  \neq \sigma_i(s)$ for all $1\leq i\leq k$, such a $t$ cannot belong to  $\bigcup_{i=1}^k C_i^{\sigma_i(s)}$. Hence, for any such element $t$, there must exist some $i_0 >k$ such that $t \in C_{i_0}^{\sigma_{i_0}(s)}$. But since $t$ is arbitrarily close to $\theta_j(s)$, and the cell condition $C_{i_0}$ has $\square_2=<$, this implies that $\theta_j(s)\in C^{\sigma_{i_0}(s)}$, which is a contradiction. 

We have now shown that for $1 \leqslant i \leqslant k$, the sets $\Sigma_s^{(i)}$ contain at most $k$ elements. By Lemma \ref{lem:finiteskolem}, there is a definable way to choose an element from these sets uniformly in $s$. In particular, there exists a function $\sigma_1: S \to K$ such that $X$ splits by definable choice as $C_1^{\sigma_1}$ and $(\{C_2, \ldots, C_r\}, \Sigma')$, where $\Sigma'$ is as in Definition \ref{def:defsplit}. % \[\Sigma' = \{(s, \zeta_2(s), \ldots, \zeta_r(s)) \in S \times K^{r-1} \mid (s, \sigma_1(s), \zeta_2(s), \ldots ) \in \Sigma\}.\]
%\textcolor{red}{REMARK: We could leave out the description of $\Sigma'$, since it already appears in the definition of splitting by definable choice.}
Applying this procedure $k$ times shows that we can split off $k$ classical cells and be left with a multi-cell satisfying the conditions of $(ii)$. 
%\textcolor{red}{ALSO ADAPT PROOF OF THM 4.7}
%This implies both that $(\{C_i\}_i,\Sigma)$ splits at $k$ and moreover that for all $s\in \Sigma_3:= \Sigma^{(1, \ldots, k)}$, the fiber $(\Sigma_3)_s$ is finite. 
%$(iv)$ $X_4$ is precisely the part of $X$ that remains after we have split off $X_1, X_2$ and $X_3$.
\end{proof}

In the next lemma, we will show that one can definably fix the order of the potential centers for every component: 
\begin{lemma} \label{lemma:refinement2}
Let $\mathcal{A} = (\{C_i\}_i, \Sigma)$ be a multi-cell satisfying the conditions in part (ii) of Lemma \ref{lemma:refinement1}.
There exists a multicell $\mathcal{A}' = (\{C_i\}_i, \Sigma')$ with 
 $\Sigma'\subseteq \Sigma$,  such that  
 \begin{itemize}
 \item[(i)] $\mathcal{A}(K)=\mathcal{A}'(K)$; 
 \item[(ii)]  %the following properties hold 
for all $s \in S$, all
$\sigma(s) = (\sigma_1(s), \ldots, \sigma_r(s)),\ \theta(s) = (\theta_1(s), \ldots, \theta_r(s)) \in \Sigma'_s$, and all $1 \leqslant j \leqslant r$, it holds that 
\[\order \sigma_j(s)=\order\theta_j(s).\]
\end{itemize}

\end{lemma}

\begin{proof}
Use induction to define a chain of sets $\Sigma_l\subseteq S\times K^r$ for $0 \leqslant l\leqslant r$, with $\Sigma_0:= \Sigma$. Write $(s,\sigma) = (s, \sigma_1(s), \ldots, \sigma_r(s))$ for elements of $\Sigma$. Assuming $\Sigma_{l-1}$ has been defined, set
\[
\Sigma_{l}:=\{(s,\sigma)\in \Sigma_{l-1} \ \mid \ \forall (s,\sigma') \in \Sigma_{l-1}: \ord\, \sigma'_l(s)\leqslant \ord \,\sigma_l(s)\}.
\]

%Use induction to define a chain of sets $\Sigma_n\subseteq S\times K^n$ for $0 \leq n\leq r$, with $\Sigma_0:= S$. Assuming $\Sigma_{n-1}$ has been defined, set 
%\[
%\Sigma_{n}=\{(s,\sigma_1,\ldots,\sigma_n)\in \Sigma_{n-1}\times K: (s,\sigma_{n})\in \Sigma^{(n)}\wedge \forall (s,\theta)\in \Sigma^{(n)}:\ord\,\sigma_n\geqslant \ord\,\theta\}.
%\]
Note that this is well-defined, as by condition (b) of pre-admissibility, $\alpha_l(s)$ is an upper bound for $\ord(\sigma_l(s))$, since $\sigma_l(s) \neq 0$ for the multi-cells we consider in this lemma. Moreover, by Lemma 2 and Theorem 6 of  \cite{clu-presb03}, $P$-minimal definable subsets of $\Gamma_K$ are Presburger-definable, and every such set has a maximal element if it is bounded. 

We leave it to the reader to check that for each $l$, $\mathcal{A}_l:=(\{C_i\}_i,\Sigma_l)$ is indeed a multi-cell. Also, for each $l$, $\mathcal{A}_l(K)=\mathcal{A}(K)$ since the only thing we do in every step is to put restrictions on which centers we allow for each of the components: $\Sigma_1$ will fix the order of $\sigma_1(s)$, then $\Sigma_2$ will pick a subset from $\Sigma_1$ where $\ord\, (\sigma_2(s))$ is fixed, and so on. Note that at no point in the induction, $\Sigma_l$ will be empty. Setting $\mathcal{A}':=\mathcal{A}_r$ completes the proof. 
\end{proof}

\begin{lemma}\label{lemma:refinement3}
Let $\mathcal{A} = (\{C_i\}_i, \Sigma)$ be a multi-cell as obtained in Lemma \ref{lemma:refinement2} with $X=\mathcal{A}(K)$. There exists a finite partitioning of $X$ into sets $X_j \subseteq S_j \times K$ (where the $S_j$ are definable subsets of $S$), such that each part $X_j$ can be written as a finite disjoint union of multi-cells $\mathcal{A}_{jk} = (\{C_{jk,i}\}_i, \Sigma_{jk})$ over $S_j$, and
\[\order \sigma_1(s) =  \ldots = \order \sigma_{r_{jk}}(s)\] for all $(s, \sigma_1(s), \ldots, \sigma_{r_{jk}}(s)) \in \Sigma_{jk}.$
\end{lemma}
\begin{proof}
Assume that the refinements of Lemma \ref{lemma:refinement2} have been applied. 
Let $\mathrm{Perm}$ be the set consisting of all tuples $\Delta = (\triangle_k)_k$ of length ${r\choose 2}$, where each $\triangle_k$ is an element of the set $\{<,>,=\}$, and $k \in \{(k_1, k_2) \mid 1 \leqslant k_1 < k_2<r\}$. Now partition $S$ into sets 
\[
S_{\Delta}:= \{s \in S \mid \forall (s, \sigma) \in \Sigma: \ord \sigma_{k_1} \ \triangle_k \ \ord \sigma_{k_2}\}. 
\]
Since $\mathrm{Perm}$ is a finite set, this gives us a finite partitioning of $S$, which in turn induces a partitioning of $X$ into multi-cells 
$(\{C_{\delta(i)}\}_i, \Sigma_{\Delta})$. Here $\delta$ is a permutation of $\{1, \ldots, r\}$ and $\Sigma_{\Delta}$ is obtained from $\Sigma \subseteq S \times K^r$ by restricting $S$ to $S_{\Delta}$, and reordering the components, such that they are ordered by valuation. That is, for each multi-cell there is a tuple $(\square_k)_{k<r}$ where each $\square_k$ is either $<$ or $=$ such that, for every section $\sigma$ of $\Sigma_{\Delta}$, 
\[
\ord \sigma_k(s) \ \square_k \ \ord \sigma_{k+1}(s), \text{ for all $s \in S_{\Delta}$ and all $1 \leqslant k < r$}.
\] 
We will now focus on one such multi-cell over a set $S_{\Delta}$ (which we will denote again as $(\{C_{i}\}_i, \Sigma)$ for simplicity), and show how it can be split by projection to obtain the lemma.
Let $k\in\{1,\ldots,r-1\}$ be such that for all $(s, \sigma_1, \ldots,\sigma_{r})\in \Sigma$, we have that 
\[ 
\ord \sigma_1 = \ldots = \ord \sigma_{k} < \ord\sigma_{k+1}.
\]
If no such $k$ exists we are done. Otherwise, it suffices to show that $(\{C_i\}_i,\Sigma)$ splits at $k$. For if it does, $X^{(1, \ldots, k)}$ is a multi-cell satisfying the condition stated in the lemma, and we can iterate the process for $X^{(k+1, \ldots, r)}$. This process must stop because we are decreasing the ambient dimension of $\Sigma$ (indeed, $\Sigma^{(k+1, \ldots, r)} \subseteq S\times K^{r-k}$). 

Let us now show that one can indeed split $X$ at $k$: if $(s,t)\in X^{(1, \ldots, k)}\cap X^{(k+1, \ldots, r)}$, there are $(s,\sigma_1,\ldots,\sigma_k)\in \Sigma^{(1, \ldots, k)}$, $(s,\theta_1,\ldots,\theta_{r-k})\in \Sigma^{(k+1,\ldots,r)}$ and some $1\leqslant j \leqslant r-k$ such that by Remark \ref{rem:order}, 
\[
\ord(t)=\ord(\sigma_1)<\ord(\theta_j)=\ord(t),
\] which is a contradiction.  
\end{proof}

We have now done all the preparatory work to prove Theorem \ref{thm:refinement}: 

\begin{proof}[Proof of Theorem \ref{thm:refinement}]
By Theorem \ref{thm:goopartition}, we may suppose that $X$ is an admissible multi-cell. Using Lemmas \ref{lemma:refinement1}, \ref{lemma:refinement2} and \ref{lemma:refinement3}, $X$ can be partitioned
as a finite union of classical cells and multi-cells $(\{C_i\}_i,\Sigma)$ satisfying conditions (ii) and (iii) of Definition \ref{def:array}. Moreover, each $C_i^{\Sigma^{(i)}}$ satisfies condition (1)-(3) of Definition \ref{def:clustered_cell}. 

All operations used in the previous lemmas preserve admissibility, so it can assumed that each multi-cell $(\{C_i\}_i,\Sigma)$ is admissible. 
Without loss of generality, we may suppose that $X$ is defined by one such multi-cell $(\{C_i\}_{1\leq i\leq r}\,\Sigma)$. 

To ensure condition (i) from Definition \ref{def:array}, it remains to show that each $C_i^{\Sigma^{(i)}}$ satisfies condition (4) of Definition \ref{def:clustered_cell}. To obtain this condition, it may be that we have to add extra elements to $\Sigma$. Consider the set $\Sigma'$ defined by
\[
\Sigma':=\{(s,x_1,\ldots,x_r)\in S\times K^r \mid \bigwedge_{i=1}^r (\exists c)[(s,c)\in \Sigma^{(i)} \wedge x_i\in B_{\rho_{i, \text{max}}(s) +m}(c)] \}.
\]
The set $\Sigma'$ is obtained from the original set $\Sigma$ by adding, for every for every $c \in \Sigma^{(i)}_s$, all elements in the ball $B_{\rho_{i, \text{max}}(s) +m}(c)$. This ensures that each $C_{i}^{{\Sigma'}^{(i)}}$ now satisfies condition (4) of definition \ref{def:clustered_cell}. It is easy to check that $(\{C_i\}_i,\Sigma')$ still defines the same set $X$, and still satisfies conditions (i)-(iii) of Definition \ref{def:array}. 

Before we can discuss condition (iv) of Definition \ref{def:array}, we need to introduce the following notion. Let $\sigma_i$ be a potential center contained in $\Sigma^{(i)}$. We say that $\sigma_i(s)$ is an \emph{admissible center} (for some $s \in S$) if it does not violate condition (d) of the definition of admissibility (Definition \ref{def:admissible}). More precisely, we mean the following. Let $B$ be the maximal ball in $X_s$ that contains $\sigma_i(s)$. Then $\sigma_i(s)$ is an admissible center if, for any section $\sigma$ of $\Sigma$ that has $\sigma_i$ as a component, the ball $B$ is contained within a single cell of the decomposition of $X_s$ induced by $\sigma(s)$.

When replacing the original set $\Sigma$ by $\Sigma'$, we may have added centers which are not admissible (the reader can check that the conditions of pre-admissibility will never be violated). Yet, note that by construction, any ball in ${\Sigma'}^{(i)}_s$ of size $\rho_{i, \text{max}}(s) + m$ still contains at least one admissible center.  
%However, it may be that this new multi-cell is no longer admissible as we may have added centers to $\Sigma'$ which are not admissible (we will say that a center $\sigma_i \in \Sigma^{(i)}$ is admissible, if, for any section $\sigma$ of $\Sigma$ that has $\sigma_i$ as a component, and every $s \in S$, the decomposition of $X_s$ induced by $\sigma(s)$ is admissible). 

Let us now show that this implies condition (iv) from Definiton \ref{def:array}. 
 %Moreover, conditions (i) and (iii) imply that for the resulting multi-cells,  the admissibility condition (d) simplifies to the following: if $\sigma = (\sigma_1, \ldots, \sigma_r)$ is a section of $\Sigma$, then, if  $X_s$ contains a ball $B$ around some $\sigma_i(s)$, that ball $B $ must be contained within a single cell fiber $C_j^{\sigma_j(s)}$ with $i \neq j$. 
%We now claim that this implies that condition (iv) from the definition of cell array must hold as well.
%Let $X= (\{C_i\}_i,\Sigma)$ be a multi-cell satisfying conditions $(i)-(iii)$ from the definition of cell array, along with the above version of admissibility condition $(d)$. We will now explain why condition $(iv)$ must be satisfied as well. 
Without loss of generality, we can take $i = 1$. Consider all possible sections of $\Sigma'$ which are of the form $(\sigma_1(s), \zeta_{2}(s), \ldots, \zeta_r(s))$. Each such section induces a partition
\[X_s = C_1^{\sigma_1(s)} \cup C_2^{\zeta_2(s)} \cup \ldots \cup C_r^{\zeta_r(s)}.\]
Now consider the maximal ball $B$ around $\sigma_1(s)$. We need to distinguish between two cases. It may be that this ball does not contain any admissible centers. However, in that case the ball must have a radius strictly bigger than $\rho_{i, \text{max}}(s) + m$, in which case condition (iv) holds. If the ball does contain an admissible center, we may as well assume that $\sigma_1(s)$ itself is admissible. Hence, there should be a single cell in the decomposition that contains the maximal ball $B$ around $\sigma_1(s)$. This has to be one of the cells $C_i^{\zeta_i(s)}$ (since $\sigma_1(s) \not \in C^{\sigma_1(s)}$). 

Let us assume that $B \subset C_2^{\zeta_2(s)}$. Note that, if the ball $B$ would be strictly bigger than the ball \(B_{\rho_{1, \max}(s) +1 }(\sigma_1(s))\), then the cells $C_1^{\sigma_1(s)}$ and $C_2^{\zeta_2(s)}$ would have non-empty intersection, which is a contradiction.
\end{proof}

% !TEX root = centers_main.tex

\section{On the structure of the trees of potential centers } \label{sec:separation}

%Given an array $(\{C_i\}_i,\Sigma)$ with $\Sigma\subseteq S\times K^r$ and $i\in\{1,\ldots, r\}$, consider its associated clustered cell $C_i^{\Sigma^{(i)}}$. The main goal of this section is to give a uniform bound for all $1\leq i\leq r$ and all $s\in S$ of the number of different cell fibers $C_i^{\sigma_i(s)}$ induced by the clustered cell $C_i^{\Sigma^{(i)}}$ (here $\sigma=(\sigma_1,\ldots,\sigma_r)$ is some section of $\Sigma$). Let us formalize this idea. 
%
%\begin{definition}\label{def:equivalence} Let $C^\Sigma$ be a clustered cell. For $s\in S$, elements $\sigma,\sigma'\in \Sigma_s$ are said to be \emph{$(C,\Sigma_s)$-equivalent} if they define the same cell fiber over $s$, that is, if 
%\[
%(\forall t) (C(s,\sigma,t)\leftrightarrow C(s,\sigma',t)).
%\]
%Given sections $\sigma,\sigma':S \to K$ of $\Sigma$, $\sigma$ and $\sigma'$ are \emph{$(C,\Sigma_s)$-equivalent} if $\sigma(s)$ and $\sigma'(s)$ are \emph{$(C,\Sigma_s)$-equivalent}, that is, if $C^{\sigma(s)}=C^{\sigma'(s)}$. 
%\end{definition}
%We will sometimes write \emph{equivalent} rather than $(C,\Sigma_s)$-equivalent, when the meaning is clear from the context.
%\\\\
%Let $(\{C_i\}_i,\Sigma)$ be an array with $\Sigma\subseteq S\times K^r$. The main proposition of this section can now be stated as follows.  

Let $C^{\Sigma}$ be a clustered cell. As we have observed before, there may exist different sections $\sigma, \sigma'$ of $\Sigma$ such that the potential cells $C^{\sigma}$ and $C^{\sigma'}$ do not define the same set. To formalize this observation, let us introduce the following equivalence relation.

\begin{definition}\label{def:equivalence} Let $C^\Sigma$ be a clustered cell.  For $s\in S$, elements $c,c'\in \Sigma_s$ are said to be \emph{$(C,\Sigma_s)$-equivalent} if they define the same cell fiber over $s$, that is, if 
\[
(\forall t) (C(s,c,t)\leftrightarrow C(s,c',t)).
\]
Given sections $\sigma,\sigma':S \to K$ of $\Sigma$, $\sigma$ and $\sigma'$ are \emph{$(C,\Sigma_s)$-equivalent} if $\sigma(s)$ and $\sigma'(s)$ are \emph{$(C,\Sigma_s)$-equivalent}, that is, if $C^{\sigma(s)}=C^{\sigma'(s)}$. 
\end{definition}
\noindent We will sometimes write \emph{equivalent} rather than $(C,\Sigma_s)$-equivalent, when the meaning is clear from the context.
\\\\
%Given a cell array $(\{C_i\}_i,\Sigma)$ with $\Sigma\subseteq S\times K^r$, one can consider the associated clustered cells $C_i^{\Sigma^{(i)}}$.
The main goal of this section is to prove the following proposition. 

\begin{proposition}\label{prop:finitebranching} Let $(\{C_i\}_i, \Sigma)$ be a cell array. There exists a uniform bound $N \in \N$, such that for all $s \in S$ and all $1\leq i\leq r$, the number of $(C_i,\Sigma^{(i)}_s)$-equivalence classes is at most $N$. 
\end{proposition}
\noindent The proof of Proposition \ref{prop:finitebranching} will rely on the combinatorial structure of the set $\Sigma$. Let us first introduce some notions which will be used in the proof.
\\\\
We start by noting that, given a clustered cell $C^\Sigma$, a section $\sigma$ of $\Sigma$ and $s\in S$, the $(C,\Sigma_s)$-equivalence class of $\sigma$ corresponds to the ball of radius $\rho_{\text{max}}(s)+m$ centered at $\sigma(s)$ (recall that $\rho_{\text{max}}$ and $m$ only depend on the cell condition $C$). This follows from the definition of clustered cell (condition (4) of Definition \ref{def:clustered_cell}). If no confusion arises, we will use the abbreviated notation $B(\sigma(s))$ for such balls of equivalent centers, i.e. 
\[
B(\sigma(s)):=B_{\rho_{\text{max}}(s)+m}(\sigma(s)).				
\]  
\begin{minipage}{0.55 \textwidth}
The picture on the right further illustrates this concept. Here we have drawn the leaves of the cell fiber $C^{\sigma_3(s)}$, and the leaves for the fibers $C^{\sigma_1(s)}$ and $C^{\sigma_2(s)}$ could be depicted similarly. 
\\\\
Note that the cell fibers $C^{\sigma_2(s)}$ and $C^{\sigma_3(s)}$ are disjoint, whereas $C^{\sigma_1(s)}$ and $C^{\sigma_2(s)}$ are not. To study possible intersection between potential cell fibers, it will be important to consider \emph{branching heights} ($\gamma_1(s)$ and $\gamma_2(s)$) in the picture), as they determine whether an intersection could possibly be nonempty. 
 \end{minipage}
 \begin{minipage}{0.05 \textwidth}\end{minipage}
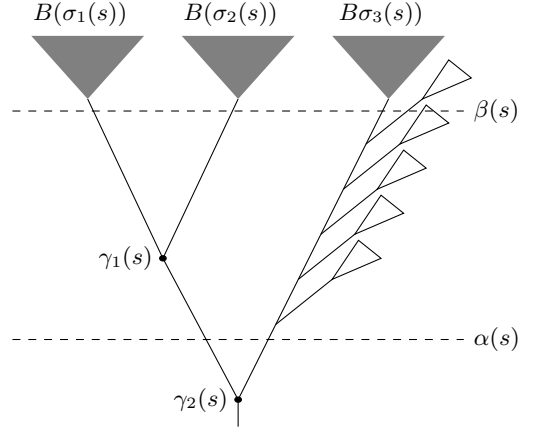
 \begin{minipage}{0.4 \textwidth}
\begin{tikzpicture}[yscale=0.8]
\path [fill = gray] (2.25,5.3) -- (3,4.25) -- (3.75,5.3) -- cycle;
\draw (2,1.6) -- (3,4.25);
\path[fill = gray] (0.25,5.3) -- (1,4.25) -- (1.75,5.3) -- cycle;
\draw (2,1.6) -- (1,4.25);
\path[fill = gray] (4.25,5.3) -- (5,4.25) -- (5.75,5.3) -- cycle;
\draw (3,-0.75) -- (5,4.25);
\draw (2,1.6) -- (3,-0.75);
\draw (3,-0.75) -- (3,-1.2);

\draw (4.7,3.5) -- (5.45,4.25);
\draw (4.4,2.75) -- (5.15, 3.5);
\draw (4.1,2) -- (4.85,2.75);
\draw (3.8,1.25) -- (4.55,2);
\draw (3.5,0.5) -- (4.25,1.25);
\draw (5.45,4.25) -- (6.1,4.6) -- (5.8,4.9) -- cycle;
\draw (5.15,3.5) -- (5.8,3.85) -- (5.5,4.15) -- cycle;
\draw (4.85,2.75) -- (5.5,3.1) -- (5.2,3.4) -- cycle;
\draw (4.55,2) -- (5.2,2.35) -- (4.9,2.65) -- cycle;
\draw (4.25,1.25) -- (4.9,1.6) -- (4.6,1.9) -- cycle;
\draw [dashed] (0,4.05) -- (6,4.05);
\draw [dashed] (0,0.25) -- (6,0.25);
\node[{right}] at (6,4.05){\footnotesize$\beta(s)$};
\node[{right}] at (6,0.25){\footnotesize$\alpha(s)$};
\node[{above}] at (0.9,5.3){\footnotesize$B(\sigma_1(s))$};
\node[{above}] at (2.9,5.3){\footnotesize$B(\sigma_2(s))$};
\node[{above}] at (4.9,5.3){\footnotesize$B\sigma_3(s))$};
\draw [fill] (3,-0.75) circle [radius = 1.25pt];
\draw [fill] (2,1.6) circle [radius = 1.25pt];
\node[{left}] at (3,-0.75){\footnotesize$\gamma_2(s)$};
\node[{left}] at (2,1.6){\footnotesize$\gamma_1(s)$};
%\draw [decorate,decoration={brace,amplitude=4pt,mirror,raise=3pt},yshift=0pt] (3,0.5) -- (3,1.25) node [black,midway,xshift=12pt] {$n$};
%\draw [decorate,decoration={brace,amplitude=4pt,mirror,raise=3pt},yshift=0pt] (3,3.5) -- (3,4.25) node [black,midway,xshift=12pt,yshift=-2pt] {$m$};
%\draw [dashed] (3.45,5) -- (3,4.25);
%\node[{above}] at (3.65,5.03){$\sigma'$};
%\draw [dashed] (1.5,3.5) -- (4,3.5);
%\node[{right}] at (4,3.46){$\rho_\text{max}(s)$};
\end{tikzpicture}

\end{minipage}

\begin{definition} Let $C^{\Sigma}$ be a clustered cell. For $s\in S$, we call $\gamma\in \Gamma_K$ a \emph{branching height} of $\Sigma_s$, if there exist sections $\sigma, \sigma'$ of $\Sigma$ which are \emph{not} $(C,\Sigma_s)$-equivalent, and for which $\ord(\sigma(s)-\sigma'(s)) = \gamma$.   
\end{definition}

Let $\mathbb{B}$ denote the set of balls of $K$, that is
 \[\mathbb{B}:= \{ B_{\gamma}(a) \mid a \in K, \gamma \in \Gamma_K\cup\{\infty \}\} .\]
The set $\mathbb{B}$, equipped with the reversed inclusion relation $\supseteq$, forms a meet semi-lattice tree. The meet of two balls $B_1$ and $B_2$, denoted by $\inf(B_1,B_2)$, corresponds to the smallest ball $B\in \mathbb{B}$ containing both $B_1$ and $B_2$. This structure is interpretable in $K$. % and, abusing terminology, by definable subsets of $\mathbb{B}$ we mean interpretable subsets. 
Note that $K$ can be identified with the set of maximal elements of $\mathbb{B}$: elements of $K$ are in definable bijection with balls of radius $\infty$ in $\mathbb{B}$, which are maximal balls with respect to reverse inclusion.

Let $C^\Sigma$ be a clustered cell. To each $\Sigma_s$  we associate a subtree $T(\Sigma_s)$ of $\mathbb{B}$ (the set of all balls) generated by the $(C,\Sigma_s)$-equivalence classes, i.e. 
\[
T(\Sigma_s):=\{B\in \mathbb{B}\mid B=\inf(B(\sigma(s)), B(\sigma'(s))),\ \text{where } \sigma,\sigma' \text{ are sections of } \Sigma\}.
\] 
Let $Y\subseteq S\times \Gamma_K$ be such that for each $s\in S$, $Y_s$ denotes the set of all branching heights of $\Sigma_s$. Each set $Y_s$ is bounded above by $\beta(s)+m$ and is uniformly definable in $s$. For each non-zero $l\in\N$, we can inductively define a function $\gamma_l:S\to\Gamma_K\cup\{-\infty\}$ as follows: let $\gamma_1(s)$ denote the biggest element of $Y_s$ and put
\[
\gamma_{l+1}(s)=
\begin{cases}
\sup(Y_s\setminus\{\gamma_1(s),\ldots,\gamma_l(s)\}) & \text{ if } Y_s\setminus\{\gamma_1(s),\ldots,\gamma_l(s)\}\neq\emptyset,\\
-\infty & \text{ otherwise}.
\end{cases}
\] 
Both $Y_s$ and the functions $\gamma_l$ depend on the ambient clustered cell $C^\Sigma$ we are working in. 

\

Let $\gamma\in Y_s$ be a branching height, and $\sigma$ a section of $\Sigma$ such that $B_{\gamma}(\sigma(s))$ is a node of $T(\Sigma_s)$. By the \emph{successors} of $B_\gamma(\sigma(s))$ in $T(\Sigma_s)$, we will mean those balls $B\in T(\Sigma_s)$ with $B\subsetneq B_\gamma(\sigma(s))$, for which there does not exist any ball $B'\in T(\Sigma_s)$ with $B\subsetneq B'\subsetneq B_\gamma(\sigma(s))$. If $B_{\gamma}(\sigma(s))$ is a node of $T(\Sigma_s)$, then the number of successors of $B_{\gamma}(\sigma(s))$ must be an integer $k$ between 2 and $q_K$. We use the first order formula $\phi_k(\sigma(s), \gamma)$ to express that $B_{\gamma}(\sigma(s))$ has exactly $k$ successors: 
\[\small
\phi_k(\sigma(s), \gamma) := (\exists c_1,\ldots, c_{k}\in \Sigma_s)(\forall \zeta\in \Sigma_s)\left(
\begin{array}{l}
\sigma(s)=c_1 \wedge \bigwedge_{i\neq j} \ord(c_i -c_j)=\gamma \ \wedge\\
\bigwedge_{i\neq j} [c_i \text{\ and\ } c_j \text{ are not $(C, \Sigma_s)$-equivalent}] \ \wedge \\
\left[\ord(\zeta-c_1)=\gamma \to \bigvee_{i\neq 1} \ord(\zeta-c_i)>\gamma\right]
\end{array}\right).
\]
\normalsize %expressing $B_{\gamma}(\sigma(s))$ has exactly $k$ successors in $T(\Sigma_s)$. 
One should be aware that for some $\gamma\in Y_s$ and some sections $\sigma$ of $\Sigma$, the ball $B_{\gamma}(\sigma(s))$ may not necessarily be a node of $T(\Sigma_s)$. We express this situation by the following first-order formula $\phi_1(\sigma(s), \gamma)$: 
\[
\phi_1(\sigma(s), \gamma) := \sigma(s)\in\Sigma_s\wedge  (\forall \zeta\in \Sigma_s)(\ord(\sigma(s) -\zeta)\neq\gamma).
\]
%\[
%\phi_1(\sigma, \gamma) := \sigma\in\Sigma_s\wedge (\exists \sigma_1, \sigma_{2}\in \Sigma_s)(
%\ord(\sigma_1 -\sigma_2)=\gamma) \wedge (\forall \zeta\in \Sigma_s)(\ord(\sigma -\zeta)\neq\gamma).
%\]
The previous discussion implies that given any $\gamma\in Y_s$ %\textcolor{magenta}{(or even any $\gamma \in \Gamma_K$)},
 and any section $\sigma$ of $\Sigma$, there exists a unique $k\in\{1,\ldots, q_k\}$ such that $\phi_{k}(\sigma(s),\gamma)$ holds. 
\begin{definition}\label{def-signature}  Let $d\in \NN\setminus\{0\}$, let  $C^\Sigma$ be a clustered cell and $\sigma$ a section of $\Sigma$. For $s\in S$, the \emph{$d$-signature} of $\sigma(s)$ is the tuple $(k_1, \ldots, k_d) \in \{1, \ldots, q_K,-\infty\}^d$ where for $i\in\{1,\ldots,d\}$
\[
k_i=
\begin{cases}
k & \text{ if } \gamma_i(s)\neq-\infty \text{ and } \phi_{k}(\sigma(s),\gamma_i(s)) \text{ holds},\\
-\infty & \text{ if } \gamma_i(s)=-\infty.
\end{cases}
\qedhere\]
\end{definition} 
\noindent Hence, if some $k_i>1$ then $B_{\gamma_i(s)}(\sigma(s))$ is a node of $T(\Sigma_s)$ with $k_i$ successors. On the other hand, if $k_i=1$ then the ball $B_{\gamma_i(s)}(\sigma(s))$ is not a node of the tree $T(\Sigma_s)$.
\\\\
% and hence $T(\Sigma_s)$ is a tree of depth at least $d$. 
The $d$-signature $(k_1,\ldots, k_d)$ of $\sigma(s)$ also encodes information about the number of branching heights: if $k_i\neq -\infty$ for all $1\leq i\leq d$, then $\Sigma_s$ has at least $d$ branching heights. \vspace{5pt}
\\
\noindent\begin{minipage}{0.65 \textwidth}
%\begin{enumerate}
%%\item[$\bullet$] if $k_i=1$, the ball $B_{\gamma_i(s)}(\sigma(s))$ is not a node of the tree $T(\Sigma_s)$; 
%%\item [$\bullet$]if $k_i>1$, $B_{\gamma_i(s)}(\sigma(s))$ is a node of $T(\Sigma_s)$ having $k_i$ successors.
%\item[$\bullet$] the least $i_0$ for which the $k_{i_0}=-\infty$ (if it exists) says $\Sigma_s$ has $i_0-1$ branching heights, so $T(\Sigma_s)$ is a finite tree of depth $i_0 - 1$;
%\item[$\bullet$] if $k_i\neq -\infty$ for all $1\leq i\leq d$, then $\Sigma_s$ has at least $d$ branching heights, so $T(\Sigma_s)$ is a tree of depth at least $d$; 
%
%\end{enumerate}
If the tree $T(\Sigma_s)$ has depth $i_0 < d$ (that is, the tree has $i_0$ branching heights), then ${i_0 +1}$ will be the least index such that $k_{i_0+1}=-\infty$. \\\\
For example, in the tree shown here, $\sigma_1$ has 3-signature $(3,1,2)$ and $\sigma_2$ has 3-signature $(2,3,2)$. The 4-signature of $\sigma_1$ is $(3,1,2, - \infty)$.
 \end{minipage}
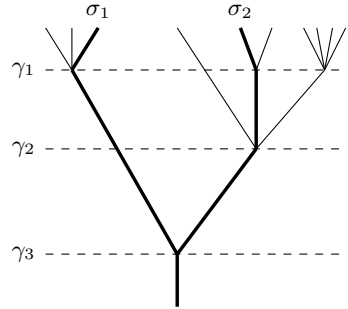
\begin{minipage}{0.05 \textwidth}
\end{minipage}
\begin{minipage}{0.3 \textwidth}
\begin{tikzpicture}[scale = 0.7]
\draw [dashed] (0.5,1) -- (6.2,1);
\draw [dashed] (0.5,3) -- (6.2,3);
\draw [dashed] (0.5,4.5) -- (6.2,4.5);
\node[{left}] at (0.5,1){\footnotesize$\gamma_3$};
\node[{left}] at (0.5,3){\footnotesize$\gamma_2$};
\node[{left}] at (0.5,4.5){\footnotesize$\gamma_1$};

\draw [line width=1.3pt] (3,0) -- (3,1);
\draw [line width=1.3pt] (3,1) -- (1,4.5);
\draw [line width=1.3pt] (1,4.5) -- (1.5,5.3);
\draw (1,4.5) -- (1,5.3);
\draw (1,4.5) -- (0.5,5.3);

\draw [line width=1.3pt] (3,1) -- (4.5,3);
\draw [line width=1.3pt] (4.5,3) -- (4.5,4.5);
\draw [line width=1.3pt] (4.5,4.5) -- (4.2,5.3);
\draw (4.5,4.5) -- (4.8,5.3);
\draw (4.5,3) -- (3,5.3);
\draw (4.5,3) -- (5.8,4.5);
\draw (5.8,4.5) -- (5.4,5.3);
\draw (5.8,4.5) -- (5.65,5.3);
\draw (5.8,4.5) -- (5.95,5.3);
\draw (5.8,4.5) -- (6.2,5.3);

\node[{above}] at (1.5,5.3){\footnotesize$\sigma_1$};
\node[{above}] at (4.2,5.3){\footnotesize$\sigma_2$};
\end{tikzpicture}
\end{minipage}\\
%{\color{blue}\begin{lemma}
%Let $(\{C_i\}_{1 \leqslant i \leqslant r}, \Sigma)$ be a cell array. Fix some $1 \leqslant i \leqslant r$ and let $\zeta$ be a section of $\Sigma^{(i)}$.
%Put $\rho(s):= \max_{c \in \Sigma^{(1)}_s} \{\ord(\zeta(s) -c)\}$.
%Given a section $\sigma_j$ of $\Sigma^{(1)}$, we write $\Theta_{\sigma_{j}(s)}$ for the top leaf of  $C_1^{\sigma_j(s)}$.  Then the following holds for every $s \in S$:
%\begin{itemize}
%\item[(i)] There exist at most $q_K^{m_1}$ leaves $\Theta_{\sigma_j(s)}$ (with $\sigma_j(s) \in \Sigma^{(1)}_s$), for which \[\left(\bigcup_{\gamma \geqslant \rho_{1, \text{max}}(s) }C_i^{\zeta(s), \gamma}\right) \cap \Theta_{\sigma_j(s)} \neq \emptyset.\]
%\item[(ii)] Let $\gamma < \rho_{1,\text{max}}(s)$. If there exists $\sigma_j \in \Sigma^{(1)}$ 
%such that $C_i^{\zeta(s), \gamma} \cap \Theta_{\sigma_j(s)}$ is nonempty, then either $\gamma$ is a branching height of $\Sigma^{(1)}_s$, or $\gamma = \rho(s)$. 
%Moreover, a leaf $C_i^{\zeta(s), \gamma}$ can intersect at most $q_K^{m_1}$ balls $\Theta_{\sigma_{j}(s)}$.
%\end{itemize}
%\end{lemma}
%\begin{proof}
%currently in the proof of the proposition...
%\end{proof}
%}
We will now show that, if the tree associated to some $\Sigma^{(i)}_s$ is infinite, then it can be assumed to be dense, in the following sense: 
\begin{lemma}\label{lemma:densetrees}
 Let $(\{C_i\}_{1 \leqslant i \leqslant r}, \Sigma)$ be a cell array defining a set $X$. Assume that there exists $s_0 \in S$ for which there are infinitely many $(C_1, \Sigma^{(1)}_{s_0})$-equivalence classes. Let $R >r$ be an integer. Then there exists a definable set $\Sigma' \subseteq \Sigma$, such that $(\{C_i\}_{1 \leqslant i \leqslant r}, \Sigma')$ is a cell array defining the same set $X$, such that all elements of $\Sigma_{s_0}^{(1)}$ have $R$-signature $(q_K, \ldots, q_K)$. 
\end{lemma}
\begin{proof}%\textcolor{purple}{REPLACED $s$ BY $s_0$ IN THIS LEMMA}\\
Let $s_0\in S$ be such that there are infinitely many $(C_1,\Sigma_{s_0}^{(1)})$-equivalence classes. For $\kappa$ an infinite cardinal number, let $\{\sigma_j\mid j< \kappa\}$ be a set of sections of $\Sigma^{(1)}$ such that
\begin{enumerate}
\item[(i)] each $(C_1,\Sigma_{s_0}^{(1)})$-equivalence class is represented by some $\sigma_j(s_0)$;
\item[(ii)] for $j<j'<\kappa$, $\sigma_{j}$ and $\sigma_{j'}$ are not $(C_1,\Sigma_{s_0}^{(1)})$-equivalent.
\end{enumerate}

Let $\gamma_l(s_0)$ be the $l^{\text{th}}$-branching height of $\Sigma^{(1)}_{s_0}$. 

\begin{claim} 
For any $d\in\N\setminus\{0\}$, there exists a finite set of ordinals $W_d$ such that for all $j<\kappa$ with $j\notin W_d$, the $d$-signature of $\sigma_j(s_0)$ equals  $(q_K, \ldots, q_K)$.
\end{claim}

Suppose that the claim is false, and let $d\in \N\setminus\{0\}$ be the smallest integer witnessing this. Let $(q_K,\ldots,q_k,k_d)$ be a $d$-signature with $k_d<q_K$ such that the set 
\[
J:=\{j<\kappa \mid  \sigma_j(s_0) \text{ has signature } (q_K,\ldots,q_k,k_d)\},
\]
is infinite in $\kappa$. The set 
\[
Z:=\bigcup_{j\in J} B_{\gamma_{d-1}(s)}(\sigma_j(s_0))
\]
is a definable subset of $K$ which is the union of infinitely many balls of radius $\gamma_{d-1}(s_0)$ (here we put $\gamma_0(s_0)$ equal to the radius of the equivalence classes of $\Sigma^{(1)}_{s_0}$, i.e. $\gamma_0(s_0):= \rho_{max}(s_0)+m_1$, where $\rho_{max}(s_0)$ is the height of the top leaves for $C_1$) %\textcolor{red}{Shouldn't this be the radius of an equivalence class, so $\rho_{\max}(s) + m_1$ - but how to indicate this $\rho_{\max}$ is for $C_1$?}
 which are maximal with respect to inclusion in $Z$. By semi-algebraic cell decomposition, this situation cannot occur in a $P$-minimal field, which shows the claim.   
\\\\
Let $r$ be the number of cell conditions in the cell array (counted with multiplicity). By our claim, we know that, whenever we fix an integer $R>r$, we can assume that the $R$-signature of $\sigma_j(s_0)$ will be $(q_K, q_K,  \ldots, q_K)$ for all $j<\kappa$, except for a finite set of indices $W_R$. Now define a set $\widetilde{W}_R$ as follows: %For each $j \in W_R$, let $\mathcal{T}_j$ denote the 
%\[ \widetilde{W}_R:= \{(s, c) \in \Sigma^{(1)} \mid \exists j \in W_R: \ord(c - \sigma_j(s)) \geqslant \gamma_{R}(s) \}\]
\[ \widetilde{W}_R:= \{c \in \Sigma_{s_0}^{(1)} \mid \exists j \in W_R: \ord(c - \sigma_j(s_0)) \geqslant \gamma_{R}(s_0) \}\]
Let $\Sigma' \subseteq \Sigma$ be the set obtained by removing the following fibers from $\Sigma_{s_0}$: 
\[
\{( c,\zeta_2,\ldots,\zeta_r)\in \Sigma_{s_0}: c \in  \widetilde{W}_R\}.
\]
The array $(\{C_i\}_i,\Sigma')$ still defines $X$ and moreover, all elements of $(\Sigma')^{(1)}_{s_0}$ have the same $R$-signature $(q_K,\ldots, q_K)$. %We may therefore assume, without loss of generality, that $\Sigma=\Sigma'$ and that all elements of $\Sigma^{(1)}_s$ have the same $R$-signature $(q_K,\ldots, q_K)$ for some integer $R>r$. 
\end{proof}

%The possibilities depend on properties of the trees $T(\Sigma^{(i)}_s)$, and the details will be discussed in many of the proofs that follow. We will always use the notational convention that names like $\sigma$, $\sigma',...$ (or maybe $\sigma_i$) are used for the potential centers of the pre-cell we focus on. Centers for potential cells associated to different pre-cells will be denoted as $\zeta, \zeta',...$.
%}

We are now ready to prove Proposition \ref{prop:finitebranching}.
\begin{proof}[Proof of Proposition \ref{prop:finitebranching}]
%It is enough to show there is a the cardinality of the trees $T(\Sigma^{(i)}_ s)$ is uniformly bounded for all $1\leq i\leq r$ and all $s\in S$. 
Permuting the cell conditions if necessary, it suffices to show the result for $\Sigma^{(1)}$. 
%{\color{red} Without loss of generality, we may suppose that the upper bound $\beta_i(s)$ of a cell condition $C_i$ is as small as possible, that is, we will assume that $\beta_i(s) -1 \equiv \ord \lambda_i \mod n_i$. Notice that such assumption does not contradict any of the initial hypothesis. } 
Suppose towards a contradiction that such a uniform bound does not exist. By compactness, possibly working over an elementary extension, let $s\in S$ be such that there are infinitely many $(C_1,\Sigma_s^{(1)})$-equivalence classes. Fix some sufficiently large value of $R$, such that at least $R>\max\{r,m_1\}$. 
Applying Lemma \ref{lemma:densetrees}, we may assume that all elements of $\Sigma^{(1)}_s$ have the same $R$-signature $(q_K,\ldots, q_K)$. 

We need to fix some notations first. We write $\sigma_j$ for potential centers in $\Sigma^{(1)}$. The top leaf of a potential cell fiber $C_1^{\sigma_j(s)}$ will be denoted by $\Theta_{\sigma_j(s)}$. Note that for $j \neq j'$, the leaves $\Theta_{\sigma_{j}(s)}$ and $\Theta_{\sigma_{j'}(s)}$ are disjoint (this follows from the assumption that $\sigma_j$ and $\sigma_{j'}$ are non-equivalent at $s$). 
\\\\
 %We will now estimate how many cell conditions are required to contain all the leaves $\Theta_{\sigma_{j}(s)}$. 
Fix a cell condition $C_i$ from the description of the array, together with a center $\zeta$ from $\Sigma^{(i)}$. Write $\rho(s)$ for the height where $\zeta(s)$ branches off from the tree of $\Sigma^{(1)}_s$, i.e. put
\[ \rho(s):= \max_{c \in \Sigma^{(1)}_s} \{\ord(\zeta(s) -c)\}.\]
Note that $\rho(s) \in \Gamma_K \cup \{\infty\}$. We want to know in what ways leaves of $C_i^{\zeta(s)}$ can intersect with balls $\Theta_{\sigma_{j}(s)}$. %Let $\gamma \in \Gamma_K \cup \{\infty\}$. 
Note that the following always holds if $t\in C_i^{\zeta(s), \gamma} \cap \Theta_{\sigma_j(s)}$. For such a $t$, $\ord(t-\zeta(s)) = \gamma$ and $\ord(t-\sigma_j(s)) = \rho_{1,\max}(s)$. Hence, one has that
\begin{eqnarray*}
\ord(\zeta(s)-\sigma_j(s)) &=& \ord\big((\zeta(s)-t)+(t-\sigma_j(s))\big)\\ & \geqslant &\min\big\{\ord(\zeta(s)-t),\ord(t-\sigma_j(s))\big\}\\ &=& \min\{\gamma, \rho_{1,\max}(s)\}.
\end{eqnarray*}
We will now first consider the leaves of $C_i^{\zeta(s)}$ for which $\gamma \geqslant \rho_{1,\text{max}}$. For these we have the following claim: 
\begin{claim}
There exist at most $q_K^{m_1}$ leaves $\Theta_{\sigma_j(s)}$ (with $\sigma_j(s) \in \Sigma^{(1)}_s$), for which \[\left(\bigcup_{\gamma \geqslant \rho_{1, \text{max}}(s) }C_i^{\zeta(s), \gamma}\right) \cap \Theta_{\sigma_j(s)} \neq \emptyset.\]
\end{claim}
Note that the above intersection will be empty unless $\rho(s) \geqslant \rho_{1,\text{max}}(s)$. Now, if $C_i^{\zeta(s), \gamma} \cap \Theta_{\sigma_j(s)}$ is nonempty for some center $\sigma_j(s)$ and some $\gamma \geqslant \rho_{1, \text{max}}$, then it must hold that 
\begin{equation*} \label{eq:whatever}\ord(\zeta(s) - \sigma_j(s)) \geqslant \rho_{1,\text{max}}.\end{equation*} Moreover, there can at most be $q_K^{m_1}$ non-equivalent centers with this property. %\eqref{eq:whatever}. Let $\{\hat{\sigma}_1(s), \ldots, \hat{\sigma}_{a}\}$ be a set of representatives, for a suitable $a \leqslant q_{K}^{m_1}$).
%If $\gamma > \rho_{1, \text{max}}(s)$ and the intersection $C_i^{\zeta(s), \gamma} \cap \Theta_{\sigma_j(s)}$ is nonempty, then $\sigma_j(s)$ must be equivalent to one of the $\hat{\sigma}_i(s)$ that satisfy property \eqref{eq:whatever}. 
Our claim follows immediately from this observation.
\\\\
For the remaining leaves of $C_i^{\zeta(s)}$, one has that
\begin{claim}
Let $\gamma < \rho_{1,\text{max}}(s)$. If there exists $\sigma_j \in \Sigma^{(1)}$ 
such that $C_i^{\zeta(s), \gamma} \cap \Theta_{\sigma_j(s)}$ is nonempty, then either $\gamma$ is a branching height of $\Sigma^{(1)}_s$, or $\gamma = \rho(s)$.
\end{claim}
\noindent\begin{minipage}{0.45 \textwidth}
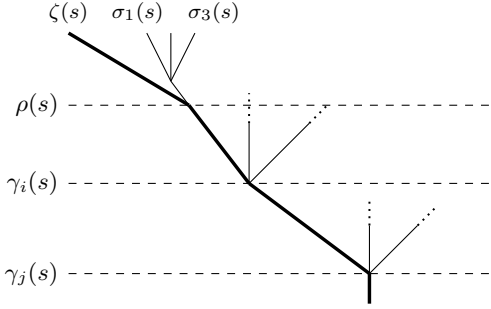
\begin{tikzpicture}[scale = 0.8]
\draw [dashed] (0,1) -- (7,1);
\draw [dashed] (0,2.5) -- (7,2.5);
\draw [dashed] (0,3.8) -- (7,3.8);
\node[{left}] at (0,1){\footnotesize$\gamma_j(s)$};
\node[{left}] at (0,2.5){\footnotesize$\gamma_i(s)$};
\node[{left}] at (0,3.8){\footnotesize$\rho(s)$};

\draw[line width = 1.3 pt] (5,0.5) -- (5,1);
%\node[{left}] at (4,1){\footnotesize$\gamma_j$};
\draw[line width = 1.3 pt] (5,1) -- (3,2.5);
\draw (5,1) -- (5,1.8);
\draw[dotted, line width=0.75pt] (5,1.8) -- (5,2.2);
\draw (5,1) -- (5.8,1.8);
\draw[dotted, line width=0.75pt] (5.8,1.8) -- (6.1,2.1);

\draw[line width = 1.3pt] (2,3.8) -- (0,5);

\draw (2,3.8) -- (1.7,4.2);
\draw[line width = 1.3pt] (3,2.5) -- (2,3.8);
\draw (3,2.5) -- (3,3.5);
\draw[dotted, line width=0.75pt] (3,3.5) -- (3,4);
\draw (3,2.5) -- (4,3.5);
\draw[dotted, line width=0.75pt] (4,3.5) -- (4.3,3.8);

\draw (1.7,4.2) -- (1.3,5);
\draw (1.7,4.2) -- (1.7,5);
\draw (1.7,4.2) -- (2.1,5);

%\draw (2,4.2) -- (1.6,5);
%\draw (2,4.2) -- (2,5);
%\draw (2,4.2) -- (2.4,5);

%\draw (3.3,4.2) -- (2.9,5);
%\draw (3.3,4.2) -- (3.3,5);
%\draw (3.3,4.2) -- (3.7,5);

\node[{above}] at (0,5){\scriptsize$\zeta(s)$};
\node[{above}] at (1.15,5){\scriptsize$\sigma_1(s)$};
%\node[{above}] at (1.7,5){\scriptsize$\sigma_2(s)$};
\node[{above}] at (2.42,5){\scriptsize$\sigma_3(s)$};
%\node[{above}] at (1.58,4.9){\scriptsize$\sigma_4$};
%\node[{above}] at (2,4.9){\scriptsize$\sigma_5$};
%\node[{above}] at (2.42,4.9){\scriptsize$\sigma_6$};
%\node[{above}] at (2.88,4.9){\scriptsize$\sigma_7$};
%\node[{above}] at (3.3,4.9){\scriptsize$\sigma_8$};
%\node[{above}] at (3.75,4.9){\scriptsize$\sigma_9$};
\end{tikzpicture}
\end{minipage}
\begin{minipage}{0.55 \textwidth}
%\textcolor{blue}{\\Indeed, let $t\in C_i^{\zeta(s), \gamma} \cap \Theta_{\sigma_j(s)}$. Then $\ord(t-\zeta(s)) = \gamma$ and $\ord(t-\sigma_j(s)) = \rho_{1,\max}(s)$. Hence, one has that
%\begin{eqnarray*}
%\ord(\zeta(s)-\sigma_j(s)) &=& \ord\big((\zeta(s)-t)+(t-\sigma_j(s))\big)\\ & \geqslant &\min\big\{\ord(\zeta(s)-t),\ord(t-\sigma_j(s))\big\}\\ &=& \min\{\gamma, \rho_{1,\max}(s)\}.
%\end{eqnarray*}
%}
Since $\gamma < \rho_{1,\max}(s)$, we must have that \[\ord(\zeta(s)-\sigma_j(s)) = \gamma.\] Note that by the definition of $\rho(s)$, we have that $\rho(s) \geqslant \gamma$. Now if $\rho(s) > \gamma$, there exists $c \in \Sigma^{(1)}_s$ such that $\ord(\zeta(s)-c) > \gamma$. We have to show that in this case $\gamma$ is a branching point. This holds since
\end{minipage}
\begin{eqnarray*}
\ord(c - \sigma_j(s)) &=& \ord\big((c-\zeta(s)) + (\zeta(s)-\sigma_j(s))\big)\\ &\geqslant& \min\big(\ord(c-\zeta(s)), \ord(\zeta(s)-\sigma_j(s)\big)\\ &=& \gamma.
\end{eqnarray*}
Again, since $\ord(c-\zeta(s)) > \gamma = \ord(\zeta(s)-\sigma_j(s))$, this must be an equality. Therefore, $c$ and $\sigma_j(s)$ are nonequivalent centers of $\Sigma^{(1)}_s$ that branch at height $\gamma$.
\\\\
We will also need to use the following.
\begin{claim}
Let $\gamma < \rho_{1, \text{max}}(s)$. 
Then a leaf $C_i^{\zeta(s), \gamma}$ can intersect at most $q_K^{m_1}$ balls $\Theta_{\sigma_{j}(s)}$. 
\end{claim}
Fix some $\gamma < \rho_{1, \text{max}}(s)$ for which there are at least two non-equivalent centers $\sigma_j(s), \sigma_{j'}(s)$ such that 
\begin{equation}\label{eq:whatever2}
C_i^{\zeta(s), \gamma} \cap \Theta_{\sigma_{j}(s)} \neq \emptyset \quad  \text{and} \quad C_i^{\zeta(s), \gamma} \cap \Theta_{\sigma_{j'}(s)} \neq \emptyset
\end{equation}
 (for other values of $\gamma$ there is nothing to prove).
Let $B_{j,j'}$ denote the smallest ball containing both $\Theta_{\sigma_{j}(s)}$ and $\Theta_{\sigma_{j'}(s)}$. 
Since $\Theta_{\sigma_j(s)}$ and $\Theta_{\sigma_{j'}(s)}$ are disjoint, \eqref{eq:whatever2} implies that  $B_{j,j'} \subset C_i^{\zeta(s), \gamma}$.
\\\\
 Put $\gamma_{j,j'}: \ord(\sigma_j(s) - \sigma_{j'}(s))$, and note that $\gamma_{j,j'}$ is a branching height of $\Sigma^{(1)}_s$. We need to consider the location of this branching height $\gamma_{j,j'}$ versus $\rho_{1, \text{max}}(s)$. %We need to distinguish between two cases, depending on the location of the branching heights $\gamma_i(s)$ versus $\rho_{1, \text{max}}$.
\\\\
 \begin{minipage}{0.55 \textwidth}
First suppose that $\gamma_{j,j'} \leqslant\rho_{1,\text{max}}(s)$.
%Let us first consider the case where $\rho_{1,\text{max}}(s) \geqslant \gamma_1(s)$. Put $\gamma_{j,j'}:= \ord(\sigma_j(s) - \sigma_{j'}(s))$.
In this situation, we find that $B_{j,j'} = B_{\gamma_{j,j'}(s)}(\sigma_j(s))$.
%and hence $B_{j,j'} \cap \Sigma^{(1)}_s \neq \emptyset$.
 \\
 Since $B_{j,j'}$ contains centers, but $\gamma_{j,j'} \leqslant \rho_{1,\text{max}}(s)$, we obtain a contradiction to condition (iv) from the definition of cell array (Definition \ref{def:array}). Hence, condition \eqref{eq:whatever2} can never be satisfied in this case.% we have to check whether the admissibility condition is not violated.  More concretely, Lemma \ref{lemma:annoying_balls_are_small} imposes that the radius of $B_{j,j'}$ has to be strictly bigger than $\rho_{1,\text{max}}(s)$. }
 \end{minipage}
 \begin{minipage}{0.1 \textwidth} \end{minipage}
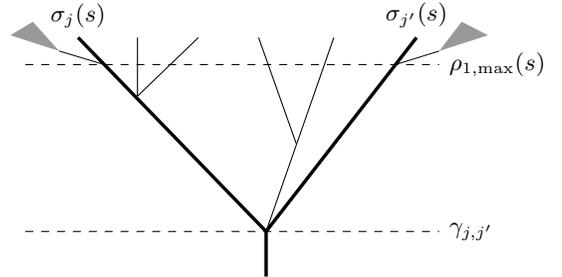
 \begin{minipage}{0.5 \textwidth}

 \begin{tikzpicture}[yscale = 0.6]
\draw [dashed] (0.8,1) -- (6.3,1);
%\draw [dashed] (0.5,3) -- (6.2,3);
\draw [dashed] (0.8,4.7) -- (6.3,4.7);
%\node[{left}] at (0.5,1){\footnotesize$$};
\node[{right}] at (6.3,1){\footnotesize$\gamma_{j,j'}$};
\node[{right}] at (6.3,4.7){\footnotesize$\rho_{1,\text{max}}(s)$};

\draw (2.29,4) -- (3.1,5.3);
\draw (2.29,4) -- (2.29, 5.3);
\draw (4.4,2.95) -- (3.9,5.3);
\draw [line width = 0.3pt](1.85,4.7) -- (1.25, 5);
\path [fill = grey] (1.25,5) -- (0.6,5.35) -- (0.9,5.65) -- cycle;
\path [fill = grey] (6.3,5) -- (6.6,5.65) -- (6.95,5.35) -- cycle;
\draw  [line width = 0.3pt](5.72, 4.7) -- (6.3, 5);

\draw [line width=1.3pt] (4,0) -- (4,1);
\draw [line width=1.3pt] (4,1) -- (1.5,5.3);
\draw (4,1) -- (4.9,5.3);
%\draw [line width=1.3pt] (2,4.5) -- (2.5,5.3);
%\draw (2,4.5) -- (2,5.3);
%\draw (2,4.5) -- (1.5,5.3);

\draw [line width=1.3pt] (4,1) -- (6, 5.3);
%\draw [line width=1.3pt] (5.5,3) -- (5.5,4.5);
%\draw [line width=1.3pt] (5.5,4.5) -- (5.2,5.3);
%\draw (5.5,4.5) -- (5.8,5.3);
%\draw (5.5,3) -- (4,5.3);
%\draw (5.5,3) -- (6.8,4.5);
%\draw (6.8,4.5) -- (6.4,5.3);
%\draw (6.8,4.5) -- (6.65,5.3);
%\draw (6.8,4.5) -- (6.95,5.3);
%\draw (6.8,4.5) -- (7.2,5.3);

\node[{above}] at (1.5,5.3){\footnotesize$\sigma_j(s)$};
\node[{above}] at (6,5.3){\footnotesize$\sigma_{j'}(s)$};
\end{tikzpicture}
 \end{minipage}
 
 Now consider the case where $\gamma_{j,j'} > \rho_{1, \text{max}}(s)$. This condition expresses that $\sigma_j(s)$ and $\sigma_{j'}(s)$ branch above
 $\rho_{1, \text{max}}(s)$. There can be at most $m_1$ such branching heights, and hence the leaf $C_i^{\zeta(s), \gamma}$ can intersect at most $q_K^{m_1}$ balls $\Theta_{\sigma_j}(s)$. This proves the claim.
  %\textcolor{blue}{However, we have that $\gamma_{j,j'} \leqslant \gamma_1(s) \leqslant \rho_{1,\text{max}}(s)$. This contradicts Lemma \ref{lemma:annoying_balls_are_small}, and hence condition \eqref{eq:whatever2} can never be satisfied in this case.}
 \\\\
\begin{minipage}{0.52 \textwidth} After a possible reordering, we can assume that the elements $\sigma_j(s) \in \Sigma^{(1)}_s$ are ordered in such a way that for each $l \leqslant R$, the potential centers $\sigma_1(s), \ldots, \sigma_{q_k^l}(s)$ generate a finite tree of depth $l$. \\\\ The picture shows an example for $q_K =3$ and $l =2$.
\end{minipage}
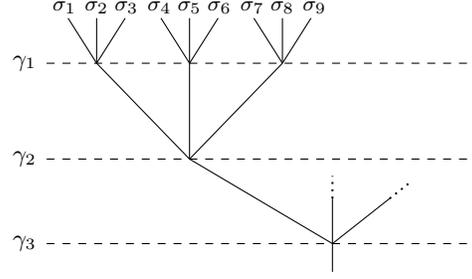
\begin{minipage}{0.05 \textwidth}
\qquad
\end{minipage}
\begin{minipage}{0.3 \textwidth}
\begin{tikzpicture}[xscale = 0.95, yscale = 0.75]
\draw [dashed] (0,1) -- (6,1);
\draw [dashed] (0,2.5) -- (6,2.5);
\draw [dashed] (0,4.2) -- (6,4.2);
\node[{left}] at (0,1){\footnotesize$\gamma_3$};
\node[{left}] at (0,2.5){\footnotesize$\gamma_2$};
\node[{left}] at (0,4.2){\footnotesize$\gamma_1$};

\draw (4,0.5) -- (4,1);
\draw (4,1) -- (2,2.5);
\draw (4,1) -- (4,1.8);
\draw[dotted, line width=0.75pt] (4,1.8) -- (4,2.2);
\draw (4,1) -- (4.8,1.8);
\draw[dotted, line width=0.75pt] (4.8,1.8) -- (5.1,2.1);

\draw (2,2.5) -- (0.7,4.2);
\draw (2,2.5) -- (2,4.2);
\draw (2,2.5) -- (3.3,4.2);

\draw (0.7,4.2) -- (0.3,5);
\draw (0.7,4.2) -- (0.7,5);
\draw (0.7,4.2) -- (1.1,5);

\draw (2,4.2) -- (1.6,5);
\draw (2,4.2) -- (2,5);
\draw (2,4.2) -- (2.4,5);

\draw (3.3,4.2) -- (2.9,5);
\draw (3.3,4.2) -- (3.3,5);
\draw (3.3,4.2) -- (3.7,5);

\node[{above}] at (0.25,4.9){\scriptsize$\sigma_1$};
\node[{above}] at (0.7,4.9){\scriptsize$\sigma_2$};
\node[{above}] at (1.12,4.9){\scriptsize$\sigma_3$};
\node[{above}] at (1.58,4.9){\scriptsize$\sigma_4$};
\node[{above}] at (2,4.9){\scriptsize$\sigma_5$};
\node[{above}] at (2.42,4.9){\scriptsize$\sigma_6$};
\node[{above}] at (2.88,4.9){\scriptsize$\sigma_7$};
\node[{above}] at (3.3,4.9){\scriptsize$\sigma_8$};
\node[{above}] at (3.75,4.9){\scriptsize$\sigma_9$};
\end{tikzpicture}
\end{minipage}
\\\\\\
Now consider, for $m_1 <l<R$, the depth $l$ subtree of $T(\Sigma^{(1)}_s)$ defined above.  %\textcolor{blue}{BEST FORMULATION OF THIS?}.  
Combining the claims above, we can conclude that a single cell $C_i^{\zeta(s)}$  can never intersect more than $q_K^{m_1} + (l+1)q_K^{m_1} = (l+2) q_K^{m_1}$ top leaves $\Theta_{\sigma_j(s)}$ from this subtree (and a more careful count would probably show that this upper bound is too high). Since, for the given tree of depth $l<R$, there exist $q_K^l$ disjoint leaves $\Theta_{\sigma_{j}(s)}$, we can conclude that at least $\frac{q_K^{l-m_1}}{l+2}$ cell conditions are required to account for all top leaves. Hence, we obtain a contradiction when $l$ is sufficiently big, given that there are only a fixed number of cell conditions. %Since this number is strictly increasing with $l$ \textcolor{blue}{(and is not asymptotically bounded by a constant)}, we obtain a contradiction. 
We conclude that there cannot exist $s \in S$ for which the  the number of non-equivalent centers for $\Sigma^{(1)}_s$ is not bounded.% (and this bound only depends on the number of cell conditions, and hence is independent of $s$.)}
\end{proof}

% !TEX root = centers_main.tex

\section{Regularity} \label{sec:regularity}

The main purpose of this section is to prove Proposition \ref{prop:regular_cell_array}, which establishes that a cell array can be partitioned into finitely many \emph{regular} cell arrays. %Essentially, a cell array over $S$ is regular \textcolor{blue}{if the tree structure associated to $\Sigma$ satisfies some further conditions, uniformly in $s$, and the cell conditions do not overlap too much.}
%\textcolor{blue}{if it has a form which is somewhat uniform in $s$ DON'T REALLY THINK THIS IS A GOOD DESCRIPTION...}. 
A formal definition will be given in Subsection \ref{subsec:regular_arrays} (see Definition \ref{def:regular_array}). %Proposition \ref{prop:finitebranching} will be crucial to obtain such \textcolor{blue}{uniformity}. 
We start with some preliminaries needed to prove Proposition \ref{prop:regular_cell_array}. 

\subsection{Repartitionings}

Let $(\{C_i\}_{i\in I},\Sigma)$ be a cell array defining a set $X$. In this subsection we describe three procedures to obtain %from $(\{C_i\}_{i\in I},\Sigma)$ 
a new cell array $(\{C_i'\}_{i\in I'},\Sigma')$ that defines the same set $X$. These procedures are called \emph{repartitionings of $(\{C_i\}_{i\in I},\Sigma)$} and will be used often in what follows. Some care is needed to make sure that the new pair $(\{C_i'\}_{i\in I'},\Sigma')$ still  satisfies all conditions from the definition of a cell array (Definition \ref{def:array}). The details are given in the following lemma-definition.

\begin{lemma-definition} \label{lem-def:repartitioning}
Let $\mathcal{A}=(\{C_i\}_{1 \leqslant i \leqslant r}, \Sigma)$ be a cell array over $S$ defining a set $X$. 
\begin{itemize}
\item[(a)] Let $\delta: S \to \Gamma_K$ be a definable function. Given a cell condition $C_i$, there exists a definable set $\Sigma' \subseteq S \times K^{r+1}$ such that 
\[\mathcal{A}':=(\{C_1, \ldots, C_{i-1}, C_{i|(\alpha_i, \delta)}, C_{i|(\delta-1, \beta_i)}, C_{i+1},\ldots, C_r\}, \Sigma')\] 
is a cell array defining the same set $X$. 
%\\ Also, $\Sigma^{(j)} = \Sigma'^{(j)}$ for $j <i$ , $\Sigma^{(j)} = \Sigma'^{(j+1)}$ for $j >i$.
\item[(b)] Given a cell condition $C_i$, and $\ell \in \N \setminus\{0\}$, let $C_{i,j}$, for $0\leqslant j < \ell$ be the cell condition
\[ C_{i,j}(s,c,t) := \alpha_i(s) \ \square_1 \ \ord(t-c) \ \square_2 \ \beta_i(s) \wedge t-c \in \pi^{jn} \lambda Q_{\ell n_i,m_i}.\]
There exists a definable set $\Sigma' \subseteq S \times K^{r+\ell-1}$ such that 
\[\mathcal{A}':=(\{C_1, \ldots, C_{i-1}, C_{i,0}, \ldots,C_{i,\ell-1}, C_{i+1},\ldots, C_r\}, \Sigma')\] 
is a cell array defining the same set $X$.
\item[(c)] Given a cell condition $C_i$, and $\ell' \in \N$,  let $C_{i,j}$ denote the cell condition
\[ C_{i,j}(s,c,t) := \alpha_i(s) \ \square_1 \ \ord(t-c) \ \square_2 \ \beta_i(s) \wedge t-c \in  \lambda_j Q_{ n_i,m_i+\ell'},\]
where the elements $\lambda_j$ are representatives of each of the $q_K^{\ell'}$ disjoint subballs of size ($\ord \lambda + m + \ell'$) of $B_{\ord \lambda +m}(\lambda)$.
%Here each $\lambda_j$ is chosen such that $\lambda = \lambda_j + \pi^{\ord $ and $\ac_{m+\ell'}(\lambda_j - \lambda) = a_j$, where $\{a_1, \ldots, a_{q_K}^{\ell'}\}$ are a set of representatives for $\mathcal{O}_K \mod \pi^{\ell}$.
%Put $r':= r+q_{K}^{\ell'}-1$. 
Put $r':= q_{K}^{\ell'}$. 
There exists a definable set $\Sigma' \subseteq S \times K^{r+r'-1}$ such that the repartitioning 
\[\mathcal{A}':=(\{C_1, \ldots, C_{i-1}, C_{i,1, }\ldots,C_{i, r'}, C_{i+1},\ldots, C_r\}, \Sigma')\] 
%\[\mathcal{A}':=(\{C_1, \ldots, C_{i-1}, C_{i, \ell'}\ldots,C_{i, \ell'}, C_{i+1},\ldots, C_r\}, \Sigma')\] 
%where we have replaced $C_i$ by $q_{K}^{\ell'}$ copies of $C_{i,\ell'}$, 
is a cell array defining the same set $X$. \qedhere
\end{itemize}
\end{lemma-definition}
\begin{proof}
First consider part $(a)$. We will show how to define a set $\Sigma'$ such that conditions $(i)$ and $(iv)$ from the definition of cell array are still satisfied for the repartitioning.  Conditions $(ii)$ and $(iii)$ are left to the reader (but they should be rather obvious). Write $\rho_{(\alpha_i, \delta), \text{max}}(s)$ for the height of the top leaf for fibers of $C_{i|(\alpha, \delta)}$. First put
\[D_{i,s}:= \{ c\in K \mid \exists c' \in \Sigma^{(i)}_s: \ord(c-c') \geqslant \rho_{(\alpha_i, \delta), \text{max}}(s) + m_i\}.\]
Now, put $\zeta:= (\zeta_1, \ldots, \zeta_{i-1}, \zeta', \zeta_{i}, \ldots, \zeta_r)$, and let $\Sigma'$ be the set
\[\Sigma':=\{(s, \zeta) \in S \times K^{r+1} \mid \zeta_j \in \Sigma^{(j)}_s \wedge \zeta' \in D_{i,s} \wedge  
\phi(s,\zeta) = X_s\},\]
where $\phi(s,\zeta)$ is the formula expressing that the centers $\zeta$ induce a partition of $X_s$:
\[\phi(s,\zeta):= \left[\bigcup_{j\neq i} C_j^{\zeta_j} \ \cup C_{i|(\alpha_i, \delta)}^{\zeta'}\cup C_{i|(\delta-1, \beta_i)}^{\zeta_i} = X_s\right].\]
It should be clear that with this set $\Sigma'$, the repartitioning still defines the same set $X$, and that condition $(i)$ still holds. 

It remains to check condition $(iv)$. Note that there is only something to prove for the cell condition $C_{i|(\alpha_i, \delta)}$. Fix an $s \in S$. The set of centers for the clustered cell fiber associated to $s$ and $C_{i|(\alpha_i, \delta)}$ is then $D_{i,s}$. Suppose towards a contradiction that $(iv)$ is not satisfied for some $c \in D_{i,s}$, i.e. that $X_s$ contains a ball $B_\gamma(c)$, for some $\gamma \leqslant \rho_{(\alpha_i, \delta), \text{max}}(s)$. By construction, there exists $\zeta_i \in \Sigma^{(i)}_s$ such that $c$ and $\zeta_i$ are 
$(C_i, \Sigma^{(i)}_s)$-equivalent. However, this implies that $\zeta_i \in B_\gamma(c)$. But since $\zeta_i$ was already a potential center for the clustered cell $C_i^{\Sigma^{(i)}}$ induced by the original cell array, this contradicts condition $(iv)$ for the original cell array.

For case $(b)$, we will assume that $i =1$ to ease the notation, but the same idea can obviously be applied for other components.
For $0\leqslant j <r$, let $\rho_{1j, \text{max}}(s)$ denote the height of the top leaf for fibers of $C_{1,j}$. Let $D_{j,s}$ be the set
\[D_{j,s}:= \{c_j \in  K \mid \exists c' \in \Sigma^{(1)}_s: \ord(c_j-c')\geqslant \rho_{1j,\text{max}}(s) +m_1\},\]
and put $\zeta:= (c_0, \ldots, c_{\ell-1},\zeta_2, \ldots, \zeta_{r})$.  Now, let $\Sigma'$ be the set
\[ \Sigma' :=\{(s, \zeta) \in S \times K^{r + \ell-1} \mid c_j \in  D_{j,s} \wedge \zeta_i \in \Sigma^{(i)}_s \wedge \phi(s,\zeta)\},\]
where $
\phi(s,\zeta)$ is the formula
\[\phi(s,\zeta) := \left[\bigcup_{j=0}^{\ell-1} C_{1j}^{c_j} \cup \bigcup_{j=2}^rC_j^{\zeta_j} = X_s\right].\]
%The required set $\Sigma'$ can then be defined from this as follows, by putting
%\[\Sigma':= \{(s, c_0(s), \ldots, \zeta_{r}(s)) \in \widehat{\Sigma} \mid \bigcup_{j=0}^{\ell-1} C_{1j}^{c_i(s)} \cup \bigcup_{j=2}^rC_j^{\zeta_j(s)} = X_s\}. \]
%\[ \Sigma' =\{(s, \zeta_1, \ldots, \zeta_{i-1}, c, \ldots, c, \zeta_{i+1}, \ldots, \zeta_r) \in S \times K^{r + \ell-1} \mid (s, \zeta_1, \ldots, \zeta_{i-1}, c, \zeta_{i+1}, \ldots, \zeta_r) \in \Sigma.\}\]
We leave it to the reader to check that all conditions are satisfied in this case.

For $(c)$, the set $\Sigma'$ can be defined in a similar way. Note that in this case, the potential centers for the new cells $C_{i,j}$ are the same ones as for the old $C_i$, but each equivalence class splits in $q_K^{\ell'}$ smaller equivalence classes. Since there are no `new' centers, and the value of $\rho_{i,\text{max}}$ does not change, condition (iv) from the definition of cell array will be preserved. %TO CONTINUE? 
\end{proof}

\subsection{Regular cell arrays}\label{subsec:regular_arrays}

In order to give the formal definition of regularity we need the following definitions first. 

\begin{definition}\label{def:uniform_tree_structure} A clustered $C^\Sigma$ over $S$ is said to have  \emph{uniform tree structure} if for all $s,s'\in S$, the trees $T(\Sigma_s)$ and $T(\Sigma_{s'})$ are isomorphic. 
\end{definition}

Here, a function $f:T_1\to T_2$ between trees $T_1$ and $T_2$ is a tree isomorphism if $f$ is a bijection and both $f$ and $f^{-1}$ are order preserving. We will also need the following additional definitions for types of clustered cells. 

\begin{definition} Let $C^{\Sigma}$ be a clustered cell. Then $C^{\Sigma}$ is said to be 
 \begin{itemize}
 \item \emph{large ($M$-large)}, if there exists $M \in \N$ with $M>1$, such that $|\alpha(s) - \beta(s)|> M$ for all $s \in S$; 
 %\item \sout{\emph{large} if it is $M$-large for some $M>1$; }
 \item \emph{uniformly bounded ($M$-bounded)}, if there exists some $M \in \N$ with $M\geqslant1$, such that ${|\alpha(s) - \beta(s)|\leqslant M}$ for all $s \in S$;
 \item \emph{small}, if there exists a definable function $\gamma: S \to \Gamma_K$, such that for any potential center $\sigma:S \to K$, $C^{\sigma}$ is of the form
 \[C^{\sigma} = \{(s,t) \in S \times K \mid \ord(t-\sigma(s)) = \gamma(s) \wedge t-\sigma(s) \in \lambda Q_{n,m}\}.\qedhere\]
 \end{itemize}
 \end{definition}
\noindent We are now ready to define regular cell arrays. 
\begin{definition}\label{def:regular_array} A cell array $(\{C_i\}_{i\in I},\Sigma)$ is said to be \emph{regular} if it satisfies the following conditions:
\begin{enumerate}
\item[(R1)] There exists $n,m\in \NN$ such that all cell conditions are described using the same set $Q_{n,m}$.
\item[(R2)] For $i,i' \in I$, either $(\alpha_i(s), \beta_i(s)) \cap (\alpha_{i'}(s),\beta_{i'}(s)) = \emptyset$ for all $s \in S$ , or $(\alpha_i(s), \beta_i(s)) = (\alpha_{i'}(s),\beta_{i'}(s))$ for all $s \in S$; cell conditions $C_i, C_{i'}$ that share the same interval will be called \emph{parallel}.
\item[(R3)] There is a natural ordering on the cell conditions, that is,  either two cells are parallel, or, for any two non-parallel cells $C_i$ and $C_{i'}$, we
have that either $C_i$ lies \emph{on top of} $C_{i'}$ (if $\beta_{i'}(s) \leqslant \alpha_i(s) +1$) or $C_i$ lies \emph{below} $C_{i'}$ (if $\beta_i(s) \leqslant \alpha_{i'}(s) +1$).
\item[(R4)] If $C_i$ and $C_{i'}$ are copies of the same cell condition, then $\Sigma^{(i)} = \Sigma^{(i')}$.
\item[(R5)] For each $i\in I$, the clustered cell $C_i^{\Sigma^{(i)}}$ has uniform tree structure. 
\item[(R6)] If $C_i$ is large and $\gamma(s)$ is a branching height of $\Sigma^{(i)}_s$, then $\gamma(s) \leqslant \alpha_i(s)$.\qedhere
\end{enumerate}
\end{definition}

\begin{remark}\label{rem:preservation} For $x=\{1,\ldots,6\}$, let $\mathcal{A}=(\{C_i\}_{i\in I},\Sigma)$ be a cell array satisfying condition (R$x$) from Definition \ref{def:regular_array}. If $S$ is partitioned into sets $S_1,\ldots,S_l$, then each cell array $\mathcal{A}_{|S_j}$ also satisfies condition (R$x$). %\sout{Indeed, all conditions in Definition \ref{def:regular_array}, except condition (R2), are defined locally for each $s\in S$, hence they must also hold in each cell array $\mathcal{A}_{|S_j}$. Condition (R2) is also preserved since each cell array $\mathcal{A}_{|S_j}$ has the same set of cell conditions as $\mathcal{A}$.} 
In particular, if $\mathcal{A}$ is a regular cell array, then so are the arrays $\mathcal{A}_{|S_j}$. 
\end{remark}

\begin{lemma}\label{lem:uniform_tree_structure} Let $\mathcal{A}=(\{C_i\}_i,\Sigma)$ be a cell array. There is a definable partition of $S$ into sets $S_1,\ldots, S_l$ such that for each $j\in \{1,\ldots,l\}$, each clustered cell in $\mathcal{A}_{|S_j}$ has uniform tree structure. 
\end{lemma}

\begin{proof}
By Proposition \ref{prop:finitebranching}, there exist only finitely many tree isomorphism types for the trees $T(\Sigma_s^{(i)})$, for all $s\in S$ and all $1\leq i\leq r$. Since the tree isomorphism type of the finite tree $T(\Sigma_s^{(i)})$ is a definable condition, the result follows by a straightforward partitioning of $S$.
\end{proof}

\begin{lemma}\label{lem:regular1} Let $X \subseteq S\times K$ be a set defined by a cell array $\mathcal{A} =(\{C_i\}_i,\Sigma)$. There exist cell arrays $\mathcal{A}_j$, satisfying conditions  (R1) - (R5), such that the induced sets $\mathcal{A}_j(K)$ form a finite partition of $X$.
%
%There exists a definable partition of $X$ into sets $X_j \subseteq S_j \times K$, such that the sets $S_j$ partition $S$, and each $X_j$ can be described using a cell array $(\{C_i'\}_{i' \in I_j}, \Sigma_{I_j})$ that satisfies conditions (R1)-(R5). 
\end{lemma} 

\begin{proof}
Condition (R1) is obtained through a repartitioning of the original array $(\{C_i\}_i,\Sigma)$. Put $n:= \text{lcm}_i\{n_i\}$ and $m:= \text{max}_i\{m_i\}$. %be the least common \textcmultiple of $m_i,n_i$ for all $1\leq i\leq r$. 
By applying procedures (b) and (c) outlined in Lemma-Definition \ref{lem-def:repartitioning} to each cell $C_i$ with respect to $l_i:=\frac{n}{n_i}$ (for procedure (b)) and $l_i':=m-m_i$ (for procedure (c)), one obtains a repartinioning where all cell conditions are defined using the same set $Q_{n,m}$. We may therefore assume without loss of generality that $X=(\{C_i\}_i,\Sigma)$ already satisfies condition (R1).

Let us now first give the main ideas for a procedure to achieve conditions (R2) and (R3). We want to cut up the intervals in pieces such that there is never any overlap between them. If there were no parameter $s$ involved, one could simply do the following.  
If $C_1, C_2$ were cell conditions for which, say \[\alpha_2 < \alpha_1 < \beta_2 < \beta_1,\] we would split both conditions: replace $C_1$ by a condition $C_{1,1}$ with interval $(\alpha_1, \beta_2)$ and a condition $C_{1,2}$ with interval $(\beta_2-1, \beta_1)$. Similarly, split $C_2$ in a condition $C_{2,1}$ with interval $(\alpha_2, \alpha_1 +1)$ and a condition $C_{2,2}$ with interval $(\alpha_1, \beta_2)$. Each split will induce a new array representation of the set. Repeating this until there is no more overlap between intervals would achieve the first  condition of the lemma. 

In order to do this uniformly in $s$, one needs to make sure that the \emph{interval structure} is the same for all $s \in S$. This means that we need to first do a partitioning of $S$ to ensure that all the boundary points $\alpha_i(s), \beta_i(s)$ are ordered in the same way for all $s \in S$. Since this is a finite set, this can be done by a finite partition, so let $S_1,\ldots,S_l$ be such a partition. By Remark \ref{rem:preservation}, each cell array $\mathcal{A}_{|S_j}$ still satisfies condition (R1). Finally, we apply the above idea to cut the intervals of each cell array $\mathcal{A}_{|S_j}$  using a repartitioning as in (a) of Lemma-Definition \ref{lem-def:repartitioning}. Note that this new cell array satisfies both conditions (R2) and (R3). Moreover, the repartitioning (a) does not change the values of $n$ or $m$ used in $Q_{n,m}$ %$Q_{(\cdot,\cdot)}$ 
for any of the cell conditions, so the new cell arrays still satisfy condition (R1). Hence, without loss of generality we may suppose that $X=(\{C_i\}_i,\Sigma)$ already satisfies conditions (R1)-(R3).

For condition (R4), suppose that $C_i$ and $C_j$ are the same cell condition for $i\neq j$. At this point, there need not be any connection between the sets $\Sigma^{(i)}$ and $\Sigma^{(j)}$. However, we can replace both $\Sigma^{(i)}$ and $\Sigma^{(j)}$ by $\Sigma^{(i)} \cup \Sigma^{(j)}$, and propagate this to $\Sigma$ itself in the obvious way: if $\sigma_i \in \Sigma^{(i)}$, $\sigma_j \in \Sigma^{(j)}$, and $(s, \ldots, \sigma_i, \ldots, \sigma_j, \ldots)$ is contained in $\Sigma$, then add $(s, \ldots, \sigma_j, \ldots, \sigma_i, \ldots)$ to $\Sigma$ if necessary. 
This ensures condition (R4). In addition, since we did not change any cell condition, conditions (R1)-(R3) are still satisfied. 

Finally, by Lemma \ref{lem:uniform_tree_structure} and Remark \ref{rem:preservation} each cell array satisfying (R1)-(R4) can be partitioned into finitely many cell arrays satisfying (R1)-(R5). 
\end{proof}

\begin{proposition}\label{prop:regular_cell_array} Let $\mathcal{A}=(\{C_i\}_{i\in I},\Sigma)$ be a cell array with $\mathcal{A}(K) = X$. 
There exist regular cell arrays $\mathcal{A}_j$, such that the induced sets $\mathcal{A}_j(K)$ form a finite partition of $X$.
%There is a definable partition of $X$ into sets $X_j \subseteq S_j \times K$, such that the sets $S_j$ definably partition $S$, and each $X_j$ can be described using a regular cell array. 
\end{proposition}

\begin{proof}
By Lemma \ref{lem:regular1}, we can assume that $\mathcal{A}$ %$X=(\{C_i\}_{i\in I},\Sigma)$
 already satisfies conditions (R1)-(R5), so it remains to show how to obtain condition (R6). 

Let $i\in I$ and $N \in \N$ be such that $C_i^{\Sigma^{(i)}}$ is a large clustered cell for which each fiber $(C_i^{\Sigma^{(i)}})_s$ has exactly $N$ branching heights $\gamma_1(s)>\cdots>\gamma_N(s)$. Put $I':=\{i\in I \mid C_{i'} \text{ is parallel to } C_i\}$. In the next steps of the proof, we will always apply the same repartitionings to each of the cell conditions in $\{C_i\}_{i\in I'}$, simultaneously. By condition (R5), the partitioning process described below can be carried out in a definable way, uniformly in $s$.

Consider the set \[\Delta(s):=\{\gamma_j(s) + k : 1\leq j\leq N, -m\leq k\leq m\},\] %given by 
%\[
%\delta_{j,k}(s):= \gamma_j(s)+k. 
%\]
where $m$ is the integer value in the set $Q_{n,m}$ used to describe all cell conditions (such an $m$ exists by (R1)). Partitioning $S$ into finitely many parts if necessary (which is allowed by Remark \ref{rem:preservation}), we may assume that the set $\{\alpha_1(s), \beta_1(s)\} \cup \Delta(s)$ is ordered in the same way for all $s \in S$ (with respect to the ordering $<$). Write $\delta_1(s) < \delta_2(s) < \ldots < \delta_{L}(s)$ for the elements of $\Delta(s) \cap (\alpha_1(s), \beta_1(s))$, and put $\delta_0(s):= \alpha_1(s)+1, \delta_{L+1}(s):= \beta_1(s)$. We now apply a repartitioning as in (a) of Lemma-Definition \ref{lem-def:repartitioning}, with respect to each function $\delta_j(s)$ and each cell $C_{i}$ for $i\in I'$. That is, we replace each cell condition $C_{i}$ by cell conditions 
\[C_{i,j}:= C_{i\,|(\delta_{j}-1, \delta_{j+1})},\]
for each $1\leq j\leq L$. Note that some of these conditions may induce empty sets (in which case we will drop the corresponding cell condition). 

The value of $m$ and $n$ does not change in these new cell conditions, so (R1) is preserved. The fact that the repartitioning is applied for all parallel cells simultaneously preserves both (R2) and (R3). The same is true for (R4). Indeed, if $C_1$ and $C_2$ are copies of the same cell condition (in the original array), then the above procedure produces cell conditions $C_{1,j}$, resp. $C_{2,j}$ such that for each $j$, $C_{1,j} = C_{2,j}$. Because condition (R4) holds for the original array, one has that $\Sigma^{(1)} = \Sigma^{(2)}$. This equality is preserved when applying repartitioning (a) of  Lemma-Definition \ref{lem-def:repartitioning} to both cell conditions. Since this is the only way to obtain multiple copies of the same cell condition, conditon (R4) must be preserved. 
 By Lemma \ref{lem:uniform_tree_structure} and Remark \ref{rem:preservation}, we can assume (R5) is also satisfied. 

Let us now explain how this partitioning will ensure (R6). Consider again the large cell condition $C_i$ from the original array, and its set of potential centers $\Sigma^{(i)}$. By the repartitioning, this cell condition was replaced by smaller cell conditions $C_{i,j}$. 
%for $C^{\Sigma^{(1)}}$ (and all copies of it). 
The set of potential centers for each part $C_{i| (\delta_{j}-1, \delta_{j+1})}$ (which we will denote as $\Sigma^{(i,j)}$), is defined from the set of potential centers for $C_i$, by procedure (a) outlined in Lemma-Definition \ref{lem-def:repartitioning}. In that procedure, either equivalence classes are preserved, or it may be that some equivalence classes merge, and are replaced by a ball containing both original classes: 
%is the same as that for $C$. However, the equivalence classes may change, 
indeed, any two centers in $\Sigma^{(i)}_s$ whose branching height is above $\delta_{j+1}+m$ are equivalent with respect to $C_{i| (\delta_{j}-1, \delta_{j+1})}$. So the tree $T(\Sigma_s^{({i,j})})$ associated to any of the cell conditions $C_{i,j}$ can have at most the same number of branching heights as the tree of $C_i$ (and will probably have less). 

Moreover, for \emph{large} cell conditions $C_{i, j}$ (deduced from $C_i$ or a copy of $C_i$), our construction assures there are no branching heights between $\delta_{j}$ and $\delta_{j+1} +m $, which indeed leaves us with a cell condition for which no branching heights are bigger than the lower bound of the cell. 
\\\\
A similar procedure should be repeated for the remaining parallel, large cell conditions. Note that this indeed ends after a finite number of steps, since the number of branching heights possibly contradicting (R6) only decreases at each step. 
\end{proof}

The following lemma gives a property of regular cell arrays that will be used often.  

\begin{lemma}\label{lemma:disjointness} Let $(\{C_i\}_{i\in I}, \Sigma)$ be a regular cell array and $i \in I$. If $\sigma_1(s), \sigma_2(s) \in \Sigma_s^{(i)}$ are non-equivalent centers, then $C_i^{\sigma_1(s)} \cap C_i^{\sigma_2(s)} = \emptyset.$
\end{lemma} 
\begin{proof}
Assume that $C_i$ is a large cell condition, as otherwise there is nothing to prove. If $\sigma_1(s), \sigma_2(s) \in \Sigma^{(i)}_s$ are non-equivalent centers, then condition (R6) implies that $\ord (\sigma_1(s) - \sigma_2(s)) \leqslant \alpha_i(s)$. Hence, for $(s,t) \in C_i^{\sigma_1}$ we have that \[\ord(t-\sigma_2(s)) = \ord((t-\sigma_1(s))+(\sigma_1(s) - \sigma_2(s)))\leqslant \alpha_i(s),\] which means that $(s,t) \not \in C_i^{\sigma_2}$.
\end{proof}

%\input{centers_regularity_backup.tex}
% !TEX root = centers_main.tex
\section{Separating cell arrays} \label{sec:largesmall}
%In this section, we will always assume that all occurring cell arrays and clustered cells are regular. %It may be that some of the partitioning we do here breaks  condition (R5), but by Lemma \ref{lem:uniform_tree_structure} and Remark \ref{rem:preservation} it can always be restored after a further finite partitioning, \textcolor{red}{which will not always be mentioned explicitly}. We will provide more details for the remaining regularity conditions.
In this section, we will need to keep track of the multiplicity with which a given cell condition occurs in a cell array. Since in a regular array, the associated set of potential centers is the same for each copy of a given cell condition, we will regroup this information, and, in the proofs that follow, whenever convenient adopt the following notation for regular cell arrays. The notation
\[ (\{C_i^{\langle k_i \rangle} \}_{1 \leqslant i \leqslant l}, \langle \Sigma \rangle), \]
with $\langle\Sigma \rangle \in S \times K^l$ will denote an array where the cell condition $C_i$ occurs with multiplicity $k_i$. The associated set of potential centers for $C_i$ will be denoted as $\langle\Sigma \rangle^{(i)}$, and corresponds to the projection of the fibers of $\langle\Sigma \rangle$ onto the $i$-th coordinate. Given a set $\langle\Sigma \rangle$, it should be clear to the reader how this set can be \emph{expanded} to the set $\Sigma \subseteq S \times K^{k_1 + \ldots + k_l}$ used in the standard notation. We will only use this condensed notation for regular arrays.% that satisfy all properties listed in Lemma \ref{lemma:small-large}.
\\\\
Our goal in this section is to show that, possibly after further partitioning or applying certain transformations, one can definably split a cell array into clustered cells $C_i^{\langle \Sigma \rangle^{(i)}}$.
Since these clustered cells are derived from regular cell arrays, they will inherit certain properties of regularity. %In particular,
%we will say that a clustered cell $C^{\Sigma}$ is regular if it has the structure of a regular cell array. 
%Note that not all the properties of arrays are relevant for clustered cells. The main difference between a general clustered cell and a regular one is that we have more control over the tree structure. 
The following terminology will be useful.  
\begin{definition}\label{def:multiball}
Let $k>0$ be an integer.  A set $H \subseteq S \times K$ is called a \emph{multi-ball of order $k$} over $S$, if every fiber $H_s$ (for $s \in S$) is a union of $k$ disjoint balls of the same radius.
%\item A multi-ball is called \emph{regular} if it has uniform tree structure.
%\item We  say that a \textcolor{TealBlue}{semi-}regular clustered cell $C^{\Sigma}$ has order $k$ if $\Sigma$ is a multi-ball of order $k$. 
\end{definition}
 
\begin{definition}\label{def:regularcluster}
A clustered cell $C^{\Sigma}$ is called \emph{regular of order $k$} if it is regular (when considered as a cell array) and %for every $s \in S$, 
$\Sigma$ is a multi-ball of order $k$, where the $k$ balls
coincide with the $k$ different $(C, \Sigma_s)$-equivalence classes. 
\end{definition}
\noindent In particular, the regularity condition (R6) implies that if two sections $\sigma, \sigma'$ of $\Sigma$ are not $(C, \Sigma_s)$-equivalent, then $C^{\sigma(s)} \cap C^{\sigma'(s)} = \emptyset,$ and hence for every $s\in S$, we have that, if $\sigma_1, \ldots, \sigma_k$ are sections of $\Sigma$ for which $\{\sigma_1(s), \ldots \sigma_k(s)\}$ are representatives of the $k$ equivalence classes in $\Sigma_s$, then
\[ C^{\sigma_1(s)} \cup C^{\sigma_2(s)} \cup \ldots \cup C^{\sigma_k(s)}\]
is a partition of $(C^{\Sigma})_s$. %{\color{orange}REMARK: did we define the notation $C^{c_i}$? It seems a bit strange to write it like this without the specific $s$ appearing in the notation.}
%\begin{proof} This follows immediately from the definition of regular cell array and from Lemma \ref{lemma:disjointness}.
%\end{proof}

%{\color{orange} REMARK: maybe we should make a comment somewhere that the splitting procedures from def 4.8 and 4.9 preserve the conditions of regularity. Uniform tree structure might be broken, but can be restored after finite partitioning.}
\begin{remark}\label{rem:regularsplit}
The splitting procedures outlined in Definitions \ref{def:split} and \ref{def:defsplit} can also be used for regular cell arrays, and the regularity condition is preserved under splits by projection. We leave it to the reader to check that, in particular, condition (R5) about uniformity in the tree structure is preserved. When applying a split by definable choice, condition (R5) might get lost initially, but this can always be restored by a further finite partitioning (as described in Lemma \ref{lem:uniform_tree_structure}) if necessary.
\end{remark}

Let us start by considering the cases where a clustered cell can be split off without modifying the array first. Here we use the terminology and notations of Definition \ref{def:split}. 

\begin{lemma} \label{lemma:largecellsplit}
 Let $\mathcal{A} = (\{C_i^{\langle k_i \rangle} \}_{1 \leqslant i \leqslant l}, \langle \Sigma \rangle)$ be a regular cell array, with $\mathcal{A}(K) = X$ and $l>1$, %, and let $C_i$ %(occurring with multiplicity $k_i$) be a cell condition 
 for which $C_1^{\langle \Sigma \rangle^{(1)}}$ is a regular clustered cell of order $k_1$. 
Then $\mathcal{A}$ can be partitioned as the union of $C_1^{\langle \Sigma \rangle^{(1)}}$ and the regular cell array $\left(\{C_i^{\langle k_i \rangle} \}_{2 \leqslant i \leqslant l}, \langle \Sigma \rangle^{(2, \ldots, l)}\right)$.
%\textcolor{red}{Then $X_i$ can be split off from $\mathcal{A}$ definably.}
\end{lemma}
\begin{proof}
%Without loss of generality, we may assume that $i =1$.  Then, using the terminology and notations of Definition \ref{def:split}, 
 The suggested split is a split at $k_1$ (by projection). The regularity claim follows from Remark \ref{rem:regularsplit}. Note that $C_1^{\langle \Sigma \rangle^{(1)}} = X^{(1, \ldots, k_1)}$. What needs to be checked is whether
\[C_1^{\langle \Sigma \rangle^{(1)}} \cap %\cancel{\left(\{C_i^{\langle k_i \rangle} \}_{2 \leqslant i \leqslant l}, \langle \Sigma \rangle^{(2, \ldots, l)}\right)} 
 X^{(k_1 +1, \ldots, \sum k_i)}= \emptyset.\]
The reason this intersection is empty is as follows. For any section $\sigma = (\sigma_{1,1}, \ldots, \sigma_{1,k_1}, \\ \sigma_{2, 1}, \ldots, \sigma_{l,k_l})$ of $\Sigma$, we get a partition
\begin{equation}\label{Xspart} X_s = \bigcup_{i=1}^{k_1} C_1^{\sigma_{1,i}(s)} \cup \left[\bigcup_{i=1}^{k_2} C_2^{\sigma_{2,i}(s)} \cup \ldots \cup \bigcup_{i=1}^{k_l} C_l^{\sigma_{l,i}(s)}\right],\end{equation}
where the elements $\sigma_{1,i}(s)$ are $k_1$ distinct (i.e., non-equivalent) elements of $\langle \Sigma \rangle_s^{(1)}$. However, by our assumption, this set only consists of $k_1$ equivalence classes. Hence, for any possible choice of $\sigma$, $\bigcup_{i=1}^{k_1} C_1^{\sigma_{1,i}(s)}$ is the same set, so a nonempty intersection would  imply the existence of a $\sigma$ that contradicts the fact that \eqref{Xspart} gives a partition of $X_s$.
%\textcolor{purple}{Put $\mathcal{A}':=\left(\{C_i^{\langle k_i \rangle} \}_{2 \leqslant i \leqslant l}, \langle \Sigma \rangle^{(2, \ldots, l)}\right)$.}
\end{proof}

Given a regular cell array $(\{C_i^{\langle k_i \rangle} \}_{1 \leqslant i \leqslant l}, \langle \Sigma \rangle)$, let us now consider a cell condition $C_1$ for which $\langle\Sigma\rangle^{(1)}$ is a multi-ball of order strictly bigger than $k_1$. In this case, the reasoning in the previous proof implies that there exists some center $\widehat{\sigma}$ in $\langle \Sigma \rangle^{(1)}$, and a section $\sigma' = (\sigma'_{1,1}, \ldots, \sigma'_{1,k_1}, \sigma'_{2, 1}, \ldots, \sigma'_{l,k_l})$ of $\Sigma$ such that for every $s$, 
\[C_1^{\widehat{\sigma}(s)} \cap \left[\bigcup_{i=1}^{k_2} C_2^{\sigma'_{2, i}(s)} \cup \ldots \cup \bigcup_{i=1}^{k_l} C_l^{\sigma'_{l,i}(s)}\right] \neq \emptyset\]
(and hence obviously $\widehat{\sigma}(s)$ is not equivalent to any element of $\{\sigma'_{1,1}(s), \ldots, \sigma'_{1,k_1}(s)\}$). We will refer to this situation by saying that $\widehat{\sigma}(s)$ admits \emph{external exchange}.
%
%Note that this is a definable condition, and hence that the subset  $\widehat{\Sigma} \subset \langle \Sigma \rangle^{(i)}$ consisting of those elements that do not admit external exchange is definable. Suppose that $\widehat{\Sigma}$ is a multiball of order $k'_i < k_i$. {\color{orange} REMARK: maybe $k_i'$ could depend on $s$? Should we say that we possibly have to partition $S$ according to the value of $k_i'$?} It can be checked easily that the clustered cell $C_i^{\widehat{\Sigma}}$ can be split off definably (in the remaining array, one needs to reduce the multiplicity of $C_i$ from $k_i$ to $k_i -k'_i$, and replace $\langle \Sigma \rangle^{(i)}$ by $\langle \Sigma \rangle^{(i)} \backslash \widehat{\Sigma}$).
%
%From now on, we will assume that all clustered cells which satisfy the conditions described above have been split off, and hence we are left with cell arrays where all potential centers admit external exchange. 
The following lemma shows that the property of external exchange has consequences for the size of a large cell. 

\begin{lemma}\label{lemma:externalbounded}
Let $\mathcal{A} = (\{C_i^{\langle k_i \rangle}\}_i, \langle\Sigma \rangle)$ be a regular cell array with $\mathcal{A}(K) =X$, and $C_j$ a large cell condition for which $\langle\Sigma \rangle^{(j)}$ is a multi-ball with order $k > k_j$.
%with $C_1^{\Sigma^{(1)}}$ a large clustered cell. If $C_1^{\Sigma^{(1)}}$ admits external exchange, 
 Then there exists $M \in \N$ such that $C_j$ is $M$-bounded. %\textcolor{blue}{Did we define this notion?}
\end{lemma}
\begin{proof}
Fix a large cell condition from the cell array, which will be denoted as $C_{\lambda}$:  
\[C_\lambda(s,c,t) = \alpha(s) < \ord (t-c) < \beta(s) \wedge t- c \in \lambda Q_{n,m}.\]
%(Here $\lambda$ refers to the condition $t-\sigma(s) \in \lambda Q_{n,m}$ in the description of the cell condition.) 
We write $k_\lambda$ for its multiplicity and $\langle\Sigma\rangle^{(\lambda)}$ for its set of potential centers. By assumption, $\langle\Sigma\rangle^{(\lambda)}$ is a multiball of order $k > k_\lambda$.  Let $\widehat{\sigma}$ be as in the discussion preceding this lemma.
%\textcolor{red}{We may assume (by the discussion preceding this lemma) that every potential center in $\widehat{\sigma} \in \langle\Sigma\rangle^{(\lambda)}$ admits external exchange.} 
%We will now try to count how many cells are needed to make this exchange possible. 
Hence, there exists a section $\sigma = (\sigma_{1}, \ldots, \sigma_{k_{\lambda}}, \zeta_1, \ldots, \zeta_{r_1 + r_2})$, such that, for all $s \in S$, $\widehat{\sigma}(s)$ is not $(C_\lambda, \langle\Sigma\rangle^{(\lambda)}_s)$-equivalent to any of the elements of $\{\sigma_{1}(s), \ldots, \sigma_{k_{\lambda}}(s)\}$. We write the corresponding decomposition of $X_s$ as
\[ X_s =%\bigcup_{i=1}^{k_2} C_1^{\sigma'_{2, i}(s)} \cup
\left[C_\lambda^{\sigma_{1}(s)} \cup \ldots \cup C_\lambda^{\sigma_{k_{\lambda}}(s)} \right] \cup \left[\bigcup_{i=1}^{r_1} C_i^{\zeta_{i}(s)}\cup \bigcup_{i=1}^{r_2} D_i^{\zeta_{r_1 +i}(s)}\right], \]
where the cells $C_i$ are parallel to $C_\lambda$ and the cells $D_i$ are non-parallel to $C_\lambda$. (We allow that $C_i = C_j$ for $i \neq j$ and similarly for $D_i$.) %The question we want to answer is: how many cell fibers $C_i$ and $D_i$ are there such that $C_{\lambda}^{\widehat{\sigma}(s)} \cap C_i^{\zeta_{i}(s)} \neq \emptyset$, resp. $C_{\lambda}^{\widehat{\sigma}(s)} \cap D_i^{\zeta_{i}(s)} \neq \emptyset$? 
Note that by Lemma \ref{lemma:disjointness}, the intersections $C_\lambda^{\widehat{\sigma}(s)} \cap C_\lambda^{\sigma_{i}(s)}$ are all empty, and hence \[C_\lambda^{\widehat{\sigma}(s)} \subset \left[\bigcup_{i=1}^{r_1} C_i^{\zeta_{i}(s)}\cup \bigcup_{i=1}^{r_2} D_i^{\zeta_{r_1 +i}(s)}\right].\] 
%\sout{Indeed, suppose that there would exist some $\sigma_{i}(s) \in \{\sigma_{1}(s), \ldots, \sigma_{k_{\lambda}}(s)\}$ such that $C_\lambda^{\widehat{\sigma}(s)} \cap C_\lambda^{\sigma_i(s)} \neq \emptyset$. By {\color{ForestGreen} parts 3 and 5} of Lemma \ref{lemma:small-large}, this can only happen if $\sigma_i(s)$ is equivalent to $\widehat{\sigma}(s)$, which would contradict our assumptions on $\widehat{\sigma}(s)$.}% that $\widehat{\sigma}(s) \not \in \{\sigma_{1}(s), \ldots, \sigma_{k_{\lambda}}(s)\}$.
%{\color{ForestGreen}\begin{claim} We may assume that $C_1^{\sigma(s)} \subseteq \left[\bigcup_{i=k_1}^{r} C_i^{\sigma'(s)} \right] \neq \emptyset$.
%\end{claim}
%
%Suppose there was $1\leq i<k_1$ such that $C_1^{\sigma(s)}\cap C_i^{\sigma'(s)}\neq\emptyset$. By parts 3 and 5 of Lemma \ref{lemma:small-large}, this implies that $C_1^{\sigma(s)}= C_i^{\sigma'(s)}$, contradicting that $C_1^{\sigma(s)}$ intersects some cell $C_i^{\sigma'(s)}$ with $k_1\leq i\leq r$. This shows the claim. }
We will show that there exists a fixed bound $N \in \N$ such that, for any $s \in S$, each of the intersections $C_{\lambda}^{\widehat{\sigma}(s)} \cap C_i^{\zeta_{i}(s)}$, resp. $C_{\lambda}^{\widehat{\sigma}(s)} \cap D_i^{\zeta_{r_1+i}(s)}$ can contain points of at most $N$ leaves of $C_{\lambda}^{\widehat{\sigma}(s)}$. The statement of the lemma follows from this, since clearly this implies that the larger the interval $(\alpha(s), \beta(s))$ in the description of $C_{\lambda}$ gets, the  more cells will be involved in this exchange process, yet the decomposition is finite.
\\\\
Let us first consider the non-parallel cells $D_i$.

\begin{claim} For every $s\in S$, and any $1 \leqslant i \leqslant r_2$, at most one leaf of $C_{\lambda}^{\widehat{\sigma}(s)}$ can intersect the cell fiber $D_i^{\zeta_{r_1+i}(s)}$.
\end{claim}
Write $(\alpha(s), \beta(s))$ for the interval associated to $C_{\lambda}$, and $(\alpha_i(s), \beta_i(s))$ for the interval associated to $D_i$. By assumption, these intervals have empty intersection. First consider the case where $D_i$ lies \emph{above} $C_\lambda$ (i.e., $\beta(s)  \leqslant \alpha_{i}(s) +1$). Suppose that $D_i^{\zeta_{r_1+i}(s)}$ contains a point $t$ from a leaf $C_{\lambda}^{\widehat{\sigma}(s), \gamma}$. Then $\ord(t-\zeta_{r_1+i}(s)) > \alpha_i(s)$, and hence $\ord(\zeta_{r_1+i}(s) - \widehat{\sigma}(s)) = \gamma$. But this implies that the cell fiber $D_i^{\zeta_{r_1+i}(s)}$ cannot possibly contain points from other leaves of $C_{\lambda}^{\widehat{\sigma}(s)}$. Hence, at most 1 leaf of $C_{\lambda}^{\widehat{\sigma}(s)}$ can intersect with $D_i^{\zeta_{r_1+i}(s)}$.

On the other hand, when $D_i$ lies \emph{below} $C_{\lambda}$ (i.e., $\beta_i(s) \leqslant \alpha(s) +1$), a cell fiber $D_i^{\zeta_{r_1+i}(s)}$ can contain at most a single leaf of $C_{\lambda}^{\widehat{\sigma}(s)}$ (or no leaf at all). Indeed, if $D_i^{\zeta_{r_1+i}(s)}$ would contain points from more than one leaf of $C_{\lambda}^{\widehat{\sigma}(s)}$, then $D_i^{\zeta_{r_1+i}(s)}$ would contain a ball $B_r(\widehat{\sigma}(s))$ which contains those leaves. It is easy to check that this ball $B_r(\widehat{\sigma}(s))$ would have radius $r < \rho_\text{max}(s)$, which contradicts condition (iv) of the definition of cell arrays (Definition \ref{def:array}).% Lemma \ref{lemma:annoying_balls_are_small}.
%{\color{orange} REMARK: So we can take $N_1$ to be 1?}
\begin{claim} %There is an integer $N_2$ such that at most $N_2$ 
For every $s \in S$, and any $1 \leqslant i \leqslant r_1$, at most $2m$ leaves of $C_{\lambda}^{\widehat{\sigma}(s)}$ can intersect the cell fiber  $C_i^{\zeta_i(s)}$.
\end{claim}
%In this case we have that $(\alpha(s), \beta(s)) = (\alpha
Consider a cell fiber $C_i^{\zeta_i(s)}$ for which $C_i^{\zeta_i(s)} \cap C_{\lambda}^{\widehat{\sigma}(s)} \neq \emptyset.$ 
%{\color{magenta}\sout{One can check that in this case,} 
%\[
%\alpha(s) < \ord (\widehat{\sigma}(s) - \zeta_i(s)) < \beta(s) +m. \ \ REMOVE
%\]}
Put $\gamma_0(s):= \ord (\widehat{\sigma}(s) - \zeta_i(s))$. It is sufficient to show that $C_i^{\zeta_i(s)} \cap C_{\lambda}^{\widehat{\sigma}(s)} \subseteq C_{\lambda | (\gamma_0(s) -m, \gamma_0(s) +m) }^{\widehat{\sigma}(s)}$, as this set cannot contain more than $2m$ leaves. 

Suppose that the intersection contains some $t \in K$ for which $\ord (t- \widehat{\sigma}(s)) \geqslant \gamma_0(s) +m$. Note that this implies that $\gamma_0(s) +m \leqslant \rho_{\text{max}}(s) $. One can check that for such a $t$ to exist, $C_i^{\zeta_i(s)}$ needs to contain the whole ball $B_{\gamma_0(s)+m}(\widehat{\sigma}(s))$, which would again contradict condition (iv) of Definition \ref{def:array} %Lemma \ref{lemma:annoying_balls_are_small} 
, since it would mean that $X_s$ contains a ball $B_r(\widehat{\sigma}(s))$ with radius $r < \rho_{\text{max}}(s) + 1$.  %\eqref{eq:branchingbound} implies that $\gamma_0(s)+m < \ldots$...}

Finally, suppose the intersection contains some $t \in K$ for which $\ord (t- \widehat{\sigma}(s)) \leqslant \gamma_0(s) -m$. In this case, we would have that $\ord(t- \widehat{\sigma}(s)) = \ord(t-\zeta_i(s)) \leqslant \gamma_0(s) -m$, and hence
the fact that $(t- \widehat{\sigma}(s)) \in \lambda Q_{n,m}$ would imply that also $(t- \zeta_i(s)) \in \lambda Q_{n,m}$. However, this contradicts the assumption that $t \in C_i^{\zeta_i(s)}$, 
% \[(t- \widehat{\sigma}(s)) \in \lambda Q_{n,m} \iff (t- \zeta_i(s)) \in \lambda Q_{n,m},\]
%which is a contradiction 
since $C_i$ is a  parallel cell condition different from $C_{\lambda}$ (and hence $\acm(\lambda_i) \neq \acm(\lambda)$.)
\end{proof}
A consequence of this lemma is the following.

\begin{proposition}\label{prop:large-small}
Let $\mathcal{A} = (\{C_i^{\langle k_i\rangle}\}_{1 \leqslant i \leqslant l}, \langle\Sigma \rangle)$ be a regular cell array defining a set $X$. There exists a finite partition of $\mathcal{A}$ into arrays $(\mathcal{A}_j)_{j\in J}$, such that for each $j\in J$, $\mathcal{A}_j$ is either a regular clustered cell, or a regular cell array only containing small cell conditions. 
\end{proposition}
\begin{proof}
Let $C_i$ be a large cell condition and assume that $\langle\Sigma \rangle^{(i)}$ is a multi-ball of order $l_i$. If $k_i=l_i$, then by lemma \ref{lemma:largecellsplit}, the clustered cell $C_i^{\langle\Sigma \rangle^{(i)}}$ can be split off. Moreover, since $\mathcal{A}$ is regular, so is $C_i^{\langle\Sigma \rangle^{(i)}}$. 

Now if $l_i>k_i$, by  Lemma \ref{lemma:externalbounded} there exists $M \in \N$ such that $C_i$ is $M$-bounded. %{\color{orange} REMARK: we first need to say that we can assume that all the centers admit external exchange.} 
Partitioning $S$ if necessary (and using Remark \ref{rem:preservation}), we may assume that for all $s \in S$, the interval $(\alpha_i(s), \beta_i(s))$ contains exactly $M'$ elements for some $M'\leqslant M$. % {\color{orange}and for all $s$}. 
Define functions $\delta_1 < \ldots < \delta_{M'}$, such that for each $s \in S$, $(\alpha_i(s), \beta_i(s)) =\{\delta_1(s), \ldots,\delta_{M'}(s)\}$ . %are all the elements in the interval $(\alpha_i(s), \beta_i(s))$. 
Let $\mathcal{A}'$ be the cell array one obtains by applying repartitioning (a) of Lemma-Definition \ref{lem-def:repartitioning} simultaneously to all cell conditions parallel to $C_i$, with respect to the functions $\delta_i$. That is, $\mathcal{A}'$ is obtained from $\mathcal{A}$ by replacing the cell condition $C_i$ (and each cell condition parallel to $C_i$) by $M'$ small cell conditions (and adjusting $\Sigma$ accordingly). 

Note that $\mathcal{A}'$ still satisfies all properties of regularity except possibly (R5), but by Lemma \ref{lem:uniform_tree_structure} and Remark \ref{rem:preservation}, there exists a definable partition of $S$ into sets $S_j$ such that  each array $\mathcal{A}'_{|S_j}$ is regular. Moreover, each such array has at least one large cell condition less than the original cell array $\mathcal{A}$. Iterating the process for the remaining large cell conditions on each $\mathcal{A}'_{|S_j}$ completes the proof. 
\end{proof}

\subsection{Dealing with the remaining small cell arrays}

%\begin{definition}
%We say that two cells $C_1$ and $C_2$ in a cluster $( \{C_i\}_i, \Sigma)$ \emph{exchange points} at level $(\gamma, \gamma')$ for a given $s$, if there exist centers $\sigma \in \Sigma_s^{(1)}$,  $\zeta \in \Sigma_s^{(2)}$, such that $(s, \sigma, \zeta, \ldots) \in \Sigma)$, and the associated leaves
% $C_{1}^{\sigma, \gamma}$ and  $C_{2}^{\zeta, \gamma'}$ have nonempty intersection. \textcolor{red}{Check: do I still use this??} %\textcolor{magenta}{and both centers can occur simultaneously, but how to formulate this?}.
%\end{definition} 
%\noindent In the proofs below, we will use the word \emph{cone} to refer to a subset of $K$ whose elements can be described by a condition
%\[\ord (t-\sigma) = \gamma \wedge t-\sigma \in \lambda Q_{n,m}.\]
%Such a set will be denoted as $\Delta(\sigma; \gamma, \lambda)$. \textcolor{red}{Is this still used?}
%\\\\\\
Let us now have a closer look at the remaining small cell arrays, and how their structure can be simplified. %\sout{For the next lemma, we consider the remaining, unseparated part of our original set (which now consists of small cell conditions only). In order to say anything useful, we will first need to do some normalizations, to ensure that small cell conditions only differ in their height function $\gamma(s)$. {\color{purple} As these normalizations may break some of the properties of regular cell arrays, let us start by making a list of the properties we want to preserve.}}
\\\\ %For the next lemma we consider a regular cell array $\mathcal{A}$ only containing small cell conditions. 
We will do some normalizations first, to ensure that small cell conditions only differ in their height functions $\gamma(s)$. 
These normalizations will not change the actual cells that partition $\mathcal{A}(K)$, in the sense that, if $C$ was a cell condition from $\mathcal{A}$, and $\sigma$ a corresponding potential center, then if the normalization replaces $C$ by $C'$, there will exist a corresponding center $\sigma'$ such that $C^{\sigma} = (C')^{\sigma'}$. In particular, the original cell condition $C$ will be replaced by a condition $C'$ in which $\ac_m(t - \sigma'(s))$ will always be equal to 1. 

Unfortunately, it is not obvious whether the normalization procedure described in Lemma \ref{lemma:lambda} does preserve all properties of regular cell arrays. The definition below (of small regular multi-cells) lists those properties that will still be relevant for subsequent proofs. Other properties may or may not be preserved, but we will pay no further attention to them. 

%As some of these properties are no longer relevant for the remainder of the proof, we will only focus our attention on preserving those properties that will still be necessary. Hence the following definition, which  
%
%\textcolor{purple}{In order to obtain the desired normalization, the sets of potential centers will have to be changed. An implication of this is that we need to be a bit careful as to whether all properties of regular cell arrays are preserved by this process. The properties that we will need later on are grouped in the following definition of \emph{small regular multi-cell}.
%
%I think at this point it is important that convey the following information: our objective is to decompose a regular cell array consisting only of small cell conditions into regular clustered cells. In order to do that we will need to apply some normalizations. By applying these normalizations some conditions of regular cell arrays will be dropped, but as they will be no longer needed for decomposing the starting array into finitely many clustered cells it won't be a problem. However, since the partition will be an iterative procedure, we will list which properties we need to keep after applying the normalizations

\begin{definition} \label{def:smallregmulti} 
A multi-cell $\mathcal{A} = (\{C_{\gamma_j}\}_{1 \leqslant j \leqslant r}, \Sigma)$ is called a \emph{small regular multi-cell} if the following properties hold:
\begin{itemize}
\item[(S1)] All cell conditions $C_{\gamma_j}$ are small cell conditions of the form
\[ \ord(t-\sigma(s)) = \gamma_j(s) \wedge \acm(t-\sigma(s)) \equiv 1 \mod \pi^m,\]
for some $m\in \N$ independent of $j$. Also, for all $s\in S$ it holds that \[\gamma_1(s) < \ldots < \gamma_r(s).\]
\item[(S2)] Each $C_{\gamma_j}^{\Sigma^{(j)}}$ is a clustered cell.
\item[(S3)] For any $1 \leqslant i,j \leqslant r$, and any $\sigma_i \in \Sigma^{(i)}, \sigma_j \in \Sigma^{(j)}$, it holds that $\ord \sigma_i(s) = \ord \sigma_j(s)$ for all $s\in S$. 
%every section $\sigma = (\sigma_1, \ldots, \sigma_r)$ of $\Sigma$ and all $s \in S$, we have that $\r
\item[(S4)] If $C_{\gamma_i}$ and $C_{\gamma_j}$ are copies of the same cell condition, then $\Sigma^{(i)} = \Sigma^{(j)}$.
\item[(S5)]  Each clustered cell $C_{\gamma_j}^{\Sigma^{(j)}}$ has uniform tree structure.  \qedhere
\end{itemize}
\end{definition}
The listed conditions correspond to condition (i) and (ii) in the definition of cell array, and conditions (R1)-(R5) in the definition of regularity, specialized to the case where all cell conditions have the form specified in the above definition. Condition (R6) is no longer relevant since all cell conditions are assumed to be small. Note that by condition (S4) we can use the condensed notation that we introduced at the beginning of the section and write small regular multi-cells in the form $(\{C_{\gamma_j}^{\langle k_j \rangle}\}_{1\leqslant j \leqslant r}, \langle \Sigma \rangle)$.
%
%\textcolor{purple}{Firstly, we will again assume that the same value of $n,m$ is used in all cell conditions, repartitioning as in Lemma-definition \ref{lem-def:repartitioning} if necessary. We also need the following.}
%
%The first normalization concerns the sets $Q_{n,m}$ used in the description of the cell conditions. Up till now, we have allowed the values of $n,m$ to be different in each cell condition. For the arrays we consider in this section, we will assume that all cell conditions use the same value for $n$ and $m$. This can be achieved by taking $n:=\mathrm{lcm}\{n_i\}$ and $m:= \max\{m_i\}$ and repartitioning accordingly. A second normalization is described in the lemma below.
%
%Note that these normalizations may break the admissibility property which we had imposed in earlier sections. This property was a technical condition which was mainly needed for the proof of Lemma \ref{lemma:externalbounded}, and it will play no further role in the remainder of this paper.
\\\\
 In the proof of Lemma \ref{lemma:lambda} below, we will show how to transform regular cell arrays with only small cell conditions into small regular multi-cells.
\begin{lemma} \label{lemma:lambda}
Let $\mathcal{A}$ %= (\{C_{\gamma_j}\}_j, \Sigma)$ 
be a regular cell array, where all cell conditions are small. 
There exists a finite partition of $\mathcal{A}$ %into regular arrays $\mathcal{A}_i:= (\{C_{\gamma_{j}}^{\langle k_j \rangle}\}_j, \langle\Sigma_i\rangle)$, such that for each $i$, there exists %
%a definable transformation of $\Sigma_i$ into a set $\Sigma_i'$, such that
into small regular multi-cells $\mathcal{B}_i$.% for which $\mathcal{A}_i(K) = \mathcal{B}_i(K)$.

% = (\{C_{\gamma_{j}}'^{\langle k_j \rangle}\}_j, \langle\Sigma_i'\rangle)$ is  and $\mathcal{A}_i(K) = \mathcal{B}_i(K)$.
% and some $m \in \N$, such that 
%the cell arrays $(\{C_{\gamma_{ij}}\}_j, \Sigma_i)$ and $(\{C_{\gamma_{ij}}'\}_j, \Sigma_i')$ induce the same set, but every small cell condition $C_{\gamma_{ij}}'$ is of the form
%\[ C_{\gamma_{ij}}' = [\ord (t-\sigma(s)) = \gamma_{ij}(s) \wedge \acm(t-\sigma(s)) =1].\]
\end{lemma}
%
%\begin{lemma} \label{lemma:lambda}
%Let $X = (\{C_{\gamma_j}\}_j, \Sigma)$ be a \textcolor{TealBlue}{regular} cell array, where all cell conditions are small. 
%There exists a finite partition of $X$ into \textcolor{TealBlue}{regular} arrays $X_i:= (\{C_{\gamma_{ij}}\}_j, \Sigma_i)$, satisfying the following.
%\item  For each $i$, there exists a definable transformation of $\Sigma_i$ into a set $\Sigma_i'$ and some $m \in \N$, such that 
%the cell arrays $(\{C_{\gamma_{ij}}\}_j, \Sigma_i)$ and $(\{C_{\gamma_{ij}}'\}_j, \Sigma_i')$ induce the same set, but every small cell condition $C_{\gamma_{ij}}'$ is of the form
%\[ C_{\gamma_{ij}}' = [\ord (t-\sigma(s)) = \gamma_{ij}(s) \wedge \acm(t-\sigma(s)) =1].\]
%\end{lemma}

\begin{proof} Given a small cell condition, we may as well assume that it has the form $C_{\gamma, \lambda}$, where
\[C_{\gamma, \lambda}^{\sigma} := \{(s,t) \in S \times K \mid \ord (t- \sigma(s)) = \gamma(s) \wedge \acm(t-\sigma(s)) = \acm(\lambda)\},\]
 and $\lambda \in K$ %(or $R_K/M_k^m$) 
 with $\ord\, \lambda =0$. Indeed, the condition that $\ord (t-\sigma(s)) \equiv k \mod n$ can in this case be expressed as a condition on $\gamma(s)$, and thus on $S$. Hence, after a finite partitioning of $S$, this last condition is either obvious, or the set is empty. %To obtain the same value of $m$ in all cells, we may need to partition a bit further. %\textcolor{red}{(This may break the admissibility condition, but we do not care about this condition here, as we are going to be changing centers anyway). SHOULD WE CARE ABOUT FIXING IT FURTHER ON?}
%Let us say that $C_{\ga
\\\\
 %\textcolor{red}{For now, we will assume that the residue field has at least 3 elements.}
 Now let $\mathcal{A} =(\{C_{\gamma, \lambda}\}_{\gamma, \lambda}, \Sigma)$ be a regular cell array where each cell condition has the form described above. We will show how to define small regular multicells $\mathcal{B}_k =(\{C_{\gamma_i}^{\langle k_{i} \rangle}\}_{i}, \Sigma_k)$ such that the sets $\mathcal{B}_k(K)$ form a partition of $\mathcal{A}(K) =:X$.
  
  Fix a cell condition $C_{\gamma,\lambda}$ from the description of the array, and write $\Sigma^{(\gamma, \lambda)}$ for its set of potential centers. % where $\lambda = a \pi^r$ (and $\ord\, a = 0$). 
  %We may assume that $\ac_m(\lambda)\neq 1$ as otherwise there is nothing to prove. 
 %Define functions $\delta$ and $\delta_{\lambda}$ in the following way.
  Put $r:= \ord(\lambda-1)$,
  and note that we may suppose that $r < m$, since otherwise we would have that $\ac_m(\lambda)=1$, in which case there is nothing to prove.
Now let $\delta_{\lambda}: \Gamma_K \to \Gamma_K$ be the function defined by
 \[\delta_\lambda(\gamma):= \gamma + r.\]
 Hence, $\delta_{\lambda}$ is simply the constant function  $\gamma \mapsto \gamma$ when $\ac_1(\lambda) \neq 1$. When $\ac_1(\lambda)=1$, we write $\lambda_1$ for the element of $\mathcal{O}_K\backslash \mathcal{M}_K$ satisfying  $\lambda = 1 + \pi^{r} \lambda_1$.  Define a function $\Lambda: K\to K$ by putting
% \begin{align*} 
% \delta_{\lambda}(\gamma ) &:= \left\{ \begin{array}{ll}
% \gamma & \text{if \ } \ac_1(\lambda) \neq 1\\
%  \gamma + \ord(\lambda -1) =\gamma +r & \text{otherwise.}
% \end{array}\right.,	\\ \intertext{and}
\[ \Lambda(\lambda) := \left\{ \begin{array}{ll}
 \lambda-1 & \text{if \ } \ac_1(\lambda) \neq 1\\
  \lambda_1&  \text{otherwise,}
 \end{array}\right. .\] 
 %\end{align*}
 %where $\lambda_1$ is the element of $\mathcal{O}_K\backslash \mathcal{M}_K$ for which there exists $r \in \N \backslash\{0\}$ such that  $\lambda = 1 + \pi^{r} \lambda_1$. Note that we may suppose that $r < m$, since by assumption $\ac_m(\lambda)\neq 1$. 
 Let $T^{(\gamma,\lambda)}$ be the following set:
\[T^{(\gamma,\lambda)} = \{(s,b) \in S\times K \mid \ord\, b = \delta_{\lambda}(\gamma(s)) \wedge \acm(b) = \acm(\Lambda(\lambda)) \}.\]
 %Note that there exists $\kappa \in \Gamma_K$ such that $\gamma = \kappa n+r$. 
 %Now fix an element $b\in K$ such that $\ord\, b = \gamma$ and $\rho_{n,m}(b) =  \lambda-\pi^r$, 
 %and replace $\Sigma^{(\gamma, \lambda)}$ by the translate $T_{\lambda} + \Sigma^{(\gamma, \lambda)}$. 
 We will write $\Sigma^{(\gamma, \lambda)} + T^{(\gamma,\lambda)}$ for the set $\{(s,b_1 + b_2) \mid  (s,b_1) \in \Sigma^{(\gamma, \lambda)} \ \wedge \ (s,b_2) \in T^{(\gamma,\lambda)} \}$, and for any section $\sigma$ of $\Sigma^{(\gamma, \lambda)}$, the set $\sigma + T^{(\gamma,\lambda)}$ is defined similarly. Our claim is now that 
 \begin{claim}$C_{\gamma, \lambda}^{\Sigma^{(\gamma, \lambda)}} = C_{\gamma, 1}^{\Sigma^{(\gamma, \lambda)} + T^{(\gamma,\lambda)}}$.
 \end{claim} For this it is sufficient to show that, for any section $\sigma$ of $\Sigma^{(\gamma, \lambda)}$, it holds that
 \begin{equation} \label{something}
  C_{\gamma,\lambda}^{\sigma} = C_{\gamma,1}^{\sigma + T^{(\gamma,\lambda)}}.
  \end{equation}
 Fix a section $\sigma$, and some $s \in S$. Choose $b \in K$ such that $(s,b) \in T^{(\gamma,\lambda)}$, and put $\zeta(s):= \sigma(s) + b$. We will prove the inclusion $\subset$ in \eqref{something}, by checking that $C_{\gamma, \lambda}^{\sigma(s)} \subset C_{\gamma, 1}^{\zeta(s)}$. Take $t \in C_{\gamma, \lambda}^{\sigma(s)}$. 
Then we have that
\[ \ord (t-\zeta(s)) = \ord(t-(\sigma(s)+b)) = \ord((t-\sigma(s)) -b) = \ord(t-\sigma(s)), \]
since either $\ord(t-\sigma(s)) = \ord\, b$ and  $\ac_1(t-\sigma(s)) \neq \ac_1(b)$, or else $\ord(t-\sigma(s)) < \ord\, b $ (when $\ac_1(\lambda) =1$). We also find that, if $\ac_1(\lambda) \neq 1$, then 
\[\acm(t-\zeta(s)) =
\acm(t-\sigma(s)) -\acm(b) = \acm(\lambda) - \acm(\Lambda(\lambda)) =1,\]
and \[\acm(t-\zeta(s))\equiv \acm(t-\sigma(s)) - \pi^r \acm(b) \equiv \lambda - \pi^r \lambda_1 \equiv 1 \mod \pi^m\] if  $\ac_1(\lambda) =1$.
%\end{array}\right.\]
This proves the inclusion $\subset$. The other inclusion can be proven in a similar way. 
\\\\
 In order to show that this procedure will give us a multi-cell with the desired properties, we need the following further observation. 

\begin{claim}
Every equivalence class-ball in the multi-ball $\Sigma^{(\gamma_i, \lambda_{ij})}$ is  translated to a ball with the same radius and with the same valuation. 
\end{claim}
Indeed, $\Sigma^{(\gamma_{i}, \lambda_{ij})}$ is a multi-ball where all the balls have radius $\gamma_i(s) +m$. The set $T^{(\gamma_i,\lambda_{ij})}$ is a multi-ball of order 1 for which the radius of the balls is at least $\gamma_i(s)+m$. This means that, if $B$ is one of the balls of radius $\gamma_i(s)+m$ from $\Sigma^{(\gamma_{i}, \lambda_{ij})}$, then $B+T^{(\gamma_i,\lambda_{ij})}_s$ will again be a ball of radius $\gamma_i(s)+m$. Hence, we are just translating $\Sigma^{(\gamma_{i}, \lambda_{ij})}_s$ without changing the tree structure. 
Furthermore, the elements of $T^{(\gamma_i,\lambda_{ij})}$ have valuation at least $\gamma_i(s)$, while the elements of $B$ have valuation at most $\gamma_i(s) - 1$ (by condition (2) from Definition \ref{def:clustered_cell}). Therefore, the translation will preserve the valuation of the elements of $\Sigma^{(\gamma_i, \lambda_{ij})}_s$. 
\\\\
The multi-cells $\mathcal{B}_k$ can now be defined as follows. For any fixed height function $\gamma_i$, we replace all cell conditions $C_{\gamma_i, \lambda_{ij}}$ by $C_{\gamma_i}:= C_{\gamma_i,1}$, so the multiplicity $k_{i}$ is given by the number of cell conditions of the form $C_{\gamma_i,\lambda_{ij}}$ occurring in the description of $\mathcal{A}$. 

A set $\widehat{\Sigma}$ can then be defined in the following way. Let $\gamma_1, \ldots, \gamma_l$ be the height functions occurring in the cell conditions $C_{\gamma_i, \lambda_{ij}}$ from $\mathcal{A}$. Put $c := (c_{1,1}, \ldots, c_{1,k_1}, \ldots, c_{l,1}, \ldots, c_{l,k_l})$, and write $\phi(s,c)$ for the formula expressing that the cell fibers $C_{\gamma_i}^{c_{i,j}}$ form a partition of $X_s$.
Then put
\[\widehat{\Sigma}:=\{(s,c) \in S \times K^{k_1 + \ldots+ k_l} \mid c_{i,j} \in \Sigma^{(\gamma_{i}, \lambda_{ij})} + T^{(\gamma_i,\lambda_{ij})} \wedge \phi(s,c)\}.\]
Now, the pair  $(\{C_{\gamma_i}^{\langle k_i\rangle}\}_i, \widehat{\Sigma})$  is a multi-cell defining the set $\mathcal{A}(K) = X$.  We leave it to the reader to check that condtions (S1)-(S3) from Definition \ref{def:smallregmulti}  follow from the above claim. 
%Note that for a small cell $C_i$ in the array $\mathcal{A}$ with height function $\gamma_i(s)$, the set of potential centers $\Sigma^{(i)}$ is a multi-ball where all the balls have radius $\gamma_i(s) +m$. In the proof of Lemma \ref{lemma:lambda} we will define a multi-ball $T^{(i)}$ of order 1 for which the radius of the balls is \textcolor{red}{at least $\gamma_i(s)+m$} . To obtain the desired normalizations, we will transform $\Sigma^{(i)}_s$ into $\Sigma^{(i)}_s + T^{(i)}_s$. If $B$ is one of the balls of radius $\gamma_i(s)+m$ from $\Sigma^{(i)}_s$, then $B+T^{(i)}_s$ will again be a ball of radius $\gamma_i(s)+m$. Hence we are just translating $\Sigma^{(i)}_s$ without changing the tree structure. 
%
%Furthermore, the elements of $T^{(i)}_s$ will be such that their valuation is at least $\gamma_i(s)$, while the elements of $B$ are known to have valuation at most $\gamma_i(s) - 1$ (by property (ii) from Definition \ref{def:clustered_cell}). Therefore the translation will preserve the valuation of the elements of $\Sigma^{(i)}_s$. From these facts it follows easily that many of the properties of the regular cell array $\mathcal{A}$ will be preserved by this translation. 

However, note that projections  $\widehat{\Sigma}^{(i,j_1)}$ and $\widehat{\Sigma}^{(i,j_2)}$ need not be equal in general, even though the corresponding cell condition is $C_{\gamma_i}$ in both cases. Hence, we will need to repeat the procedure described in the proof of Lemma \ref{lem:regular1} to obtain condition (S4). Applying this procedure to $\widehat{\Sigma}$ will yield a set $\Sigma'$, and the reader can check that the multi-cell $(\{C_{\gamma_i}^{\langle k_i \rangle},\Sigma')$ still satisfies conditions (S1)-(S3). A further partitioning of $S$ into sets $S_k$,  like in Lemma \ref{lem:uniform_tree_structure}, will then yield small regular multi-cells $\mathcal{B}_k:= (\{C_{\gamma_i}^{\langle k_i \rangle}\}_i, \Sigma'_{|S_k})$, such that the sets $\mathcal{B}_k(K)$ partition $\mathcal{A}(K)$.
\end{proof}

\begin{lemma}\label{lemma:small-finite}
Let $\mathcal{A}
 = (\{C_{\gamma_i}^{\langle k_i \rangle}\}_{1 \leqslant i \leqslant l}, \langle\Sigma\rangle)$ be a regular array consisting only of small cells $C_{\gamma_i}$. There exists a definable, finite partition of $S$ into sets $S_j$, and, for each $\mathcal{A}_{|S_j}(K)$, a finite partition into regular clustered cells.
%  definable sets $\Sigma_{i,j} \subseteq S_j \times K$ and a finite partitioning of $X$ into \textcolor{TealBlue}{regular} clustered cells $C_{\gamma_i}^{\Sigma_{i,j}}$ of finite order over $S_j$.}
\end{lemma}
\begin{proof}Applying Lemma \ref{lemma:lambda}, we may as well assume that $\mathcal{A}$ is a small regular multi-cell. 
Let $\gamma_1(s)< \ldots< \gamma_l(s)$ be the height functions for the cell conditions in $\mathcal{A}$, and write $\Sigma^{(\gamma_i)}$ for the set of potential centers of the clustered cell associated to $C_{\gamma_i}$. Put $\mathcal{A}(K):=X$. We will first focus on the cells with the smallest leaves, i.e. the cells at height $\gamma_l(s)$. As discussed before, we may assume that $\Sigma^{(\gamma_l)}$ contains centers that admit external exchange.
\\\\
For a center $\sigma$ in $\Sigma^{(\gamma_l)}$ to admit external exchange, there must exist a center $\zeta$ for a lower level $\gamma_j$ (with $j < l$), such that $C_{\gamma_l}^{\sigma(s)} \subset C_{\gamma_j}^{\zeta(s)}$. %Let us assume that $\gamma_j(s)$ is the maximal such height (i.e., closest to $\gamma_l(s)$).
%
%To simplify the discussion a bit, let us assume that $j = l-1$, but the general case is very similar. 
Now consider a decomposition of $X_s$ that contains the potential cell $C_{\gamma_l}^{\sigma(s)}$ as one of its components. This decomposition cannot contain the ball $B:= C_{\gamma_j}^{\zeta(s)}$ as a single leaf at height $\gamma_{j}(s)$, nor as a subset of a leaf at a lower height $\gamma_{j'}$ (for $j'<j$). Indeed, the presence of the ball $C_{\gamma_l}^{\sigma(s)}$ means that such a decomposition could never be a partition. 

%We now claim that the maximality of $\gamma_j(s)$ implies that $\Sigma^{(\gamma_l)}_s$ must contain $l_i := q_K^{\gamma_l(s) - \gamma_{i}(s)}$ potential centers $\{\sigma_1, \ldots, \sigma_{l_i}\}$ (one of them being $\sigma$), such that the union $\bigcup_{r =1}^{l_i}C_{\gamma_l}^{\sigma_r(s)}$ equals a ball of size $\gamma_j(s) +m$.

Hence, in order to represent the points of the ball $B$, we will need a union of smaller balls (small potential cell fibers of heights strictly bigger that $\gamma_j(s)$), where clearly the number of balls one can use is bounded by the sum of the multiplicities of the cell conditions $C_{\gamma_{j+1}}, \ldots, C_{\gamma_l}$. Note that this implies that, if there is exchange possible between two heights $\gamma_i(s)$ and $\gamma_j(s)$, then necessarily the distance $|\gamma_j(s) - \gamma_i(s)|$ is finite (as otherwise one would need infinitely many balls). Moreover, there exists a uniform upper bound for this distance (depending on the respective multiplicities of $C_{\gamma_i}$ and $C_{\gamma_j}$).
%Also, while there may be more than one possible combination of balls $C_{\gamma_i}^{\zeta_i(s)}$ that yields $B$, clearly the number of options is finite as there are only finitely many cell conditions and only finitely many non-equivalent potential centers.
\\\\
Since we are working with a small regular multi-cell, the tree structure for each $\Sigma^{(\gamma_i)}_s$ is independent of $s$ and therefore the number of nonequivalent potential centers at each height is independent of $s$ as well. However, as the tree structure does not fix the distance between the height functions $\gamma_i(s)$, we still need to be a bit careful. % about the exchange structure (i.e. between which cell conditions $C_{\gamma_i}$ there is exchange of points for a given $s$), as this might not be fully independent of $s$.
\\
\\
What the above discussion shows is that, if a center $\sigma(s)$ in $\Sigma^{(\gamma_l)}_s$ admits external exchange, then this implies that $X_s$ must contain a ball $B'$ of radius $\gamma_{l-1}(s) +m$, such that $C_{\gamma_l}^{\sigma(s)} \subset B'$. We will now rewrite the array so that such balls $B'$ can be represented as small cells at height $\gamma_{l-1}(s)$. 

Note that the number of potential centers of $\Sigma^{(\gamma_l)}_s$ that are involved in this, will depend on the distance between $\gamma_l(s)$ and $\gamma_{l-1}(s)$, a number which may vary with $s$. Hence, in order to work uniformly, we will need to partition the set $S$. 
Put $n_k:= q_K^k$ and let $\phi_{k}(s)$ be the definable condition stating that
\[\phi_{k}(s):= n_k < k_l \wedge (\exists\ \sigma_1, \ldots, \sigma_{n_k} \in \Sigma^{(\gamma_l)}_s)[ \cup_{i=1}^{n_k} C_{\gamma_l} ^{\sigma_i(s)} \text{ is a ball of radius } \gamma_{l-1}(s)+m ] \]
%Note that if $\phi_{j,k}(s)$ holds for some tuple $(j,k)$, then $|\gamma_j(s) - \gamma_l(s)| =k$, and that for every $k$ and $s$, there can be at most one value $j$ for which $\phi_{j,k}$ holds. 
Now partition $S$ into sets $S_{k}$ defined as
\[ S_{k}:= \{ s \in S \mid  |\gamma_l(s) - \gamma_{l-1}(s)| = k \text{ and } \phi_{k}(s) \text{ holds} \}.\]
Clearly, this gives a partition of $S$, since by assumption there is exchange between $C_{\gamma_l}$ and lower heights. Also, the partition must be finite %\textcolor{red}{since the number of cell conditions is an upper bound for $k$.}
since we had already remarked that there exists a uniform upper bound for $k$.

Each such set can then be further partitioned as a finite union of sets $S_{k,r}$, where $r$ is the number of disjoint balls of radius $\gamma_{l-1}(s) +m$ that can be formed for a given $s$ using leaves $C_{\gamma_l} ^{\sigma_i(s)}$. This number $r$ is finite since the number of non-equivalent potential centers is finite.
\\\\
 Now fix one such set $S_{k,r}$. The given partition of $S$ naturally induces a partition of $\mathcal{A}$ into small regular multi-cells $\mathcal{A}_{k,r}:=\mathcal{A}_{|S_{k,r}}$, with $X_{k,r}:= \mathcal{A}_{k,r}(K)$ (where all properties are preserved by Remark \ref{rem:preservation}). To unburden notation below, we will simply denote $\mathcal{A}_{k,r}$ as $(\{C_{\gamma_i}^{\langle k_i \rangle}\}_i, \langle\Sigma\rangle)$.

Because of the way $\mathcal{A}_{k,r}$ was defined, we know that there must exist $r$ disjoint sets, each consisting of $n_{k}$ non-equivalent centers $\{\sigma_1, \ldots \sigma_{n_{k}}\}$ in $\langle\Sigma\rangle^{(\gamma_l)}$, such that for each $s$, the union
\begin{equation}\label{union}\bigcup_{i=1}^{n_{k}} C_{\gamma_l}^{\sigma_i(s)}\end{equation} equals a single ball $B'(s)$ of radius $\gamma_{l-1}(s) +m$. Note that it is possible that $\langle\Sigma\rangle^{(\gamma_{l-1})}$ currently does not contain a center $\zeta'(s)$ such that $B'(s) = C_{\gamma_{l-1}}^{\zeta'(s)}$. However, it is possible to definably extend $\langle\Sigma\rangle^{(\gamma_{l-1})}$ to include such a center. Indeed, put
\[\widetilde{\Sigma}_{l-1}:= \{(s,\zeta(s)) \in S \times K \mid \exists \sigma_1(s), \ldots, \sigma_{n_k}(s) \in \Sigma^{(\gamma_k)}_s : C_{\gamma_{l-1}}^{\zeta(s)} = \bigcup_i C_{\gamma_l}^{\sigma_i(s)}\}.\] 
This gives us a set whose fibers consist of  centers $\zeta(s)$ such that $C_{\gamma_{l-1}}^{\zeta(s)}$ is equal to one of the balls $B'(s)$.
We will now replace $\mathcal{A}_{k,r}$ by 
%the array $(\{C_{\gamma_i}^{\langle k_i\rangle}\}, \langle \Sigma \rangle)$ defining $X_{k,r}$ by
 $\mathcal{A}_{k,r}':= (\{C_{\gamma_i}^{\langle k'_i\rangle}\}, \langle \Sigma' \rangle)$, where
\[ k'_i := \left\{ \begin{array}{lcl} k_i & & i < l-1,\\
k_i +r &\text{if}& i = l-1,\\
k_i - rn_k & &  i = l,\\
\end{array}\right.\] replacing cell conditions at height $\gamma_{l}$ by a concurrent number of cell conditions at height $\gamma_{l-1}$. The potential centers can be adjusted accordingly: if we put \[c := (c_{11} \ldots, c_{1k_1'},\dots, c_{l1}, \ldots, c_{lk_l'}),\] then $\Sigma'$ can be defined as $\Sigma' := \{(s,c)\in S_{k,r}\times K^{\sum k_i'} \mid \psi_{k,r}(s,c)\}$, where $\psi_{k,r}$ is the formula
\[\psi_{k,r}(s,c):= c_{ij} \in \langle \Sigma\rangle ^{(i)}_s \text{ for } i \neq l-1 \ \wedge \ c_{l-1,j} \in \langle \Sigma\rangle^{(l-1)}_s \cup (\widetilde{\Sigma}_{l-1})_s \ \wedge \ \bigcup_{i,j} C_{\gamma_i}^{c_{ij}} = (X_{k,r})_s.\]
It should be clear that $\mathcal{A}'_{k,r}$ still satisfies conditions (S1)-(S4), and that $\mathcal{A}_{k,r}(K) = \mathcal{A}'_{k,r}(K)$. It may be that (S5) no longer holds, but this can be remedied by a further partitioning of $S$ if necessary.
%$(\{C_{\gamma_i}^{\langle k_i\rangle}\}, \langle \Sigma \rangle)$ and $(\{C_{\gamma_i}^{\langle k'_i\rangle}\}, \langle \Sigma' \rangle)$ define the same set. 
Moreover, we claim that after this transformation, there is no further exchange possible between cells $C_{\gamma_l}$ and cells at lower heights. The reason is simply that the condition for exchange is no longer satisfied, as the original leaves $C_{\gamma_l}^{\sigma}$ that were part of a bigger ball are now represented inside a bigger leaf at height $\gamma_{l-1}$. Hence, since there is no more exchange, the remaining cell conditions $C_{\gamma_l}$ can now be split off definably. 

Repeating the same procedure $l-2$ more times for the remaining small regular multi-cells will result in a union of regular clustered cells. %Note that, as we have been creating new centers, it may be that these clustered cells are no longer regular, but this can be remedied by a further partitioning if necessary.
\end{proof}

\section{A decomposition into regular clustered cells}\label{subsec:reduced}

%\textcolor{red}{make sure to add some definition of multi-balls!}

%\subsection{The Cell Decomposition Theorem} 

We are now ready to state a full, detailed version of our cell decomposition theorem. We tried to make the statement reasonably self-contained.

\begin{theorem}[Clustered cell decomposition]\label{thm:celldecomposition}
Let $X \subseteq S \times K$ be a set definable in a $P$-minimal structure. Then there exist $n,m \in \N\backslash\{0\}$ and a finite partition of $X$ into definable sets $X_i \subseteq S_i \times K$ of the one of the following forms
\begin{itemize}
\item[(i)] Classical cells 
\[X_i= \{ (s,t) \in S_i \times K \mid \alpha_{i}(s) \ \square_1 \ \ord(t-c_{i}(x)) \ \square_2 \ \beta_{i}(s) \wedge t - c_{i}(s) \in \lambda_{i} Q_{n,m} \},\]
where $\alpha_{i}, \beta_{i}$ are definable functions $S_i \to \Gamma_K$, the squares $\square_1,\square_2$ may denote either $<$ or    $\emptyset$ (i.e. `no condition'), and $\lambda_{i} \in K$. The center $c_{i}: S_i \to K$ is a definable function (which may not be unique). 
\item[(ii)] Regular clustered cells $X_i=C_{i}^{\Sigma_{i}}$ of order $k_{i}$. \item[] Let $\sigma_1, \ldots, \sigma_{k_{i}}$ be (non-definable) sections of the definable multi-ball $\Sigma_{i} \subseteq S_i \times K$, such that for each $s \in S_i$, the set $\{\sigma_1(s), \ldots, \sigma_{k_{i}}(s)\}$ contains representatives of all $k_{i}$ disjoint balls covering $(\Sigma_{i})_s$. Then $X_{i}$ partitions as
% The definable set $\Sigma_{i} \subseteq S_i \times K$ is a  multi-ball of order $k_{i}$ and $X_i$ partitions as
\[X_{i} = C_{i}^{\sigma_1} \cup \ldots \cup C_{i}^{\sigma_{k_{i}}},\]
%where $\sigma_1, \ldots, \sigma_{k_{i}}$ are sections of $\Sigma_{i}$ such that for each $s \in S_i$, the set $\{\sigma_1(s), \ldots, \sigma_{k_{i}}(s)\}$ contains representatives of all $k_{i}$ disjoint balls covering $(\Sigma_{i})_s$. 
where each set $C_{i}^{\sigma_l}$ is of the form
\[ C_{i}^{\sigma_l} = \{(s,t) \in S_i \times K \mid \alpha_{i}(s) \ < \ \ord(t-\sigma_l(s)) \ < \ \beta_{i}(s) \wedge t-\sigma_l(s) \in \lambda_{i} Q_{n,m}\}.\]
Here $\alpha_{i}, \beta_{i}$ are definable functions $S_i \to \Gamma_K$, $\lambda_{i} \in K \backslash \{0\}$, and $\ord\, \alpha_{i}(s) \geqslant \ord \,\sigma_l(s)$ for all $s \in S_i$. Finally, we may suppose no section of $\Sigma_i$ is definable.  
\end{itemize}
\end{theorem}

\begin{proof}
By Theorem \ref{thm:refinement}, there is a partition of  $X$ into classical cells and cell arrays $(\{C_j \}_j, \Sigma)$. 
If different values of $m_i, n_i$ occur for different cell conditions in the partition, put $m:= \max_i\{m_i\}$ and $n:= \text{lcm}_i\{n_i\}$. The classical cells in the decomposition can be partitioned in a straightforward way to obtain cells described using the set $Q_{n,m}$.

By Proposition \ref{prop:finitebranching}, we know that there exists a uniform upper bound $N$ for the number of $(C_j, \Sigma^{(j)}_s)$-equivalence classes. This allows us to obtain Proposition \ref{prop:regular_cell_array}, where we show that any cell array can be partitioned as a finite union of regular cell arrays. Moreover, recall that the first step in this proof uniformizes the value of $n$ and $m$ within an array, and we can use the procedure described there to make sure that the same $n,m$ are used uniformly for all cell arrays in the partition of $X$. Later steps in the proof will never need to modify the values of $n$ and $m$ again.

 %and Lemmas \ref{cor:sym} \textcolor{TealBlue}{and \ref{lemma:small-large}}, we may assume each array is regular. 
 In Proposition \ref{prop:large-small} and Lemmas \ref{lemma:small-finite}, we show how to split a regular cell array into a finite union of regular clustered cells  of finite order. %, which proves Part (ii).
 If for one of the clustered cells in our partition, the corresponding set $\Sigma_i$ would admit a definable section, then the splitting procedure from Definition \ref{def:defsplit} can be used to partition off one or more classical cells, until no more definable sections remain. So we can indeed suppose that no definable sections exist.
\end{proof}

\subsection{Final remarks} \label{subsec:two-sorted}

While we have presented our cell decomposition theorem in a two-sorted context, allowing the variables in $S$ to be both $K$-variables and $\Gamma_K$-variables, it should be clear that Theorem \ref{thm:celldecomposition} can also be applied to one-sorted $P$-minimal structures. For instance, Mourgues' result (specifically the implication $(i)\to (ii)$) can easily be derived from it.

\begin{theorem}[Mourgues] Let $(K,\cL)$ be a (one-sorted) $P$-minimal field. Then the following are equivalent:
\begin{enumerate}
\item[(i)] $(K,\cL)$ has definable Skolem functions;
\item[(ii)] every definable set can be decomposed into a finite number of classical cells. 
\end{enumerate}
\end{theorem}

Note that the above theorem is only relevant to the one-sorted case. The reason is that two-sorted $P$-minimal structures will never admit definable Skolem functions. Indeed, there cannot exist a definable section of the valuation map $\ord: K \to \Gamma_K$, since its image would be an infinite discrete set. The existence of such a set would imply that the structure is actually not $P$-minimal. 
\\\\
Since we are working with two sorts, two types of cell decompositions need to be considered, depending on the sort of the last variable. 
Our focus in the current paper is on definability for the field sort, and more specifically on definable sets $X \subset S \times K$ where the last variable is a $K$-variable. In fact, it would probably be more precise to call our main result a $K$-cell decomposition theorem, where a $K$-cell may either be a classical cell or a regular clustered cell.
Cell decomposition is significantly less complicated for sets $X \subseteq S \times \Gamma_K$, and the following $\Gamma$-cell decomposition was already obtained in \cite{cubi-leen-2015}:
\begin{theorem}[$\Gamma$-cell decomposition] \label{prop:prep}
Let $X \subseteq S\times \Gamma_K$ be definable in a $P$-minimal structure $(K,\Gamma_K)$. There exists a finite partition of $X$ in $\Gamma$-cells $B$ of the form
%A subset $B\subseteq S\times \Gamma_K$ is a \emph{$\Gamma$-cell} if it is of the form
\[B= \left\{(s,\gamma)\in D\times \Gamma_K \left|\begin{array}{l} \alpha(s)\ \square_1 \ \gamma \ \square_2 \ \beta(s) \ \ \wedge \ \ 
\gamma \equiv k\mod n \end{array}\right\}\right.,\]
where $D$ is a definable subset of $S$, $\alpha_i, \beta_i$ are definable functions $D\to\Gamma_k$, $k, n\in \NN$ and the squares $\square_i$ may denote $<$ or $\emptyset$. %\emph{no condition}. 
\end{theorem}

The version given in \cite{cubi-leen-2015} is actually slightly stronger than what is presented here. %(but we decided to omit the more technical details)
Additionally one has that,   given a definable function $f:X \subseteq S \times \Gamma_K \to \Gamma_K$, there exists a finite partition of $X$ into $\Gamma$-cells such that on each part,  %can be partitioned into finitely many $\Gamma$-cells, \textcolor{purple}{in such a way as to ensure that on each part, 
the function $f$ is linear in the last $\Gamma$-variable (see  \cite{cubi-leen-2015}).  

Readers familiar with other cell decomposition theorems %might expect that cells $X\subseteq S\times T$ (where $T$ is either $\Gamma_K$ or $K$), 
may have noticed that in both Theorem \ref{thm:celldecomposition} and \ref{prop:prep}, no further conditions are imposed on the parameter set $S$ (besides definability). In many similar-style theorems, cells are defined inductively, in the sense that the set $S$ is required to be a cell as well, and similarly for its consecutive projections. We have not insisted on this, mainly because it would have required us to include more details on $\Gamma$-cell decomposition, which is not something which we wanted to focus on in this paper. We are however convinced that such an inductive cell decomposition theorem can be derived quite easily from Theorem \ref{thm:celldecomposition}, when taking into account both $K$-cell and $\Gamma$-cell decomposition. 
\\\\
Both $\Gamma$-cell and $K$-cell decompositions are important, and sometimes they need to be used simultaneously (see for instance Proposition 4.5 and Corollary 4.6 in \cite{cubi-leen-2015}). We intend to write a sequel to this paper, where some further applications of these theorems (related to $p$-adic integration) will be discussed. \\\\
To finish this article, we pose the following open question. %\textcolor{red}{rewrite this bit once we have made a final decision on the contents of the note?}

%\begin{question}\label{question1} Does there exist a $P$-minimal structure with a definably \textcolor{red}{inseparable} multi-ball $\Sigma\subseteq S\times K$ of order strictly bigger than 1?
%\end{question}
\begin{question}
Can every regular clustered cell of finite order be decomposed into finitely many regular clustered cells of order 1?
\end{question}

A positive answer to this question would considerably simplify the cell decomposition theorem presented in this paper. %Indeed, the implication would be that every inseparably clustered cell would be of order 1, and hence not really \emph{clustered} as such. 
%Indeed, it would be sufficient to consider cells  with as center either a definable function, or a definable family of balls.
Unfortunately, there are some indications that the answer should be no. We intend to discuss this issue in more detail in a note which we will publish separately.

\section{Acknowledgements}
The first author was supported by the European Research Council under the European Community's Seventh Framework Programme (FP7/2007-2013) with ERC Grant Agreement nr. 615722 (MOTMELSUM).
The second author was supported by ERC grant agreements nr. 615722
(MOTMELSUM) and nr. 637027 (TOSSIBERG). 
During the realization of this project, the third author was a postdoctoral fellow of the Fund for Scientific Research - Flanders (Belgium) (F.W.O.).

\bibliographystyle{plain}
\bibliography{bibliography}

\end{document}